\documentclass[twoside]{article}
\usepackage{latexsym}
\usepackage{pstricks}
\usepackage{amssymb}
\usepackage{amsmath}
\usepackage{amsfonts}
\usepackage{amsthm,array}
\usepackage{array}
\usepackage{epsfig}
\usepackage{graphicx}
\usepackage[numbers,sort&compress]{natbib}

\newtheorem{theorem}{Theorem}[section]

\newtheorem{lemma}{Lemma}[section]
\newtheorem{corollary}{Corollary}[section]

\newtheorem{definition}{Definition}[section]

\newtheorem{remark}{Remark}[section]

\newtheorem{remark-definition}{Remark and Definition}[section]
\newtheorem{rem-not}{Remark and Notation}[section]

\begin{document}
\title{\bf Positive Definiteness of $4$th Order $3$-Dimensional Symmetric Tensors with entries $-1$, $0$, $1$\thanks{The first author's work was supported by the National Natural Science Foundation of P.R. China (Grant No.12171064), by The team project of innovation leading talent in chongqing (No.CQYC20210309536) and by the Foundation of Chongqing Normal university (20XLB009)} }
\date{}
\author{ Li Ye, Yisheng Song\thanks{Corresponding author E-mail: yisheng.song@cqnu.edu.cn}\\
School of Mathematical Sciences,  Chongqing Normal University, \\
Chongqing, 401331, P.R. China. \\ Email: neutrino1998@126.com (Ye); yisheng.song@cqnu.edu.cn (Song)} 
\maketitle

{\noindent\bf Abstract.} It is well-known that a symmetric matrix with its entries $\pm1$ is not positive definite.  But this is not ture for symmetric tensors (hyper-matrix). In this paper, we mainly  dicuss  the positive (semi-)definiteness criterion of a class of $4$th order $3$-dimensional symmetric tensors  with entries $t_{ijkl}\in\{-1,0,1\}$. Through theoretical derivations and detailed classification discussions, the criterion for determining the positive (semi-)definiteness of such a class of tensors   are provided based on the relationships and number values of its entries. Which establishes some unique properties of higher  symmetric tensors that distinct from ones of matrces

\vspace{.3cm}
\vspace{.3cm}


\section{Introduction}
\setcounter{section}{1}

The positive definiteness and positive (semi-)definiteness of tensors were initially introduced by Qi \cite{qi1}. When the order $m=2$, the concept of a positive (semi-)definite tensor coincides with that of a positive (semi-)definite matrix. The Sylvester's Criterion, as a well-known method, can efficiently determines the positive (semi-)definiteness of a matrix. The positive (semi-)definiteness of an $m$th order $n$-dimensional symmetric tensor $\mathcal{T}=(t_{i_1i_2\cdots i_m})$  is essentially equivalent to  solve the positive (non-negative) conditions of an $m$th degree homogeneous polynomial of $n$ variables, $f_{\mathcal{T}}(x)$, denoted by
$$f_{\mathcal{T}}(x)=\sum_{i_1\cdots i_m=1}^nt_{i_1\cdots i_m}x_{i_1}\cdots x_{i_m},$$
where $x=(x_1,x_2,x_3,\cdots,x_n)^\top\in\mathbb{R}^n$ \cite{qi1,LQW2016,DLB2018,LWZ2014,WLX2016}. This  positive (semi-)definiteness  problem has   been widely used  in many  scientific and engineering fields \cite{bd,nk,nkp,nkr,mf,ma,wh,qlf,glfy,lk,SQ2024}. However,  it is well-known that this problem  is  an NP-hard problem in general ($n>2$) even though the order $m=4$  \cite{HL2013,MK1987}. 

For a quartic binary homogeneous polynomial ($n=2$),  the positive (non-negative) conditions has been perfectly found.   By determining the number of real roots of a equation, the positive (non-negative) conditions of a quartic binary homogeneous polynomial were established by Rees \cite{re}, Lazard \cite{la}, Gadem-Li \cite{gl}, Ku \cite{ku}, and Jury-Mansour \cite{jm}. Wang-Qi \cite{wq} improved the proof of the above conclusions and perfectly gave the  positive (non-negative) conditions of such a polynimial untill 2005.  Recently,  Qi-Song-Zhang \cite{qsz} took the distinct approach to provide  new necessary and sufficient conditions  in different forms. This actually gives the positive (semi-)definite criterion of  $4$th order $2$-dimensional symmetric tensors, and moreover, it different from ones of matrices. In fact, a symmetric matrix with its entries $\pm1$ is not positive definite since by Sylvester's Criterion, its principal matrice 
$$\begin{pmatrix}
	1 & 1 \\
	1 & 1 
\end{pmatrix}\ \mbox{ and } \ \begin{pmatrix}
	1 & -1 \\
	-1 & 1 
\end{pmatrix}$$
must not be positive definite. However, a $4$th order $2$-dimensional symmetric tensor $\mathcal{T}=(t_{i_1i_2\cdots i_m})$  with $|t_{ijkl}|=1$ is positive definite if $t_{1122}=1$ and $t_{1112}t_{1222}= -1$ (see Lemma \ref{lem:24}), i.e., $$\mathcal{T}=\begin{pmatrix}
	\begin{pmatrix}
		1\ &\ 1 \\
		1\ &\ 1 
	\end{pmatrix}& \begin{pmatrix}
		1 & 1 \\
		1 & -1
	\end{pmatrix}\\[0.4cm]
	\begin{pmatrix}
		1 & 1 \\
		1 & -1 
	\end{pmatrix} & \begin{pmatrix}
		1 & -1 \\
		-1 & 1 
\end{pmatrix}\end{pmatrix}\mbox{ or }\begin{pmatrix}
	\begin{pmatrix}
		1 & -1 \\
		-1 & 1 
	\end{pmatrix}& \begin{pmatrix}
		-1 & 1 \\
		1 &1 
	\end{pmatrix}\\[0.4cm]
	\begin{pmatrix}
		-1& 1 \\
		1 & 1
	\end{pmatrix}& \begin{pmatrix}
		1\ &\ 1 \\
		1\ &\ 1
\end{pmatrix}\end{pmatrix}.$$
For a $4$th order $3$-dimensional symmetric tensor with its entries $\pm1$,  Song \cite{S2025} showed its positive definiteness conditions after added cyclic symmetric condition; Song-Liu \cite{sl}  provided  its ( strict) copositivity conditions. Recently, Hu-Yan \cite{HY2024} presented a DCA (difference of convex algorithm) method for quartic minimization over the sphere.
For a  matrix $M$ with its entries $\pm1,0$ and every diagonal entry being $1$,  Hoffman-Pereira \cite{aj} showed $M$ is positive semi-definite if and only if it has no $3\times 3$ principal submatrices which, after principal rearrangement, are of the form
$$
\begin{pmatrix}
	1 & -1  & -1\\
	-1 & 1 & -1\\
	-1 & -1 & 1
\end{pmatrix},\ \ \begin{pmatrix}
	1 & -1  & -1\\
	-1 & 1 & 0\\
	-1 & 0 & 1
\end{pmatrix},\ \ \begin{pmatrix}
	1 &\ 1  &\ 1\\
	1 &\ 1 &\ 0\\
	1 &\ 0 &\ 1
\end{pmatrix},\ \ \begin{pmatrix}
	1 & 1  & -1\\
	1 & 1 & 0\\
	-1 & 0 & 1
\end{pmatrix}, \ \ \begin{pmatrix}
	1 & -1  & 1\\
	-1 & 1 & 1\\
	1 & 1 & 1
\end{pmatrix}.
$$
Obviously, the above 5 $3\times 3$ matrices are not positive semi-definite. Spontaneously,  we want to know that what sort of conditions  is met a $4$th order $3$-dimensional symmetric tensor with entries $0, \pm1$ to be positive semi-definite. 

Motived by the above results,  we mainly focus on the positive (semi-)definite criterion of $4$th order $3$-dimensional symmetric tensors whose elements are restricted in the set $\{-1,0,1\}$. By separating the values of $t_{1112}t_{1222}$, $t_{2223}t_{2333}$ and $t_{1333}t_{1113}$ into groups,   a series of systematic analyses are conducted on positive (semi-)definiteness of such a class of symmetric tensors. The sum-of-squares forms of most of the cases are presented, so that their positive (semi-)definiteness can be directly obtained. Additionally, the methods in \cite{nka}, which based on a theorem of \cite{eil,ei}, is utilized to reach our main conclusions. And with the help of these results, the criterion for determining the positive semi-definiteness of tensors are obtained  evenwhen there is diagonal element equals to $0$.

\section{Preliminaries}
\setcounter{section}{2}

Denote ${\mathcal{S}}_{m,n}$ by a set of $m$th order $n$-dimensional symmetric tensors, and ${\mathcal{E}}_{m,n}$ by a set of $m$th order $n$-dimensional symmetric tensors in which every entry is $-1$, $0$, or $1$, $\widehat{{\mathcal{E}}}_{m,n}\subset{\mathcal{E}}_{m,n}$ by a set in which every diagonal entry is $1$.
\begin{definition}\label{def1}\textup{\cite{qi1}}
Let ${\mathcal{T}}=(t_{i_1i_2\cdots i_m})\in{\mathcal{S}}_{m,n}$. $\mathcal{T}$ is called
\begin{itemize}
\item[(a)] {\bf positive semi-definite} if $m$ is an even number and in the Euclidean space ${\mathbb{R}}^n$, its associated homogeneous polynomial
    $${\mathcal{T}} x^m=\sum_{i_1,i_2,\cdots,i_m=1}^nt_{i_1i_2\cdots i_m}x_{i_1}x_{i_2}\cdots x_{i_m}\geq0;$$
\item[(b)] {\bf positive definite} if $m$ is an even number and ${\mathcal{T}}x^m>0$ for all $x\in{\mathbb{R}}^n\setminus\{0\}$.
\end{itemize}
\end{definition}
The following theorems and lemmas will be required for the subsequent work.
\begin{lemma}\label{lem:21}\textup{\cite{S2025}}
Let ${\mathcal{T}}=(t_{ijkl})\in{\mathcal{S}}_{4,n}$. Then $\mathcal{T}$ is positive definite if and only if
$$\left
\{\begin{array}{ll}
{\mathcal{T}}=0\Rightarrow x=0,\\
\mbox{there is a }y\in{\mathbb{R}}^n\setminus\{0\}\mbox{ such that }{\mathcal{T}}y^4>0.
\end{array}
\right.$$
\end{lemma}

Let ${\mathcal{T}}=(t_{ijkl})\in{\mathcal{S}}_{4,2}$. Then for $x=(x_1,x_2)^{\top}$,
\begin{equation}\label{4-2homo}
{\mathcal{T}} x^4=t_{1111}x_1^4+4t_{1112}x_1^3x_2+6t_{1122}x_1^2x_2^2+4t_{1222}x_1x_2^3+t_{2222}x_2^4.
\end{equation}

\begin{lemma}\label{lem:22}\textup{\cite{qsz,S2025}}
Let $\mathcal{T}=(t_{ijkl})$  be a 4th-order 2-dimensional symmetric tensor with its entires $|t_{ijkl}|\leq1$ and $t_{1111}=t_{2222}=1.$  Then \begin{itemize}
	\item[(i)]   $\mathcal{T}$ is positive semi-definite if and only if$$\begin{cases}
		-\dfrac13\leq t_{1122}\leq1,	(t_{1112}-t_{1222})^2\leq6t_{1122}+2,\\
		27(t_{1122}+2t_{1112}t_{1122}t_{1222}-t_{1122}^3-t_{1222}^2-t_{1112}^2)^2\le(1-4t_{1112}t_{1222}+3t_{1122}^2)^3.
	\end{cases}$$
	\item[(ii)]   $\mathcal{T}$ is positive definite if and only if $$\begin{cases}
		\dfrac13\leq t_{1122}<1$, $2t_{1112}^2+1=3t_{1122}$, $t_{1112}=t_{1222}; \\
		-\dfrac13< t_{1122}\leq1,	(t_{1112}-t_{1222})^2\leq6t_{1122}+2,\\
		27(t_{1122}+2t_{1112}t_{1122}t_{1222}-t_{1122}^3-t_{1222}^2-t_{1112}^2)^2< (1-4t_{1112}t_{1222}+3t_{1122}^2)^3.
	\end{cases}$$
\end{itemize}
\end{lemma}

For ${\mathcal{T}}\in\widehat{{\mathcal{E}}}_{4,2}$,  the following conclusions are easily established by Lemma \ref{lem:22}.

\begin{lemma}\label{lem:23}
	Let ${\mathcal{T}}=(t_{ijkl})\in\widehat{{\mathcal{E}}}_{4,2}$. Then
	\begin{itemize}
		\item[(i)] $\mathcal{T}$ is positive semi-definite if and only if $t_{1122}=t_{1112}=t_{1222}=0$, or $t_{1122}=1$;
		\item[(ii)] $\mathcal{T}$ is positive definite if and only if if and only if $t_{1122}=t_{1112}=t_{1222}=0$, or  $t_{1122}=1$ and $t_{1112}t_{1222}\in\{0,-1\}$.
	\end{itemize}
\end{lemma}
\begin{proof}
	It follows from Lemma \ref{lem:22} that $t_{1122}=1$ or $t_{1122}=0$.
	
	(i) If $t_{1122}=1$, then by Lemma \ref{lem:22} (i),  $\mathcal{T}$ is positive semi-definite if and only if
	$$27(t_{1222}-t_{1112})^4\leq 64(1-t_{1112}t_{1222})^3.$$
	By this time, the above inequality holds if and only if \begin{center}
		either $t_{1112}t_{1222}=1$, or $t_{1112}t_{1222}=-1$, or $t_{1112}t_{1222}=0$.
	\end{center}
	
	Similarly, when $t_{1122}=0$, $\mathcal{T}$ is positive semi-definite if and only if
	$$(t_{1112}-t_{1222})^2\leq 2\mbox{ and }27(t_{1222}^2-t_{1112}^2)^2\leq (1-4t_{1112}t_{1222})^3.$$
	It is easy to verify that the above inequalities hold if and only if  $t_{1112}=t_{1222}=0$.  This shows (i).
	
	Similarly,  (ii)  can be proved also.
\end{proof}

From Lemma \ref{lem:23}, we  obviously obtain  the following lemma (also see Ref. \cite{S2025}).

\begin{lemma}\label{lem:24}
Let $\mathcal{T}$$=(t_{ijkl})\in{\mathcal{S}}_{4,2}$ with its entries
$|t_{ijkl}|=1$ and $t_{1111}=t_{2222}=1$. Then
\begin{itemize}
\item[(${\romannumeral 1}$)] $\mathcal{T}$ is positive semi-definite if and only if $t_{1122}=1$;
\item[(${\romannumeral 2}$)] $\mathcal{T}$ is positive definite if and only if $t_{1122}=1$ and $t_{1112}t_{1222}= -1$.
\end{itemize}
\end{lemma}

\begin{lemma}\label{lem:25} \textup{\cite{eil,ei}}
The number of distinct real roots $N$ of a single variable polynomial,
$$F(x)=a_mx^m+a_{m-1}x^{m-1}+\cdots+a_1x+a_0=0,\ \ a_m>0$$
with $a_i$'s defined over a real number field, is given by
$$N=\mbox{var}[+,-|\triangle_1^1|,|\triangle_3^1|,\cdots,(-1)^m|\triangle_{2m-1}^1|]-\mbox{var}[+,|\triangle_1^1|,|\triangle_3^1|,\cdots,|\triangle_{2m-1}^1|]$$
where "var" denotes the number of sign variations and $|\triangle_i^1|$, $i=1,3,\cdots,2m-1$ are the inner determinants in the matrix $[\triangle^1]$ shown in following and $|\triangle_{2m-1}^1|\neq0$. Critical cases when other $|\triangle_i^1|$ may be zero are handled routinely in \textup{\cite{arn}}.
\end{lemma}

$$[\triangle^1]=\left[
\begin{array}{cccccccccc}
a_m & a_{m-1} & ~  & \cdots & a_0 & 0 & \cdots & 0 \\
0 & a_m & ~  & \cdots & a_1 & a_0 & \cdots & 0 \\
\vdots& ~ & ~  & ~  & ~  & ~  & ~ &\vdots \\
0 & \cdots & a_m & a_{m-1} & a_{m-2} & ~ &\cdots& a_0\\
0 & \cdots & 0 & ma_m & (m-1)a_{m-1} & ~ &\cdots& a_1\\
0 & \cdots & ma_m & (m-1)a_{m-1} & (m-2)a_{m-2} & ~ &\cdots& 0\\
\vdots& ~ & ~  & ~  & ~  & ~  & ~ &\vdots \\
0 & ma_m & ~  & \cdots & ~ & ~ & ~ & 0 \\
ma_m & (m-1)a_{m-1} & ~ & \cdots & a_1 & 0 & \cdots & 0
\end{array}
\right],$$
then
$$\triangle_{1}^1=ma_m, \ \  \triangle_3^1=\left|
\begin{array}{cccccccccc}
 a_m & a_{m-1} & a_{m-2} \\
0 & ma_m & (m-1)a_{m-1} \\
ma_m & (m-1)a_{m-1} & (m-2)a_{m-2}\\
\end{array}
\right|,\cdots,\triangle_{2m-1}^1=|\triangle^1|.$$

\begin{lemma}\label{lem:26}\cite{S2025}
	Assume a 4th order  symmetric tensor $\mathcal{T}$ is positive semi-definite. Then $$t_{iiii}=0  \  \Longrightarrow\  t_{iiij}=0\mbox{ and }t_{iijj}\ge0, \mbox{ for all }i,j, i\ne j.$$
\end{lemma}

\vspace{.3cm}

\section{Positive definiteness of $4$th order $3$-dimensional symmetric tensors}
\setcounter{section}{3}
Let ${\mathcal{T}}=(t_{ijkl})\in\widehat{{\mathcal{E}}}_{4,3}$. Since each principle sub-tensor of positive (semi-)definite tensor must be positive (semi-)definite \cite{qi1}, then  it follow from Lemma \ref{lem:23} that  the following necessary conditions are obvious for positive (semi-)definiteness of  ${\mathcal{T}}$:
${\mathcal{T}}$ is positive semi-definite only if  \begin{equation}\label{eq:2} t_{iijj}=t_{iiij}=t_{ijjj}=0  \mbox{ or }t_{iijj}=1, \forall \ i,j\in\{1,2,3\}, i\neq j;\end{equation}
${\mathcal{T}}$  is positive definite only if for all $ i,j\in\{1,2,3\},i\neq j$, \begin{equation}\label{eq:3}\mbox{either }t_{iijj}=t_{iiij}=t_{ijjj}=0\mbox{ or }t_{iijj}=1\mbox{ and }t_{iiij}t_{ijjj}\in\{0,-1\}.\end{equation}
Therefore, we categorize the conditions on the basis of the above Eqs. \eqref{eq:2} or \eqref{eq:3}  bubilding the positive (semi-)definiteness of  ${\mathcal{T}}$.

\begin{theorem}\label{th3.1}
Let ${\mathcal{T}}=(t_{ijkl})\in\widehat{{\mathcal{E}}}_{4,3}$ and $t_{iiij}=0$, for all $i,j\in\{1,2,3\}$, $i\neq j$. Then $\mathcal{T}$ is positive definite if and only if one of the following conditions is satisfied.
\begin{itemize}
  \item [(a)] $t_{1123}=t_{1223}=t_{1233}=t_{1122}=t_{1133}=t_{2233}=1$;
  \item [(b)] $t_{1123}=t_{1223}=t_{1233}=0$ and $t_{iijj}\in\{0,1\}$ for all $i,j\in\{1,2,3\}$, $i\neq j$;
  \item [(c)] Two of $\{t_{1123},t_{1223},t_{1233}\}$ are $-1$ and the third one is $1$, $t_{1122}=t_{1133}=t_{2233}=1$;
  \item [(d)] $t_{iijk}=\pm1$, $t_{ijjk}=t_{ijkk}=0$,  $t_{iijj}=t_{iikk}=1$ and $t_{jjkk}\in\{0,1\}$ for $i,j,k\in\{1,2,3\}$, $i\neq j$, $i\neq k$, $j\neq k$.
\end{itemize}
\end{theorem}
\begin{proof}
 Let $t_{iiij}=0$ for all $i,j\in\{1,2,3\}$, $i\neq j$. Then for $x=(x_1,x_2,x_3)^{\top}\in{\mathbb{R}}^3$,
\begin{align*}
{\mathcal{T}}x^4=&x_1^4+x_2^4+x_3^4+12(t_{1123}x_1^2x_2x_3+t_{1223}x_1x_2^2x_3+t_{1233}x_1x_2x_3^2)\\&+6(t_{1122}x_1^2x_2^2+t_{1133}x_1^2x_3^2+t_{2233}x_2^2x_3^2).
\end{align*}

``\textbf{if (Sufficiency)}."
(a) $t_{1123}=t_{1223}=t_{1233}=t_{1122}=t_{1133}=t_{2233}=1$.  Then for all $x=(x_1,x_2,x_3)^{\top}\in{\mathbb{R}}^3$,
\begin{equation}\label{th3.1a}
{\mathcal{T}}x^4=x_1^4+x_2^4+x_3^4+6(x_1x_2+x_1x_3+x_2x_3)^2\geq0.
\end{equation}

(b) $t_{1123}=t_{1223}=t_{1233}=0$ and $t_{1122}, t_{1133}, t_{2233}\in\{0,1\}$.  Then for all $x=(x_1,x_2,x_3)^{\top}\in{\mathbb{R}}^3$,
\begin{equation}\label{th3.1b}
{\mathcal{T}}x^4=x_1^4+x_2^4+x_3^4+6(t_{1122}x_1^2x_2^2+t_{1133}x_1^2x_3^2+t_{2233}x_2^2x_3^2)\geq0.
\end{equation}

(c)  $t_{iijk}=-t_{ijjk}=-t_{ijkk}=t_{iijj}=t_{jjkk}=t_{iikk}=1$ for $i,j,k\in\{1,2,3\}$, $i\neq j$, $i\neq k$, $j\neq k$. Then for all $x=(x_1,x_2,x_3)^{\top}\in{\mathbb{R}}^3$
\begin{equation}\label{th3.1c}
{\mathcal{T}}x^4= x_1^4+x_2^4+x_3^4+6(x_ix_j+x_ix_k-x_jx_k)^2\geq0.
\end{equation}

(d) $t_{iijk}=\pm1$, $t_{ijjk}=t_{ijkk}=0$, $t_{iijj}=t_{iikk}=1$, $t_{jjkk}\in\{0,1\}$ for $i,j,k\in\{1,2,3\}$, $i\neq j$, $i\neq k$, $j\neq k$. Then for all $x=(x_1,x_2,x_3)^{\top}\in{\mathbb{R}}^3$,
\begin{align}\label{th3.1d}
{\mathcal{T}}x^4&\geq x_1^4+x_2^4+x_3^4\pm12x_i^2x_jx_k+6(x_i^2x_j^2+x_i^2x_k^2)\nonumber\\
&=x_1^4+x_2^4+x_3^4+6(x_ix_j\pm x_ix_k)^2\geq0.
\end{align}

Furthermore, in (\ref{th3.1a})-(\ref{th3.1d}), it is easily verified that ${\mathcal{T}}x^4=0$ if and only if $x=(0,0,0)^{\top}$, and then, $\mathcal{T}$ is positive definite.

``\textbf{only if (Necessity)}."  It follows from the positive definiteness of  $\mathcal{T}$ with Eq. \eqref{eq:3} that \begin{center}
	$t_{iijj}=1$ or $0$ for all $i,j\in\{1,2,3\}$ and $i\neq j$.
\end{center} And in the meantime, for $x=(1,1,1)^\top$, we have\begin{align*}
	{\mathcal{T}}x^4=&3+12(t_{1123}+t_{1223}+t_{1233})+6(t_{1122}+t_{1133}+t_{2233})>0,
\end{align*}
and hence,  $$2(t_{1123}+t_{1223}+t_{1233})+(t_{1122}+t_{1133}+t_{2233})>-\dfrac12.$$
So, the following cases  could not occur, \begin{itemize}
	\item Two of $\{t_{1123},t_{1223},t_{1233}\}$ are $-1$ and the third one is $0$;
	\item $t_{1123}=t_{1223}=t_{1233}=-1 ;$
	\item $t_{1123}+t_{1223}+t_{1233}=-1 $ and $t_{1122}+t_{1133}+t_{2233}=0$ or $1$.
\end{itemize}
Next we use a proof by contradiction to prove the necessity.

\textbf{Case 1.} $t_{1123}=t_{1223}=t_{1233}=1$ and there is at least $0$ in $\{t_{1122}, t_{1133},t_{2233}\}$. Without loss the generality, we might take $t_{2233}=0$. Then  for $x=(-1,2,2)^{\top}$, we have
$${\mathcal{T}}x^4\leq x_1^4+x_2^4+x_3^4+12(x_1^2x_2x_3+x_1x_2^2x_3+x_1x_2x_3^2)+6(x_1^2x_2^2+x_1^2x_3^2)=-63<0.$$

\textbf{Case 2.} Two of $\{t_{1123},t_{1223},t_{1233}\}$ are $1$, the third one is $0$ or $-1$. Without loss the generality, we might take $t_{1123}=t_{1223}=1$. Let $x=(1,1,-1)^{\top}$, then
$${\mathcal{T}}x^4\leq x_1^4+x_2^4+x_3^4+12(x_1^2x_2x_3+x_1x_2^2x_3)+6(x_1^2x_2^2+x_1^2x_3^2+x_2^2x_3^2)=-3<0.$$

\textbf{Case 3.} One of $\{t_{1123},t_{1223},t_{1233}\}$ is $1$.

\textbf{Subcase 3.1} Two of $\{t_{1123},t_{1223},t_{1233}\}$ are $0$.  When $t_{iijk}=1$ and at least one of $\{t_{iijj},t_{iikk}\}$ is $0$ for $i,j,k\in\{1,2,3\}$, $i\neq j$, $i\neq k$, $j\neq k$. Without loss the generality, suppose $t_{1123}=1$. If at least one of $\{t_{1122},t_{1133}\}$ is $0$, take $x=(3,2,-2)^{\top}$, then
$${\mathcal{T}}x^4\leq x_1^4+x_2^4+x_3^4+12x_1^2x_2x_3+6(t_{1122}x_1^2x_2^2+t_{1133}x_1^2x_3^2+x_2^2x_3^2)=-7<0.$$

\textbf{Subcase 3.2}  Two of $\{t_{1123},t_{1223},t_{1233}\}$ are $-1$, and there is $t_{iijj}=0$ for $i,j\in\{1,2,3\}$, $i\neq j$. Without loss the generality, we might take $t_{1123}=1$, and $t_{1223}=t_{1233}=-1$. When $t_{1122}=0$, take $x=(3,2,-2)^{\top}$, then
$${\mathcal{T}}x^4\leq x_1^4+x_2^4+x_3^4+12x_1^2x_2x_3-12(x_1x_2^2x_3+x_1x_2x_3^2)+6(x_1^2x_3^2+x_2^2x_3^2)=-7<0,$$
$t_{1133}=0$ is similarly. When $t_{2233}=0$, take $x=(1,2,2)^{\top}$
$${\mathcal{T}}x^4\leq x_1^4+x_2^4+x_3^4+12x_1^2x_2x_3-12(x_1x_2^2x_3+x_1x_2x_3^2)+6(x_1^2x_2^2+x_1^2x_3^2)=-63<0.$$

\textbf{Subcase 3.3} One of $\{t_{1123},t_{1223},t_{1233}\}$ is $-1$, the other two ones are $0$ and $1$, respectively. Without loss the generality, we might take $t_{1123}=1$, $t_{1223}=-1$, and $t_{1233}=0$.  Let $x=(3,-2,2)^{\top}$, then
$${\mathcal{T}}x^4\leq x_1^4+x_2^4+x_3^4+12x_1^2x_2x_3-12x_1x_2^2x_3+6(x_1^2x_2^2+x_1^2x_3^2+x_2^2x_3^2)=-79<0.$$

\textbf{Case 4.} There is not $1$ in $\{t_{1123},t_{1223},t_{1233}\}$. Two of $\{t_{1123},t_{1223},t_{1233}\}$ are $0$, the third one is  $-1$.  When $t_{iijk}=-1$ and at least one of $\{t_{iijj},t_{iikk}\}$ is $0$ for $i,j,k\in\{1,2,3\}$, $i\neq j$, $i\neq k$, $j\neq k$. Without loss the generality, suppose $t_{1123}=-1$. If at least one of $\{t_{1122},t_{1133}\}$ is $0$, take $x=(3,2,2)^{\top}$, then
$${\mathcal{T}}x^4\leq x_1^4+x_2^4+x_3^4-12x_1^2x_2x_3+6(t_{1122}x_1^2x_2^2+t_{1133}x_1^2x_3^2+x_2^2x_3^2)=-7<0.$$
The necessity is proved.
\end{proof}

By comparing the polynomials corresponding to $t_{iiii}=1$ and $t_{iiij}=0$ with that $t_{iiii}=t_{iiij}=0$, for all $i,j\in\{1,2,3\}$, $i\neq j$, in ${\mathcal{E}}_{4,3}$. the following results can be obtained.
\begin{corollary}\label{cor1}
Let ${\mathcal{T}}=(t_{ijkl})\in{\mathcal{E}}_{4,3}$ and $t_{iiii}=0$, for all $i\in\{1,2,3\}$. Then $\mathcal{T}$ is positive semi-definite if and only if $t_{iiij}=0$ for all $i,j\in\{1,2,3\}$, $i\neq j$, and one of the following conditions is satisfied.
\begin{itemize}
  \item [(a)] $t_{1123}=t_{1223}=t_{1233}=t_{1122}=t_{1133}=t_{2233}=1$;
  \item [(b)] $t_{1123}=t_{1223}=t_{1233}=0$ and $t_{iijj}\in\{0,1\}$ for all $i,j\in\{1,2,3\}$, $i\neq j$;
  \item [(c)] Two of $\{t_{1123},t_{1223},t_{1233}\}$ are $-1$ and the third one is $1$, $t_{1122}=t_{1133}=t_{2233}=1$;
  \item [(d)] $t_{iijk}=\pm1$, $t_{ijjk}=t_{ijkk}=0$,  $t_{iijj}=t_{iikk}=1$ and $t_{jjkk}\in\{0,1\}$ for $i,j,k\in\{1,2,3\}$, $i\neq j$, $i\neq k$, $j\neq k$.
\end{itemize}
\end{corollary}
\begin{proof} It follows from Lemma \ref{lem:26} that
	$$t_{iiii}=t_{iiij}=0\mbox{ and }t_{iijj}=0 \mbox{ or }1, \mbox{ for all }i,j, i\ne j.$$
	And  then for $x=(x_1,x_2,x_3)^{\top}\in{\mathbb{R}}^3$,
$$
		{\mathcal{T}}x^4=12(t_{1123}x_1^2x_2x_3+t_{1223}x_1x_2^2x_3+t_{1233}x_1x_2x_3^2)+6(t_{1122}x_1^2x_2^2+t_{1133}x_1^2x_3^2+t_{2233}x_2^2x_3^2).
	$$
Using the similar proof techniques of Theorem \ref{th3.1},  the required conclusions follow.
\end{proof}

\begin{theorem}\label{th3.3}
Let ${\mathcal{T}}=(t_{ijkl})\in\widehat{{\mathcal{E}}}_{4,3}$ and $t_{iiij}t_{ijjj}=1$, for all $i,j\in\{1,2,3\}$, $i\neq j$.
Then $\mathcal{T}$ is positive semi-definite if and only if $t_{iiij}t_{jjjk}t_{ikkk}=t_{iijk}t_{iiij}t_{iiik}=t_{iijj}=1$ for all $i,j\in\{1,2,3\}$, $i\neq j$, $i\neq k$, $j\neq k$.
\end{theorem}
\begin{proof} Let $t_{iiii}=t_{iiij}t_{ijjj}=1$ for all $i,j\in\{1,2,3\}$, $i\neq j$.

``\textbf{if (Sufficiency)}."
Suppose $t_{iijk}t_{iiij}t_{iiik}=t_{iijj}=1$ for $i,j,k\in\{1,2,3\}$, $i\neq j$, $i\neq k$, $j\neq k$.  Then for any $x=(x_1,x_2,x_3)^{\top}\in{\mathbb{R}}^3$,
$${\mathcal{T}}x^4=(x_1+x_2+x_3)^4\geq0$$
when $t_{iiij}=1$ for all $i,j\in\{1,2,3\}$, $i\neq j$, and
$${\mathcal{T}}x^4=(x_i+x_j-x_k)^4\geq0$$
when $t_{iiij}=-t_{jjjk}=-t_{ikkk}=1$ for $i,j,k\in\{1,2,3\}$, $i\neq j$, $i\neq k$, $j\neq k$.

``\textbf{only if (Necessity)}." For  $i,j,k\in\{1,2,3\}$, $i\neq j$, $i\neq k$, $j\neq k$,
the condition, $t_{iiij}t_{ijjj}=t_{iiij}t_{jjjk}t_{ikkk}=1$, can be can be divided into two cases, 
\begin{itemize}
    \item [($\romannumeral1$)] $t_{1112}t_{2223}t_{1333}=1$;
	\item [($\romannumeral2$)] $t_{1112}t_{2223}t_{1333}=-1$.
\end{itemize}
It follows from the positive semi-definiteness of  $\mathcal{T}$ with Eq. \eqref{eq:2} that $$t_{1122}=t_{1133}=t_{2233}=1.$$

($\romannumeral1$) Since for $x=(x_1,x_2,x_3)^{\top}\in{\mathbb{R}}^3$,
\begin{align*}
{\mathcal{T}}x^4&=(x_1+x_2+x_3)^4+12[(t_{1123}-1)x_1^2x_2x_3+(t_{1223}-1)x_1x_2^2x_3+(t_{1233}-1)x_1x_2x_3^2]
\end{align*}
when $t_{1112}=t_{2223}=t_{1333}=1$, and
\begin{align*}
{\mathcal{T}}x^4=&[x_i+x_j+(-x_k)]^4+12[(-t_{iijk}-1)x_i^2x_j(-x_k)+(-t_{ijjk}-1)x_ix_j^2(-x_k)\\
&+(t_{ijkk}-1)x_ix_j(-x_k)^2]
\end{align*}
when $t_{iiij}=-t_{jjjk}=-t_{ikkk}=1$ for $i,j,k\in\{1,2,3\}$, $i\neq j$, $i\neq k$, $j\neq k$, we only need to discuss about $t_{1112}=t_{2223}=t_{1333}=1$ and there is $t_{iijk}\neq 1$, for $i,j,k\in\{1,2,3\}$, $i\neq j$, $i\neq k$, $j\neq k$.

Without loss the generality suppose $t_{1233}\neq1$ and $t_{1233}=\min\{t_{1123},t_{1223},t_{1233}\}$. Let $x=(1,1,-3)^{\top}$, then
$${\mathcal{T}}x^4\leq(x_1+x_2+x_3)^4-12(x_1^2x_2x_3+x_1x_2^2x_3+x_1x_2x_3^2)=-35<0.$$

($\romannumeral2$) Since for $x=(x_1,x_2,x_3)^{\top}\in{\mathbb{R}}^3$,
\begin{align*}
{\mathcal{T}}x^4=&[x_1+x_2+(-x_3)]^4-8(x_1^3x_2+x_1x_2^3)+12[(-t_{1123}-1)x_1^2x_2(-x_3)+(-t_{1223}-1)x_1x_2^2(-x_3)\\
&+(t_{1233}-1)x_1x_2(-x_3)^2]
\end{align*}
when $t_{1112}=t_{2223}=t_{1333}=-1$, and
\begin{align*}
{\mathcal{T}}x^4=&(x_i+x_j+x_k)^4-8(x_i^3x_j+x_ix_j^3)+12[(t_{iijk}-1)x_i^2x_jx_k+(t_{ijjk}-1)x_ix_j^2x_k\\
&+(t_{ijkk}-1)x_ix_jx_k^2]
\end{align*}
when $-t_{iiij}=t_{jjjk}=t_{ikkk}=1$, for $i,j,k\in\{1,2,3\}$, $i\neq j$, $i\neq k$, $j\neq k$, we only need to discuss about $t_{1112}=t_{2223}=t_{1333}=-1$.
Then, for $x=(1,1,2)^\top$, we have\begin{align*}
	{\mathcal{T}}x^4=-16+24(t_{1123}+t_{1223}+2t_{1233})\geq0.
\end{align*}
So, the following cases could not occur,
 \begin{itemize}
    \item $t_{1233}=-1$;
	\item $t_{1233}=0$, and both $t_{1123}$ and $t_{1223}$ are not $1$;
	\item $t_{1233}=1$ and $t_{1123}=t_{1223}=-1$.
\end{itemize}
Discuss other situations later.

\textbf{Case 1.} $t_{1233}=0$, and $t_{1123}=1$ or $t_{1223}=1$.

\textbf{Subcase 1.1} $t_{1123}=t_{1223}=1$, let $x=(2,2,-1)^{\top}$, then
$${\mathcal{T}}x^4=(x_1+x_2-x_3)^4-8(x_1^3x_2+x_1x_2^3)+24(x_1^2x_2x_3+x_1x_2^2x_3)-12x_1x_2x_3^2=-63<0.$$

\textbf{Subcase 1.2} One of $\{t_{1123},t_{1223}\}$ is $0$. Without loss the generality, we might take $t_{1123}=0$. Let $x=(-1,4,3)^{\top}$, then
$${\mathcal{T}}x^4=(x_1+x_2-x_3)^4-8(x_1^3x_2+x_1x_2^3)+12(x_1^2x_2x_3-x_1x_2x_3^2)+24x_1x_2x_3^2=-32<0.$$

\textbf{Subcase 1.3} One of $\{t_{1123},t_{1223}\}$ is $-1$. Without loss the generality, we might take $t_{1123}=-1$. Let $x=(-1,2,2)^{\top}$, then
$${\mathcal{T}}x^4=(x_1+x_2-x_3)^4-8(x_1^3x_2+x_1x_2^3)+24x_1x_2^2x_3-12x_1x_2x_3^2=-15<0.$$

\textbf{Case 2.} $t_{1233}=1$, and $t_{1123}\neq-1$ or $t_{1223}\neq-1$.

\textbf{Subcase 2.1} $t_{1123}=t_{1223}=1$. Let $x=(2,-1,2)^{\top}$, then
$${\mathcal{T}}x^4=(x_1+x_2-x_3)^4-8(x_1^3x_2+x_1x_2^3)+24(x_1^2x_2x_3+x_1x_2^2x_3)=-15<0.$$

\textbf{Subcase 2.2} One of $\{t_{1123},t_{1223}\}$ is $1$, the other one is $0$ or $-1$. Without loss the generality, we might take $t_{1123}=1$, let $x=(2,-1,2)^{\top}$, then
$${\mathcal{T}}x^4\leq(x_1+x_2-x_3)^4-8(x_1^3x_2+x_1x_2^3)+12(2x_1^2x_2x_3+x_1x_2^2x_3)=-63<0.$$

\textbf{Subcase 2.3} One of $\{t_{1123},t_{1223}\}$ is $0$, the other one is $-1$. Without loss the generality, we might take $t_{1123}=-1$. Let $x=(-1,2,2)^{\top}$, then
$${\mathcal{T}}x^4\leq(x_1+x_2-x_3)^4-8(x_1^3x_2+x_1x_2^3)+12x_1x_2^2x_3=-15<0,$$

The necessity is proved.
\end{proof}

\begin{theorem}\label{th3.5}
Let ${\mathcal{T}}=(t_{ijkl})\in\widehat{{\mathcal{E}}}_{4,3}$ and $-t_{iiij}t_{ijjj}=-t_{jjjk}t_{jkkk}=t_{iiik}t_{ikkk}=1$ for $i,j,k\in\{1,2,3\}$, $i\neq j$, $i\neq k$, $j\neq k$. Then $\mathcal{T}$ is positive semi-definite if and only if $t_{iijk}t_{jkkk}=t_{ijkk}t_{iiij}=t_{iiik}t_{jkkk}t_{iiij}=t_{1122}=t_{2233}=t_{1133}=1$, and $t_{ijjk}t_{iiik}\in\{0,1\}$.
\end{theorem}
\begin{proof}
``\textbf{if (Sufficiency)}." Suppose $t_{iijk}t_{jkkk}=t_{ijkk}t_{iiij}=t_{iiik}t_{jkkk}t_{iiij}=t_{1122}=t_{2233}=t_{1133}=1$. Then for $x=(x_1,x_2,x_3)^{\top}\in{\mathbb{R}}^3$,
$${\mathcal{T}}x^4=(x_i^2-x_j^2+x_k^2+2t_{iiik}x_ix_k+2t_{jkkk}x_jx_k+2t_{iiij}x_ix_j)^2+4(x_ix_j+t_{ijjk}x_jx_k)^2\geq0$$
when $t_{ijjk}t_{iiik}=1$, and
\begin{align*}
{\mathcal{T}}x^4=&(x_i^2-x_j^2+x_k^2+2t_{iiik}x_ix_k+2t_{jkkk}x_jx_k+2t_{iiij}x_ix_j)^2+2(x_ix_j-t_{iiik}x_jx_k)^2\\
&+2(x_i^2x_j^2+x_j^2x_k^2)\geq0
\end{align*}
when $t_{ijjk}t_{iiik}=0$.

``\textbf{only if (Necessity)}." Obviously, the condition, 
$-t_{iiij}t_{ijjj}=-t_{jjjk}t_{jkkk}=t_{iiik}t_{ikkk}=1$ for $i,j,k\in\{1,2,3\}$, $i\neq j$, $i\neq k$, $j\neq k$, can be can be divided into two cases, 
\begin{itemize}
    \item [(i)] $t_{1112}t_{2223}t_{1333}=1$;
	\item [(ii)] $t_{1112}t_{2223}t_{1333}=-1$.
\end{itemize}
Similar to the proof of Theorem \ref{th3.3}, we only need to consider $-t_{iiij}=t_{jjjk}=t_{ikkk}=1$ and $t_{1112}=t_{2223}=t_{1333}=1$, respectively.
And it follows from the positive semi-definiteness of  $\mathcal{T}$ with Eq. \eqref{eq:2} that $$t_{1122}=t_{1133}=t_{2233}=1.$$
Without loss the generality, suppose $-t_{1112}t_{1222}=-t_{2223}t_{2333}=t_{1333}t_{1113}=1$.

($\romannumeral1$) We might take $-t_{1112}=t_{2223}=t_{1333}=1$. Then for $x=(x_1,x_2,x_3)^{\top}\in{\mathbb{R}}^3$
\begin{align*}
{\mathcal{T}}x^4=&(x_1+x_2+x_3)^4-8(x_1^3x_2+x_2x_3^3)+12[(t_{1123}-1)x_1^2x_2x_3+(t_{1223}-1)x_1x_2^2x_3\\
&+(t_{1233}-1)x_1x_2x_3^2].
\end{align*}
Then, for $x=(1,-4,1)^\top$, we have\begin{align*}
	{\mathcal{T}}x^4=-16-48(t_{1123}-4t_{1223}+t_{1233})\geq0,
\end{align*}
and hence,  $$3(t_{1123}-4t_{1223}+t_{1233})\leq-1.$$
So, the following cases could not occur,
 \begin{itemize}
    \item $t_{1223}=-1$;
	\item $t_{1223}=0$ and $t_{1123}+t_{1233}\ge0$.
\end{itemize}
Next we use a proof by contradiction to prove the necessity.

\textbf{Case 1.} $t_{1123}=-1$, $t_{1233}=0$ and $t_{1223}\in\{0,1\}$. Let $x=(-5,1,3)^{\top}$, then
$${\mathcal{T}}x^4\leq (x_1+x_2+x_3)^4-8(x_1^3x_2+x_2x_3^3)-24x_1^2x_2x_3-12(x_1x_2^2x_3+x_1x_2x_3^2)=-295<0.$$
$t_{1123}=0$, $t_{1233}=-1$ and $t_{1223}\in\{0,1\}$ is similarly.

\textbf{Case 2.} $t_{1223}=1$ and $\{t_{1123},t_{1233}\}$ are not $-1$.

\textbf{Subcase 2.1} At least one of is $0$. Without loss the generality, suppose $t_{1123}=0$. Let $x=(-1,1,2)^{\top}$, then
$${\mathcal{T}}x^4\leq (x_1+x_2+x_3)^4-8(x_1^3x_2+x_2x_3^3)-12(x_1^2x_2x_3+x_1x_2x_3^2)=-16<0.$$

\textbf{Subcase 2.2} $t_{1123}=t_{1233}=1$. Let $x=(-1,1,2)^{\top}$, then
$${\mathcal{T}}x^4\leq (x_1+x_2+x_3)^4-8(x_1^3x_2+x_2x_3^3)=-40<0.$$

($\romannumeral2$) Let $t_{1112}=t_{2223}=t_{1333}=1$. Then for $x=(x_1,x_2,x_3)^{\top}\in{\mathbb{R}}^3$,
\begin{align*}
{\mathcal{T}}x^4=&(x_1+x_2+x_3)^4-8(x_1x_2^3+x_2x_3^3)\\
&+12[(t_{1123}-1)x_1^2x_2x_3+(t_{1223}-1)x_1x_2^2x_3+(t_{1233}-1)x_1x_2x_3^2].
\end{align*}
Then, for $x=(-3,1,3)^\top$, we have\begin{align*}
	{\mathcal{T}}x^4=-83+108(3t_{1123}-t_{1223}-3t_{1233})\geq0.
\end{align*}
So, the following cases could not occur,
 \begin{itemize}
    \item $t_{1123}=-1$, and $t_{1223}\neq-1$ or $t_{1233}\neq-1$;
	\item $t_{1123}=0$ and $t_{1233}=1$;
     \item $t_{1123}=0$, $t_{1223}\neq-1$ and $t_{1233}=0$;
	\item $t_{1123}=t_{1233}=1$ and $t_{1223}\neq-1$.
\end{itemize}
Discuss other situations later.

\textbf{Case 1.} $t_{1123}=t_{1223}=t_{1233}=-1$. Let $x=(1,1,1)^{\top}$. Then
$${\mathcal{T}}x^4= (x_1+x_2+x_3)^4-8(x_1x_2^3+x_2x_3^3)-24(x_1^2x_2x_3+x_1x_2^2x_3+x_1x_2x_3^2)=-7<0.$$

\textbf{Case 2.} $t_{1123}=0$.

\textbf{Subcase 2.1} $t_{1233}=0$ and $t_{1223}=-1$. Let $x=(-6,1,3)^{\top}$. Then
$${\mathcal{T}}x^4= (x_1+x_2+x_3)^4-8(x_1x_2^3+x_2x_3^3)-24x_1x_2^2x_3-12(x_1^2x_2x_3+x_1x_2x_3^2)=-368<0.$$

\textbf{Subcase 2.2} $t_{1233}=-1$ and $t_{1223}\neq-1$. Let $x=(3,1,-3)^{\top}$. Then
$${\mathcal{T}}x^4\leq (x_1+x_2+x_3)^4-8(x_1x_2^3+x_2x_3^3)-24x_1x_2x_3^2-12(x_1^2x_2x_3+x_1x_2^2x_3)=-23<0.$$

\textbf{Subcase 2.3} $t_{1233}=t_{1223}=-1$. Let $x=(3,5,5)^{\top}$. Then
$${\mathcal{T}}x^4= (x_1+x_2+x_3)^4-8(x_1x_2^3+x_2x_3^3)-12x_1^2x_2x_3-24(x_1x_2^2x_3+x_1x_2x_3^2)=-139<0.$$

\textbf{Case 3.} $t_{1123}=1$.

\textbf{Subcase 3.1} $t_{1233}=-1$. Let $x=(3,1,-3)^{\top}$, Then
$${\mathcal{T}}x^4\leq (x_1+x_2+x_3)^4-8(x_1x_2^3+x_2x_3^3)-24(x_1x_2^2x_3+x_1x_2x_3^2)=-239<0.$$

\textbf{Subcase 3.2} $t_{1233}=1$ and $t_{1223}=-1$. Let $x=(1,-1,1)^{\top}$. Then
$${\mathcal{T}}x^4= (x_1+x_2+x_3)^4-8(x_1x_2^3+x_2x_3^3)-24x_1x_2^2x_3=-7<0.$$

\textbf{Subcase 3.3} $t_{1233}=0$ and $t_{1223}\neq-1$. Let $x=(3,1,-3)^{\top}$. Then
$${\mathcal{T}}x^4\leq (x_1+x_2+x_3)^4-8(x_1x_2^3+x_2x_3^3)-12(x_1x_2^2x_3+x_1x_2x_3^2)=-23<0.$$

\textbf{Subcase 3.4} $t_{1233}=0$ and $t_{1223}=-1$. Let $x=(5,-5,3)^{\top}$. Then
$${\mathcal{T}}x^4= (x_1+x_2+x_3)^4-8(x_1x_2^3+x_2x_3^3)-24x_1x_2^2x_3-12x_1x_2x_3^2=-139<0.$$

The necessity is proved.
\end{proof}

\begin{theorem}\label{th3.8}
Let ${\mathcal{T}}=(t_{ijkl})\in\widehat{{\mathcal{E}}}_{4,3}$ and $t_{iiij}t_{ijjj}=0$, but $t_{iiij}+t_{ijjj}\neq0$ for all $i,j\in\{1,2,3\}$, $i\neq j$. Then $\mathcal{T}$ is positive definite if and only if $t_{iijj}=1$ for all $i,j\in\{1,2,3\}$, $i\neq j$, and one of the following conditions is satisfied.
\begin{itemize}
  \item [(a)] $t_{1123}=t_{1223}=t_{1233}=0$,  $t_{iiij}=t_{jjjk}=t_{ikkk}=0$ and $t_{ijjj},t_{jkkk},t_{iiik}\in\{1,-1\}$ for $i,j,k\in\{1,2,3\}$, $i\neq j$, $i\neq k$, $j\neq k$.
  \item [(b)] $t_{iijk}=t_{ijjk}=0$ and $t_{ijkk}t_{ijjj}=t_{ijjj}t_{jkkk}t_{ikkk}=\pm1$, or $t_{ijjj}t_{jkkk}t_{ikkk}=-t_{iijk}t_{jkkk}=-t_{ijjk}t_{ikkk}=1$ and $t_{ijkk}=0$, if $t_{iiij}=t_{jjjk}=t_{iiik}=0$ for $i,j,k\in\{1,2,3\}$, $i\neq j$, $i\neq k$, $j\neq k$.
\end{itemize}
\end{theorem}
\begin{proof}
``\textbf{if (Sufficiency)}." Suppose $t_{1122}=t_{1133}=t_{2233}=1.$

(a) For $x=(x_1,x_2,x_3)^{\top}\in{\mathbb{R}}^3$, we have
\begin{equation}\label{th3.8a}
{\mathcal{T}}x^4=(x_i^2+2t_{iiik}x_ix_k)^2+(x_j^2+2t_{ijjj}x_ix_j)^2+(x_k^2+2t_{jkkk}x_jx_k)^2+2(x_1^2x_2^2+x_2^2x_3^2+x_1^2x_3^2)\geq0,
\end{equation}
where $i,j,k\in\{1,2,3\}$, $i\neq j$, $i\neq k$.

(b) If $t_{iiij}=t_{jjjk}=t_{iiik}=0$ and $t_{ijjj}t_{jkkk}t_{ikkk}=\pm1$ for $i,j,k\in\{1,2,3\}$, $i\neq j$, $i\neq k$, $j\neq k$, then for $x=(x_1,x_2,x_3)^{\top}\in{\mathbb{R}}^3$,
\begin{align*}
{\mathcal{T}}x^4=&x_i^4+x_j^4+x_k^4+4(t_{ijjj}x_ix_j^3+t_{jkkk}x_jx_k^3+t_{ikkk}x_ix_k^3)\\
&+12(t_{iijk}x_i^2x_jx_k+t_{ijjk}x_ix_j^2x_k+t_{ijkk}x_ix_jx_k^2)+6(x_i^2x_j^2+x_j^2x_k^2+x_i^2x_k^2).
\end{align*}
If $t_{iijk}=t_{ijjk}=0$ and $t_{ijkk}t_{ijjj}=t_{ijjj}t_{jkkk}t_{ikkk}=\pm1$, then for any $x=(x_1,x_2,x_3)^{\top}\in{\mathbb{R}}^3$,
\begin{equation}\label{th3.8b}\aligned
{\mathcal{T}}x^4=&x_i^4+(x_j^2+2t_{ijjj}x_ix_j)^2+(x_k^2+2t_{jkkk}x_jx_k+2t_{ikkk}x_ix_k)^2\\&+2(x_jx_k+t_{ijkk}x_ix_k)^2+2x_i^2x_j^2\geq0,\endaligned
\end{equation}
where $i,j,k\in\{1,2,3\}$, $i\neq j$, $i\neq k$.

If $t_{ijjj}t_{jkkk}t_{ikkk}=-t_{iijk}t_{jkkk}=-t_{ijjk}t_{ikkk}=1$ and $t_{ijkk}=0$ for $i,j,k\in\{1,2,3\}$, $i\neq j$, $i\neq k$, $j\neq k$, then for $x=(x_1,x_2,x_3)^{\top}\in{\mathbb{R}}^3$,
$${\mathcal{T}}x^4=x_i^4+x_j^4+x_k^4+4(x_ix_j^3+x_jx_k^3+x_ix_k^3)-12(x_i^2x_jx_k+x_ix_j^2x_k)+6(x_i^2x_j^2+x_j^2x_k^2+x_i^2x_k^2)$$
when $t_{ijjj}=t_{jkkk}=t_{ikkk}=1$, and
\begin{align*}
{\mathcal{T}}x^4=&x_i^4+x_j^4+x_k^4+4[x_ix_j^3+x_j(-x_k)^3+x_i(-x_k)^3]-12[x_i^2x_j(-x_k)+x_ix_j^2(-x_k)]\\
&+6(x_i^2x_j^2+x_j^2x_k^2+x_i^2x_k^2)
\end{align*}
when $t_{ijjj}=-t_{jkkk}=-t_{ikkk}=1$, and
\begin{align*}
{\mathcal{T}}x^4=&x_i^4+x_j^4+x_k^4+4[(-x_i)x_j^3+x_jx_k^3+(-x_i)x_k^3]-12[(-x_i)^2x_jx_k+(-x_i)x_j^2x_k]\\
&+6(x_i^2x_j^2+x_j^2x_k^2+x_i^2x_k^2)
\end{align*}
when $-t_{ijjj}=t_{jkkk}=-t_{ikkk}=1$, and
\begin{align*}
{\mathcal{T}}x^4=&x_i^4+x_j^4+x_k^4+4[x_i(-x_j)^3+(-x_j)x_k^3+x_ix_k^3]-12[x_i^2(-x_j)x_k+x_i(-x_j)^2x_k]\\
&+6(x_i^2x_j^2+x_j^2x_k^2+x_i^2x_k^2)
\end{align*}
when $-t_{ijjj}=-t_{jkkk}=t_{ikkk}=1$. Thus we only need to discuss about $t_{ijjj}=t_{jkkk}=t_{ikkk}=1$. Without loss the generality, we might take $t_{1112}=t_{2223}=t_{1113}=0$ and $t_{1222}=t_{2333}=t_{1333}=1$. Then when $t_{1233}=0$, and $t_{1123}=t_{1223}=-1$, for $x=(x_1,x_2,x_3)^{\top}\in{\mathbb{R}}^3$
\begin{equation}\label{x3}
{\mathcal{T}}x^4=x_1^4+x_2^4+x_3^4+4(x_1x_2^3+x_2x_3^3+x_1x_3^3)-12(x_1^2x_2x_3+x_1x_2^2x_3)+6(x_1^2x_2^2+x_2^2x_3^2+x_1^2x_3^2).
\end{equation}

By Lemma \ref{lem:23}, when $x_1=0$, or $x_2=0$, or $x_3=0$, ${\mathcal{T}}x^4\geq0$, with equality if and only if $x_1=x_2=x_3=0$.

According the method in \cite{nka}, rewrite (\ref{x3}) as
\begin{eqnarray} \label{x33}
{\mathcal{T}}x^4=x_3^4+4(x_1+x_2)x_3^3+6(x_1^2+x_2^2)x_3^2-12(x_1^2x_2+x_1x_2^2)x_3+x_1^4+x_2^4+4x_1x_2^3+6x_1^2x_2^2.
\end{eqnarray}
Then the inner determinants corresponding to (\ref{x33}) are
$$\triangle_1^1=4,\ \ ~~~~~~~~\triangle_3^1=96x_1x_2,$$
$$\triangle_5^1=-192\times(3x_1^6+8x_1^5x_2+24x_1^4x_2^2-18x_1^3x_2^3+32x_1^2x_2^4+8x_1x_2^5+3x_2^6),$$
\begin{align*}
\triangle_7^1=&256\times(37x_1^{12}+144x_1^{11}x_2+432x_1^{10}x_2^2+229x_1^9x_2^3+480x_1^8x_2^4-144x_1^7x_2^5+3774x_1^6x_2^6\\
&-48x_1^5x_2^7 +912x_1^4x_2^8+2404x_1^3x_2^9+1344x_1^2x_2^{10}+336x_1x_2^{11}+37x_2^{12}).
\end{align*}
If $x_1,x_2\neq0$. Then when $x_1x_2>0$,
$$\triangle_3^1>0,$$
$$\triangle_5^1=-192\times(3x_1^6+8x_1^5x_2+15x_1^4x_2^2+(3x_1^2x_2-3x_1x_2^2)^2+21x_1^2x_2^4+8x_1x_2^5+3x_2^6)<0.$$
When $x_1x_2<0$,
$$\triangle_3^1<0,$$
$$\triangle_5^1=-192\times(2x_1^6+(x_1^3+4x_1^2x_2)^2+8x_1^4x_2^2-18x_1^3x_2^3+16x_1^2x_2^4+(4x_1x_2^2+x_2^3)^2+2x_2^6)<0.$$
If $x_2\neq0$,
\begin{align*}
\triangle_7^1=&256x_2^{12}\times(37y^{12}+144y^{11}+432y^{10}+229y^9+480y^8-144y^7+3774y^6-48y^5+912y^4\\
&+2404y^3+1344y^2+336y+37),
\end{align*}
where $y=\frac{x_1}{x_2}$. Using Lemma \ref{lem:25} for $\frac{\triangle_7^1}{256x_2^{12}}$, and we get the number of distinct real roots of $\frac{\triangle_7^1}{256x_2^{12}}=0$ with each $x_2\neq0$ is $0$.
Since $\frac{\triangle_7^1}{256x_2^{12}}=37>0$ for $y=0$, $\frac{\triangle_7^1}{256x_2^{12}}>0$ for any $y\in{\mathbb{R}}$, thus $\triangle_7^1>0$ for any$ x=(x_1,x_2)^{\top}\in{\mathbb{R}}^2$ and $x_2\neq0$.

Then if $x_1,x_2\neq0$, the number of distinct real roots $N$ of ${\mathcal{T}}x^4=0$ is
$$N=\mbox{var}[+,-,+,+,+]-\mbox{var}[+,+,+,-,+]=2-2=0$$
when $x_1x_2>0$, and
$$N=\mbox{var}[+,-,-,+,+]-\mbox{var}[+,+,-,-,+]=2-2=0$$
when $x_1x_2<0$.
Thus, ${\mathcal{T}}x^4=0$ if and only if $x=(0,0,0)^{\top}$. By Lemma \ref{lem:21}, ${\mathcal{T}}$ is positive definite.

Furthermore, in (\ref{th3.8a}) and (\ref{th3.8b}), it is easily verified that ${\mathcal{T}}x^4=0$ if and only if $x=(0,0,0)^{\top}$, and then, $\mathcal{T}$ is positive definite.

``\textbf{only if (Necessity)}." The condition,
$t_{iiij}t_{ijjj}=t_{jjjk}t_{jkkk}=t_{iiik}t_{ikkk}=0$, but $\{t_{iiij}+t_{ijjj},t_{jjjk}+t_{jkkk},t_{iiik}+t_{ikkk}\}$ are not zero for $i,j,k\in\{1,2,3\}$, $i\neq j$, $i\neq k$, $j\neq k$ can be divided into four cases, 
\begin{itemize}
    \item [($\romannumeral1$)] $t_{iiij}=t_{jjjk}=t_{ikkk}=0$ with $t_{ijjj}t_{jkkk}t_{iiik}=1$;
	\item [($\romannumeral2$)] $t_{iiij}=t_{jjjk}=t_{ikkk}=0$ with $t_{ijjj}t_{jkkk}t_{iiik}=-1$;
    \item [($\romannumeral3$)] $t_{iiij}=t_{jjjk}=t_{iiik}=0$ with $t_{ijjj}t_{jkkk}t_{ikkk}=1$;
    \item [($\romannumeral4$)] $t_{iiij}=t_{jjjk}=t_{iiik}=0$ with $t_{ijjj}t_{jkkk}t_{ikkk}=-1$.
\end{itemize}
Similar to the proof of Theorem \ref{th3.3}, we only need to consider $t_{ijjj}=t_{jkkk}=t_{iiik}=1$, $t_{ijjj}=t_{jkkk}=t_{iiik}=-1$, $t_{ijjj}=t_{jkkk}=t_{ikkk}=1$, and $t_{ijjj}=t_{jkkk}=t_{ikkk}=-1$ respectively. And it follows from the positive definiteness of $\mathcal{T}$ with Eq. \eqref{eq:3} that 
	$$t_{1122}=t_{1133}=t_{2233}=1.$$

($\romannumeral1$) We might take $t_{1112}=t_{2223}=t_{1333}=0$ and $t_{1222}=t_{2333}=t_{1113}=1$. Then for $x=(x_1,x_2,x_3)^{\top}\in{\mathbb{R}}^3$
\begin{align*}
{\mathcal{T}}x^4=&(x_1+x_2+x_3)^4-4(x_1^3x_2+x_2^3x_3+x_1x_3^3)\\
&+12[(t_{1123}-1)x_1^2x_2x_3+(t_{1223}-1)x_1x_2^2x_3+(t_{1233}-1)x_1x_2x_3^2].
\end{align*}

\textbf{Case 1.} At least one of $\{t_{1112},t_{2223},t_{1333}\}$ is $1$. Without loss the generality, we might take $t_{1123}=1$ and $t_{1233}=\min\{t_{1123},t_{1223},t_{1233}\}$. Let $x=(2,1,-1)^{\top}$, then
$${\mathcal{T}}x^4\leq (x_1+x_2+x_3)^4-4(x_1^3x_2+x_2^3x_3+x_1x_3^3)=-4<0.$$

\textbf{Case 2.} All of $\{t_{1123},t_{1223},t_{1233}\}$ are not $1$.

\textbf{Subcase 2.1} At least one of $\{t_{1123},t_{1223},t_{1233}\}$ is $-1$ but not all of them are $-1$. Without loss the generality, we might take $t_{1223}=-1$ and $t_{1123}=0$. Let $x=(2,-5,1)^{\top}$, then
$${\mathcal{T}}x^4\leq (x_1+x_2+x_3)^4-4(x_1^3x_2+x_2^3x_3+x_1x_3^3)-12x_1^2x_2x_3-24(x_1x_2^2x_3+x_1x_2x_3^2)=-52<0.$$

\textbf{Subcase 2.2} $t_{1123}=t_{1223}=t_{1233}=-1$, let $x=(1,1,1)^{\top}$, then
$${\mathcal{T}}x^4= (x_1+x_2+x_3)^4-4(x_1^3x_2+x_2^3x_3+x_1x_3^3)-24(x_1^2x_2x_3+x_1x_2^2x_3+x_1x_2x_3^2)=-3<0.$$

($\romannumeral2$) We might take $t_{1112}=t_{2223}=t_{1333}=-1$ and $t_{1222}=t_{2333}=t_{1113}=0$. Then for $x=(x_1,x_2,x_3)^{\top}\in{\mathbb{R}}^3$
\begin{align*}
{\mathcal{T}}x^4=&x_1^4+x_2^4+x_3^4-4(x_1^3x_2+x_2^3x_3+x_1x_3^3)\\&+12(t_{1123}x_1^2x_2x_3+t_{1223}x_1x_2^2x_3+t_{1233}x_1x_2x_3^2)+6(x_1^2x_2^2+x_1^2x_3^2+x_2^2x_3^2).
\end{align*}

\textbf{Case 1.} At least one of $\{t_{1112},t_{2223},t_{1333}\}$ is $1$. Without loss the generality, we might take $t_{1223}=1$ and $t_{1123}=\min\{t_{1123},t_{1223},t_{1233}\}$. Let $x=(-1,3,1)^{\top}$, then
$${\mathcal{T}}x^4\leq x_1^4+x_2^4+x_3^4-4(x_1^3x_2+x_2^3x_3+x_1x_3^3)+12x_1x_2^2x_3+6(x_1^2x_2^2+x_1^2x_3^2+x_2^2x_3^2)=-3<0.$$

\textbf{Case 2.} All of $\{t_{1123},t_{1223},t_{1233}\}$ are not $1$, i.e., at least one of $\{t_{1123},t_{1223},t_{1233}\}$ is $-1$. Without loss the generality, we might take $t_{1123}=-1$, let $x=(4,2,1)^{\top}$, then
$${\mathcal{T}}x^4\leq x_1^4+x_2^4+x_3^4-4(x_1^3x_2+x_2^3x_3+x_1x_3^3)-12x_1^2x_2x_3+6(x_1^2x_2^2+x_1^2x_3^2+x_2^2x_3^2)=-167<0.$$

($\romannumeral3$) We might take $t_{1112}=t_{2223}=t_{1113}=0$ and $t_{1222}=t_{2333}=t_{1333}=1$. Then for $x=(x_1,x_2,x_3)^{\top}\in{\mathbb{R}}^3$
\begin{align*}\label{2}
{\mathcal{T}}x^4=&(x_1+x_2+x_3)^4-4(x_1^3x_2+x_2^3x_3+x_1^3x_3)\\
&+12[(t_{1123}-1)x_1^2x_2x_3+(t_{1223}-1)x_1x_2^2x_3+(t_{1233}-1)x_1x_2x_3^2].
\end{align*}
Then, for $x=(1,1,-4)^\top$, we have\begin{align*}
	{\mathcal{T}}x^4=-52-48(t_{1123}+t_{1223}-4t_{1233})>0.
\end{align*}
So, the following cases could not occur,
 \begin{itemize}
    \item $t_{1233}=-1$;
	\item $t_{1233}=0$, and $t_{1123}+t_{1223}\geq-1$.
\end{itemize}
Then we only need to consider $t_{1233}=1$ and at least one of $\{t_{1123},t_{1233}\}$ is not $0$. Next suppose $t_{1233}=1$, and at least one of $\{t_{1123},t_{1233}\}$ is not $0$.

\textbf{Case 1.}  $t_{1223}=1$, let $x=(-1,2,1)^{\top}$, then
$${\mathcal{T}}x^4\leq(x_1+x_2+x_3)^4-4(x_1^3x_2+x_2^3x_3+x_1^3x_3)=-4<0.$$

\textbf{Case 2.} $t_{1223}\neq1$ and $t_{1123}=-1$, Let $x=(-3,1,4)^{\top}$, then
$${\mathcal{T}}x^4=(x_1+x_2+x_3)^4-4(x_1^3x_2+x_2^3x_3+x_1^3x_3)-24(x_1^2x_2x_3+x_1x_2^2x_3)=-36<0.$$

\textbf{Case 3.} $t_{1223}=-1$ and $t_{1123}\neq-1$, or $t_{1223}=0$ and $t_{1123}=1$. Let $x=(1,-1,1)^{\top}$, then
$${\mathcal{T}}x^4\leq(x_1+x_2+x_3)^4-4(x_1^3x_2+x_2^3x_3+x_1^3x_3)-12x_1x_2^2x_3=-7<0.$$

($\romannumeral4$) We might take $t_{1112}=t_{2223}=t_{1113}=0$ and $t_{1222}=t_{2333}=t_{1333}=-1$. Then for $x=(x_1,x_2,x_3)^{\top}\in{\mathbb{R}}^3$
\begin{align*}
{\mathcal{T}}x^4=&x_1^4+x_2^4+x_3^4-4(x_1x_2^3+x_2x_3^3+x_1x_3^3)\\
&+12(t_{1123}x_1^2x_2x_3+t_{1223}x_1x_2^2x_3+t_{1233}x_1x_2x_3^2)+6(x_1^2x_2^2+x_1^2x_3^2+x_2^2x_3^2).
\end{align*}
Then, for $x=(6,2,1)^\top$, we have\begin{align*}
	{\mathcal{T}}x^4=-207+144(6t_{1123}+2t_{1223}+t_{1233})>0.
\end{align*}
So, the following cases could not occur,
 \begin{itemize}
    \item $t_{1123}=0$ and $t_{1223}\neq1$;
    \item $t_{1123}=0$, $t_{1223}=1$ and $t_{1233}=-1$;
	\item $t_{1123}=-1$.
\end{itemize}
Then we only need to consider $t_{1123}=0$ with $t_{1223}=1$ and $t_{1233}\neq-1$, or $t_{1123}=1$ with $t_{1223}\neq0$ or $t_{1233}\neq0$.

\textbf{Case 1.} $t_{1123}=0$, $t_{1223}=1$ and $t_{1233}\neq-1$. Let $x=(-1,4,2)^{\top}$, then
$${\mathcal{T}}x^4\leq x_1^4+x_2^4+x_3^4-4(x_1^3x_2+x_2^3x_3+x_1^3x_3)+12x_1x_2^2x_3+6(x_1^2x_2^2+x_1^2x_3^2+x_2^2x_3^2)=-95<0.$$

\textbf{Case 2.} $t_{1123}=1$.

\textbf{Subcase 2.1} When $t_{1233}=1$. Let $x=(4,-1,3)^{\top}$, then
$${\mathcal{T}}x^4\leq x_1^4+x_2^4+x_3^4-4(x_1^3x_2+x_2^3x_3+x_1^3x_3)+12(x_1^2x_2x_3+x_1x_2^2x_3+x_1x_2x_3^2)+6(x_1^2x_2^2+x_1^2x_3^2+x_2^2x_3^2)=-66<0.$$

\textbf{Subcase 2.2} When $t_{1233}\neq1$ and $t_{1223}=1$. Let $x=(4,3,-1)^{\top}$, then
$${\mathcal{T}}x^4\leq x_1^4+x_2^4+x_3^4-4(x_1^3x_2+x_2^3x_3+x_1^3x_3)+12(x_1^2x_2x_3+x_1x_2^2x_3)+6(x_1^2x_2^2+x_1^2x_3^2+x_2^2x_3^2)=-60<0.$$

\textbf{Subcase 2.3} When $t_{1233}=0$ and $t_{1223}=-1$. Let $x=(1,4,2)^{\top}$, then $${\mathcal{T}}x^4=x_1^4+x_2^4+x_3^4-4(x_1^3x_2+x_2^3x_3+x_1^3x_3)+12(x_1^2x_2x_3-x_1x_2^2x_3)+6(x_1^2x_2^2+x_1^2x_3^2+x_2^2x_3^2)=-59<0$$

\textbf{Subcase 2.4} When $t_{1233}=-1$ and $t_{1223}=0$. Let $x=(6,4,-5)^{\top}$, then
$${\mathcal{T}}x^4=x_1^4+x_2^4+x_3^4-4(x_1^3x_2+x_2^3x_3+x_1^3x_3)+12(x_1^2x_2x_3-x_1x_2x_3^2)+6(x_1^2x_2^2+x_1^2x_3^2+x_2^2x_3^2)=-263<0$$

\textbf{Subcase 2.5} When $t_{1233}=t_{1223}=-1$. Let $x=(1,1,1)^{\top}$, then
$${\mathcal{T}}x^4=x_1^4+x_2^4+x_3^4-4(x_1^3x_2+x_2^3x_3+x_1^3x_3)+12(x_1^2x_2x_3-x_1x_2^2x_3-x_1x_2x_3^2)+6(x_1^2x_2^2+x_1^2x_3^2+x_2^2x_3^2)=-3<0.$$
The necessity is proved.
\end{proof}

\begin{theorem}\label{th3.9}
Let ${\mathcal{T}}=(t_{ijkl})\in\widehat{{\mathcal{E}}}_{4,3}$ and $t_{iiij}t_{ijjj}=t_{jjjk}t_{jkkk}=t_{iiik}t_{ikkk}=t_{iiik}+t_{ikkk}=0$, $t_{iiij}+t_{ijjj},t_{jjjk}+t_{jkkk}\neq0$ for $i,j,k\in\{1,2,3\}$, $i\neq j$, $i\neq k$, $j\neq k$. Then $\mathcal{T}$ is positive definite if and only if one of the following conditions is satisfied.
\begin{itemize}
  \item [(a)]  When $t_{iiij},t_{jkkk}\in\{-1,1\}$,
  \begin{itemize}
    \item [(a$_1$)] $t_{1123}=t_{1223}=t_{1233}=0$, and $t_{iikk}\in\{0,1\}$, $t_{iijj}=t_{jjkk}=1$, or
    \item [(a$_2$)] $t_{ijjk}=\pm1$, $t_{iijk}=t_{ijkk}=0$, and $t_{1122}=t_{2233}=t_{1133}=1$, or
    \item [(a$_3$)] $t_{1122}=t_{2233}=t_{1133}=-t_{iiij}t_{jkkk}t_{ijjk}=1$, $t_{iijk}t_{jkkk}=-1$ with $t_{ijkk}=0$ or $t_{ijkk}t_{iiij}=-1$ with $t_{iijk}=0$.
  \end{itemize}

  \item [(b)] When $t_{iiij},t_{jjjk}\in\{-1,1\}$,
  \begin{itemize}
    \item [(b$_1$)] $t_{1123}=t_{1223}=t_{1233}=0$, and $t_{iikk}\in\{0,1\}$, $t_{iijj}=t_{jjkk}=1$, or
    \item [(b$_2$)] $t_{ijkk}=\pm1$, $t_{iijk}=t_{ijjk}=0$, and $t_{1122}=t_{2233}=t_{1133}=1$, or
    \item [(b$_3$)] $t_{1122}=t_{2233}=t_{1133}=-t_{iiij}t_{jjjk}t_{ijjk}=-t_{ijkk}t_{iiij}=1$, $t_{iijk}=0$.
  \end{itemize}

  \item [(c)] When $t_{ijjj},t_{jjjk}\in\{-1,1\}$, and $t_{ijjj}t_{jjjk}t_{ijjk}=1$,
  \begin{itemize}
    \item [(c$_1$)] $t_{ijjj}t_{ijkk}=t_{jjjk}t_{iijk}=t_{1122}=t_{2233}=t_{1133}=1$, or
    \item [(c$_2$)] $t_{iijk}=t_{ijkk}=0$, and $t_{iikk}\in\{0,1\}$, $t_{iijj}=t_{jjkk}=1$.
  \end{itemize}
\end{itemize}
\end{theorem}
\begin{proof}
``\textbf{if (Sufficiency)}."
(a) Suppose $t_{iiij},t_{jkkk}\in\{-1,1\}$ and $t_{ijjj}=t_{jjjk}=t_{ikkk}=t_{iiik}=0$ for $i,j,k\in\{1,2,3\}$, $i\neq j$, $i\neq k$, $j\neq k$.

(a$_1$) If $t_{1123}=t_{1223}=t_{1233}=0$, and $t_{jjkk}\in\{0,1\}$, $t_{iikk}=t_{iijj}=1$. Then for $x=(x_1,x_2,x_3)^{\top}\in{\mathbb{R}}^3$,
\begin{align}
{\mathcal{T}}x^4&\geq x_1^4+x_2^4+x_3^4+4(t_{iiij}x_i^3x_j+t_{jkkk}x_jx_k^3)+6(x_i^2x_j^2+x_j^2x_k^2)\nonumber\\
 &= x_j^4+(x_i^2+2t_{iiij}x_ix_j)^2+(x_k^2+2t_{jkkk}x_jx_k)^2+2(x_i^2x_j^2+x_j^2x_k^2)\geq0.
\end{align}\label{th3.9a1}

(a$_2$) If $t_{ijjk}=\pm1$, $t_{iijk}=t_{ijkk}=0$, and $t_{1122}=t_{2233}=t_{1133}=1$. Then for $x=(x_1,x_2,x_3)^{\top}\in{\mathbb{R}}^3$,
\begin{align}
{\mathcal{T}}x^4=&(x_i^2+2t_{iiij}x_ix_j+x_k^2+2t_{jkkk}x_jx_k)^2+x_j^4+(t_{iiij}x_jx_k+t_{jkkk}x_ix_j-2x_ix_k)^2\nonumber\\
&+(x_jx_k+t_{ijjk}x_ix_j)^2\geq0,\label{th3.9a2}
\end{align}
when $t_{iiik}t_{jkkk}t_{ijjk}=1$,
\begin{align}\label{th3.9a3}
{\mathcal{T}}x^4=&(x_k^2+2t_{jkkk}x_jx_k-x_i^2-2t_{iiij}x_ix_j)^2+(t_{jkkk}x_ix_j+t_{iiij}x_jx_k+2x_ix_k)^2\nonumber\\
&+(t_{jkkk}x_jx_k-t_{iiij}x_ix_j)^2+(x_j^2+2t_{ijjk}x_ix_k)^2\geq0,
\end{align}
when $t_{iiij}t_{jkkk}t_{ijjk}=-1$.

(a$_3$) Let $t_{1122}=t_{2233}=t_{1133}=1$. If $t_{iiij}=t_{jkkk}=-t_{ijjk}=1$, and one of $\{t_{iijk},t_{ijkk}\}$ is $-1$, the other one is $0$. Without loss the generality, we might take $t_{1113}=t_{2333}=-t_{1123}=-t_{1223}=1$, $t_{1123}=0$. Then for $x=(x_1,x_2,x_3)^{\top}\in{\mathbb{R}}^3$,
\begin{equation}\label{x1}
{\mathcal{T}}x^4=(x_1^2+2x_1x_2-2x_2x_3-x_3^2)^2+2(2x_1x_3-x_1x_2+x_2x_3)^2+x_2^4-4x_1x_2x_3^2.
\end{equation}
From (\ref{x1}), we can deduce that ${\mathcal{T}}x^4>0$, when the signs of $x_1x_2<0$ are different.

By Lemma \ref{lem:23}, when $x_1=0$, or $x_2=0$, or $x_3=0$, ${\mathcal{T}}x^4\geq0$, with equality if and only if $x_1=x_2=x_3=0$.

According the method in \cite{nka}, rewrite (\ref{x1}) as
\begin{equation} \label{x11}
{\mathcal{T}}x^4=x_3^4+4x_2x_3^3+6(x_1^2+x_2^2)x_3^2-12(x_1^2x_2+x_1x_2^2)x_3+x_1^4+x_2^4+4x_1^3x_2+6x_1^2x_2^2.
\end{equation}
Then the inner determinants corresponding to (\ref{x11}) are
$$\triangle_1^1=4,\ \ ~~~~~~~~\triangle_3^1=-48x_1^2,$$
$$\triangle_5^1=-192\times[4x_1^6+(2x_1^3-x_1^2x_2)^2+2x_1^4x_2^2+(9x_1^2x_2+\frac{16}{3}x_1x_2^2)^2+3(3x_1x_2^2+x_2^3)^2+\frac{32}{9}x_1^2x_2^4],$$
\begin{align*}
\triangle_7^1&= 256\times[64x_1^{12}+192x_1^{11}x_2+912x_1^{10}x_2^2+5536x_1^9x_2^3+930x_1^8x_2^4+2904x_1^7x_2^5+(12x_1^4x_2^2\\
&~~~-13x_1^2x_2^4)^2 +216x_1^5x_2^7+2228x_1^4x_2^8+2028x_1^3x_2^9+990x_1^2x_2^{10}+252x_1x_2^{11}+37x_2^{12}].
\end{align*}
If $x_1x_2>0$, $\triangle_3^1<0, \ \ \triangle_5^1<0,\ \ \triangle_7^1>0.$
Then the number of distinct real roots $N$ of ${\mathcal{T}}x^4=0$ is
$$N=\mbox{var}[+,-,-,+,+]-\mbox{var}[+,+,-,-,+]=2-2=0,$$
when $x_1x_2>0$. Thus, an application of  Lemma \ref{lem:26} yields that ${\mathcal{T}}x^4=0$ if and only if $x=(0,0,0)^{\top}$. By Lemma \ref{lem:21}, ${\mathcal{T}}$ is positive definite.

When $t_{iiik}=t_{jjjk}=-1$, $t_{ijkk}=-1$ with one of $\{t_{iijk},t_{ijjk}\}$ is $1$, the other one is $0$, and
 $-t_{iiik}=t_{jjjk}=1$, $t_{ijkk}=1$ with $t_{iijk}=0$, $t_{ijjk}=1$, or $t_{iijk}=-1$, $t_{ijjk}=0$ are similar to (\ref{x1}), and by comparing it with (\ref{x1}) can infer that ${\mathcal{T}}x^4\geq0$ for any $x=(x_1,x_2,x_3)^{\top}\in{\mathbb{R}}^3$, and ${\mathcal{T}}x^4=0$ if and only if $x=(0,0,0)^{\top}$.

(b) Suppose $t_{iiij},t_{jjjk}\in\{-1,1\}$ and $t_{iiik}=t_{ijjj}=t_{ikkk}=t_{jkkk}=0$ for $i,j,k\in\{1,2,3\}$, $i\neq j$, $i\neq k$, $j\neq k$.

(b$_1$) If $t_{1123}=t_{1223}=t_{1233}=0$, and $t_{iikk}\in\{0,1\}$, $t_{iijj}=t_{jjkk}=1$. Then for $x=(x_1,x_2,x_3)^{\top}\in{\mathbb{R}}^3$,
\begin{align}\label{th3.9b1}
{\mathcal{T}}x^4&\geq x_1^4+x_2^4+x_3^4+4(t_{iiij}x_i^3x_k+t_{jjjk}x_j^3x_k)+6(x_i^2x_j^2+x_j^2x_k^2)\nonumber\\
&=(x_i^2+2t_{iiik}x_ix_j)^2+(x_j^2+2t_{jjjk}x_jx_k)^2+x_k^4+2(x_i^2x_j^2+x_j^2x_k^2)\geq0.
\end{align}

(b$_2$) If $t_{ijkk}=\pm1$, $t_{iijk}=t_{ijjk}=0$, and $t_{1122}=t_{2233}=t_{1133}=1$. Then for $x=(x_1,x_2,x_3)^{\top}\in{\mathbb{R}}^3$,
\begin{equation}\label{th3.9b2}
{\mathcal{T}}x^4=(x_i^2+2t_{iiij}x_ix_j+x_k^2)^2+(x_j^2+2t_{jjjk}x_jx_k+2t_{iiij}t_{jjjk}x_ix_k)^2+2(x_jx_k-t_{iiij}t_{jjjk}x_ix_j)^2\geq0,
\end{equation}
when $t_{iiij}t_{ijkk}=1$,
\begin{align}\label{th3.9b3}
{\mathcal{T}}x^4=&(x_i^2+2t_{iiij}x_ix_j-x_k^2)^2+(x_j^2+2t_{jjjk}x_jx_k-t_{iiij}t_{jjjk}x_ix_k)^2+(x_jx_k+t_{iiij}t_{jjjk}x_ix_j)^2\nonumber\\
&+(2x_ix_k+t_{ijkk}x_ix_k)^2+3x_i^2x_k^2+x_i^2x_j^2\geq0,
\end{align}
when $t_{iiij}t_{ijkk}=-1$.

(b$_3$) Let $t_{1122}=t_{2233}=t_{1133}=1$ and $t_{iijk}=0$. If $t_{iiij}=t_{jjjk}=-t_{ijjk}=1$. Without loss the generality, we might take $t_{1112}=t_{2223}=-t_{1233}=1$, then for $x=(x_1,x_2,x_3)^{\top}\in{\mathbb{R}}^3$,
\begin{align}\label{xx1}
{\mathcal{T}}x^4&=x_1^4+x_2^4+x_3^4+4(x_1^3x_2+x_2^3x_3)-12(x_1x_2x_3^2+x_1x_2^2x_3)+6(x_1^2x_2^2+x_2^2x_3^2+x_1^2x_3^2)\nonumber\\
&=x_1^4+4x_2x_1^3+6(x_2^2+x_3^2)x_1^2-12(x_2^2x_3+x_2x_3^2)x_1+x_2^4+x_3^4+4x_2^3x_3+6x_2^2x_3^2.
\end{align}

By Lemma \ref{lem:23}, when $x_1=0$, or $x_2=0$, or $x_3=0$, ${\mathcal{T}}x^4\geq0$, with equality if and only if $x_1=x_2=x_3=0$.

According the method in \cite{nka}, the inner determinants corresponding to (\ref{xx1}) are
$$\triangle_1^1=4,\ \ ~~~~~~~~\triangle_3^1=-48x_3^2,$$
\begin{align*}
\triangle_5^1&=-192\times[(\sqrt{2}x_2^3+\frac{9\sqrt{2}}{2}x_2^2x_3+6x_2x_3^2+2x_3^3)^2+(x_2^3+2x_3^3)^2+(\frac{44-29\sqrt{2}}{4}x_2^2x_3+4x_2x_3^2)^2\\
&~~~+(40-18\sqrt{2})x_2^2x_3^4+\frac{1180\sqrt{2}-1661}{8}x_2^4x_3^2],
\end{align*}
\begin{align*}\triangle_7^1&=257\times(37x_2^{12}+444x_2^{11}x_3+2334x_2^{10}x_3^2+6724x_2^9x_3^3+10605x_2^8x_3^4+6936x_2^7x_3^5-2616x_2^6x_3^6\\
&~~~-3744x_2^5x_3^7+3216x_2^4x_3^8+768x_2^3x_3^9+1152x_2^2x_3^{10}+64x_3^{12}).
\end{align*}
If $x_3\neq0$,
\begin{align*}
\triangle_7^1&=257x_3^{12}\times(37y^{12}+444y^{11}+2344y^{10}+6724y^9+10605y^8+6936y^7+2616y^6-3744y^5\\
&~~~+3216y^4+768y^3+1152y^2+64),
\end{align*}
where $y=\frac{x_2}{x_3}$. Using Lemma \ref{lem:26}  for $\frac{\triangle_7^1}{256x_3^{12}}$, and we get the number of distinct real roots of $\frac{\triangle_7^1}{256x_3^{12}}=0$ with each $x_3\neq0$ is $0$.
Since $\frac{\triangle_7^1}{257x_3^{12}}=64>0$ for $y=0$, $\frac{\triangle_7^1}{257x_3^{12}}>0$ for any $y\in{\mathbb{R}}$, thus $\triangle_7^1>0$ for any$x=(x_1,x_2)^{\top}\in{\mathbb{R}}^2$ and $x_2\neq0$.

Then if $x_1,x_2\neq0$, the number of distinct real roots $N$ of ${\mathcal{T}}x^4=0$ is
$$N=\mbox{var}[+,-,-,+,+]-\mbox{var}[+,+,-,-,+]=2-2=0.$$
Thus, ${\mathcal{T}}x^4=0$ if and only if $x=(0,0,0)^{\top}$. By Lemma \ref{lem:21}, ${\mathcal{T}}$ is positive definite.

Similarly, the conclusion is established under the condition that\begin{center}
	$t_{iiij}=t_{jjjk}=t_{ijjk}=-1$, $t_{iiij}=-t_{jjjk}=t_{ijjk}=1$ and $-t_{iiij}=t_{jjjk}=t_{ijjk}=1$.
\end{center}

(c) Suppose $t_{ijjj},t_{jjjk}\in\{-1,1\}$, $t_{iiik}=t_{iiij}=t_{ikkk}=t_{jkkk}=0$ and $t_{ijjj}t_{jjjk}t_{ijjk}=1$ for $i,j,k\in\{1,2,3\}$, $i\neq j$, $i\neq k$, $j\neq k$.

(c$_1$)  $t_{ijjj}t_{ijkk}=t_{jjjk}t_{iijk}=t_{1122}=t_{2233}=t_{1133}=1$. Then for $x=(x_1,x_2,x_3)^{\top}\in{\mathbb{R}}^3$,
\begin{align}\label{th3.9c1}
{\mathcal{T}}x^4&=(x_i^2-x_k^2)^2+(x_j^2+2t_{ijjj}x_ix_j+2t_{jjjk}x_jx_k+2t_{ijjk}x_ix_k)^2+2(x_ix_j+t_{iijk}x_ix_k)^2\nonumber\\
&~~~+2(x_ix_k+t_{ijkk}x_jx_k)^2\geq0.
\end{align}

(c$_2$) $t_{iijk}=t_{ijkk}=0$, and $t_{iikk}\in\{0,1\}$, $t_{iijj}=t_{jjkk}=1$. Then for $x=(x_1,x_2,x_3)^{\top}\in{\mathbb{R}}^3$,
\begin{equation}\label{th3.9c2}
{\mathcal{T}}x^4=x_i^4+x_k^4+(x_j^2+2t_{ijjj}x_ix_j+2t_{jjjk}x_jx_k)^2+2(x_ix_j+t_{ijjk}x_jx_k)^2\geq0.
\end{equation}

Furthermore, in ((\ref{th3.9a1})-(\ref{th3.9a3}), (\ref{th3.9b1})-(\ref{th3.9b3}), (\ref{th3.9c1}) and (\ref{th3.9c2}), it is easily verified that ${\mathcal{T}}x^4=0$ if and only if $x=(0,0,0)^{\top}$, and then, $\mathcal{T}$ is positive definite by Lemma \ref{lem:21}.

``\textbf{only if (Necessity)}." The conditions, 
$t_{iiij}t_{ijjj}=t_{jjjk}t_{jkkk}=t_{iiik}t_{ikkk}=t_{iiik}+t_{ikkk}=0$, and $t_{iiij}+t_{ijjj},t_{jjjk}+t_{jkkk}\neq0$ (that is, two of $\{t_{iiij},t_{ijjj},t_{jjjk},t_{jkkk}\}$ are not $0$) for $i,j,k\in\{1,2,3\}$, $i\neq j$, $i\neq k$, $j\neq k$, can be divided into three cases, 
\begin{itemize}
    \item [($\romannumeral1$)] $t_{iiij},t_{jkkk}\in\{-1,1\}$;
	\item [($\romannumeral2$)] $t_{iiij},t_{jjjk}\in\{-1,1\}$;
    \item [($\romannumeral3$)] $t_{ijjj},t_{jjjk}\in\{-1,1\}$.
\end{itemize}
And we only need to consider $t_{iiij}=t_{jkkk}=1$, $t_{iiij}=t_{jjjk}=1$, and $t_{ijjj}=t_{jjjk}=1$, respectively. It follows from the positive definiteness of $\mathcal{T}$ with Eq. \eqref{eq:3} that \begin{center}
	$t_{iijj}=t_{jjkk}=1$, $t_{iikk}\in\{0,1\}$.
\end{center}
And in the meantime, for $x=(1,1,1)^\top$, we have\begin{align*}
	{\mathcal{T}}x^4=23+12(t_{1123}+t_{1223}+t_{1233})+6t_{1133}>0.
\end{align*}
So, the following cases  could not occur, \begin{itemize}
	\item Two of $\{t_{1123},t_{1223},t_{1233}\}$ are $-1$, the third is $0$ and $t_{1133}=0$;
	\item $t_{1123}=t_{1223}=t_{1233}=-1$.
\end{itemize}
Next we use a proof by contradiction to prove the necessity.

($\romannumeral1$) We might take $t_{1112}=t_{2333}=1$ and $t_{1113}=t_{1333}=t_{1222}=t_{2223}=0$. Then for $x=(x_1,x_2,x_3)^{\top}\in{\mathbb{R}}^3$
\begin{align*}
{\mathcal{T}}x^4=&x_1^4+x_2^4+x_3^4+4(x_1^3x_2+x_2x_3^3)+12(t_{1123}x_1^2x_2x_3+t_{1223}x_1x_2^2x_3+t_{1233}x_1x_2x_3^2)\\
&+6(x_1^2x_2^2+x_2^2x_3^2+t_{1133}x_1^2x_3^2).
\end{align*}

\textbf{Case 1.} $t_{1123}=t_{1233}=1$. Let $x=(2,-1,2)^{\top}$, then
$${\mathcal{T}}x^4\leq x_1^4+x_2^4+x_3^4+4(x_1^3x_2+x_2x_3^3)+12(x_1^2x_2x_3+x_1x_2^2x_3+x_1x_2x_3^2)+6(x_1^2x_2^2+x_1^2x_3^2+x_2^2x_3^2)=-31<0.$$

\textbf{Case 2.} $t_{1123}\neq1$ or $t_{1233}\neq1$, and at least one of $\{t_{1123},t_{1223}t_{1233}\}$ is not $-1$. Without loss the generality, suppose $1>t_{1123}=\min\{t_{1123},t_{1233}\}$.

\textbf{Subcase 2.1} When $t_{1123}+1<t_{1223}+t_{1233}$. Let $x=(-1,1,1)^{\top}$, then
$${\mathcal{T}}x^4\leq x_1^4+x_2^4+x_3^4+4(x_1^3x_2+x_2x_3^3)+12(x_1x_2^2x_3+x_1x_2x_3^2)+6(x_1^2x_2^2+x_1^2x_3^2+x_2^2x_3^2)=-3<0.$$

\textbf{Subcase 2.2} When $t_{1123}=-1$ and $t_{1223}+t_{1233}=0$. Let $x=(-3,1,1)^{\top}$, then
$${\mathcal{T}}x^4\leq x_1^4+x_2^4+x_3^4+4(x_1^3x_2+x_2x_3^3)-12x_1^2x_2x_3+6(x_1^2x_2^2+x_1^2x_3^2+x_2^2x_3^2)=-15<0.$$

\textbf{Subcase 2.3} When $t_{1123}=t_{1223}=0$ and $t_{1233}=1$. Let $x=(1,-1,3)^{\top}$ then
$${\mathcal{T}}x^4\leq x_1^4+x_2^4+x_3^4+4(x_1^3x_2+x_2x_3^3)+12x_1x_2x_3^2+6(x_1^2x_2^2+x_1^2x_3^2+x_2^2x_3^2)=-23<0.$$

\textbf{Subcase 2.4} When $t_{1123}=t_{1233}=0$, $t_{1223}=1$ and $t_{1133}=0$. Let $x=(-2,1,1)^{\top}$, then
$${\mathcal{T}}x^4=x_1^4+x_2^4+x_3^4+4(x_1^3x_2+x_2x_3^3)+12x_1x_2^2x_3+6(x_1^2x_2^2+x_2^2x_3^2)=-4<0.$$

\textbf{Subcase 2.5} When $t_{1123}=0$ and $-t_{1223}=t_{1233}=1$. Let $x=(1,-1,1)^{\top}$ then
$${\mathcal{T}}x^4\leq  x_1^4+x_2^4+x_3^4+4(x_1^3x_2+x_2x_3^3)-12(x_1x_2^2x_3-x_1x_2x_3^2)+6(x_1^2x_2^2+x_1^2x_3^2+x_2^2x_3^2)=-11<0.$$

\textbf{Subcase 2.6} When $t_{1123}=t_{1233}=-1$ and $t_{1223}=0$. Let $x=(1,-1,1)^{\top}$, then
$${\mathcal{T}}x^4\leq x_1^4+x_2^4+x_3^4+4(x_1^3x_2+x_2x_3^3)-12(x_1^2x_2x_3+x_1x_2x_3^2)+6(x_1^2x_2^2+x_1^2x_3^2+x_2^2x_3^2)=-11<0.$$

\textbf{Subcase 2.7} When $t_{1123}=t_{1233}=0$, $t_{1223}=-1$ and $t_{1133}=0$. Let $x=(3,-1,3)^{\top}$, then
$${\mathcal{T}}x^4=x_1^4+x_2^4+x_3^4+4(x_1^3x_2+x_2x_3^3)-12x_1x_2^2x_3+6(x_1^2x_2^2+x_2^2x_3^2)=-53<0.$$

($\romannumeral2$) We might take $t_{1112}=t_{2223}=1$ and $t_{1113}=t_{1333}=t_{1222}=t_{2333}=0$. Then for $x=(x_1,x_2,x_3)^{\top}\in{\mathbb{R}}^3$
\begin{align*}
{\mathcal{T}}x^4=&x_1^4+x_2^4+x_3^4+4(x_1^3x_2+x_2^3x_3)+12(t_{1123}x_1^2x_2x_3+t_{1223}x_1x_2^2x_3+t_{1233}x_1x_2x_3^2)\\
&+6(x_1^2x_2^2+x_2^2x_3^2+t_{1133}x_1^2x_3^2).
\end{align*}

\textbf{Case 1.} $t_{1123}=1$.

\textbf{Subcase 1.1} When $t_{1233}=1$ or $t_{1223}=t_{1233}$. Let $x=(2,-1,1)^{\top}$, then
$${\mathcal{T}}x^4\leq x_1^4+x_2^4+x_3^4+4(x_1^3x_2+x_2^3x_3)+12x_1^2x_2x_3+6(x_1^2x_2^2+x_1^2x_3^2+x_2^2x_3^2)=-12<0.$$

\textbf{Subcase 1.2} When $t_{1233}\neq1$ and $t_{1223}=1$. Let $x=(2,3,-2)^{\top}$, then
$${\mathcal{T}}x^4\leq x_1^4+x_2^4+x_3^4+4(x_1^3x_2+x_2^3x_3)+12(x_1^2x_2x_3+x_1x_2^2x_3)+6(x_1^2x_2^2+x_1^2x_3^2+x_2^2x_3^2)=-199<0.$$

\textbf{Subcase 1.3} When $t_{1223}=-1$ and $t_{1233}=0$. Let $x=(1,-1,1)^{\top}$, then
$${\mathcal{T}}x^4\leq  x_1^4+x_2^4+x_3^4+4(x_1^3x_2+x_2^3x_3)+12(x_1^2x_2x_3-x_1x_2^2x_3)+6(x_1^2x_2^2+x_1^2x_3^2+x_2^2x_3^2)=-11<0.$$

\textbf{Subcase 1.4} When$t_{1223}=0$ and $t_{1233}=-1$. Let $x=(1,1,-1)^{\top}$, then
$${\mathcal{T}}x^4\leq  x_1^4+x_2^4+x_3^4+4(x_1^3x_2+x_2^3x_3)+12(x_1^2x_2x_3-x_1x_2x_3^2)+6(x_1^2x_2^2+x_1^2x_3^2+x_2^2x_3^2)=-3<0.$$

\textbf{Case 2.} $t_{1123}=0$.

\textbf{Subcase 2.1} When $t_{1223}=t_{1233}=1$. Let $x=(-3,2,3)^{\top}$, then
$${\mathcal{T}}x^4\leq x_1^4+x_2^4+x_3^4+4(x_1^3x_2+x_2^3x_3)+12(x_1x_2^2x_3+x_1x_2x_3^2)+6(x_1^2x_2^2+x_1^2x_3^2+x_2^2x_3^2)=-104<0.$$

\textbf{Subcase 2.2} When $t_{1223}=1$ and $t_{1233}\neq1$. Let $x=(1,3,-1)^{\top}$, then
$${\mathcal{T}}x^4\leq x_1^4+x_2^4+x_3^4+4(x_1^3x_2+x_2^3x_3)+12x_1x_2^2x_3+6(x_1^2x_2^2+x_1^2x_3^2+x_2^2x_3^2)=-7<0.$$

\textbf{Subcase 2.3} When $t_{1223}=0$, $t_{1233}\neq0$, and $t_{1133}=0$. Let $x=(-3,1,3)^{\top}$, then
$${\mathcal{T}}x^4=x_1^4+x_2^4+x_3^4+4(x_1^3x_2+x_2^3x_3)+12x_1x_2x_3^2+6(x_1^2x_2^2+x_2^2x_3^2)=-149<0$$
with $t_{1233}=-1$. Let $x=(2,1,-3)^{\top}$, then
$${\mathcal{T}}x^4=x_1^4+x_2^4+x_3^4+4(x_1^3x_2+x_2^3x_3)-12x_1x_2x_3^2+6(x_1^2x_2^2+x_2^2x_3^2)=-20<0$$
with $t_{1233}=-1$.

\textbf{Subcase 2.4} When $t_{1223}=-1$ and $t_{1233}\neq-1$. Let $x=(1,-3,1)^{\top}$, then
$${\mathcal{T}}x^4\leq x_1^4+x_2^4+x_3^4+4(x_1^3x_2+x_2^3x_3)-12x_1x_2^2x_3+6(x_1^2x_2^2+x_1^2x_3^2+x_2^2x_3^2)=-31<0.$$

\textbf{Subcase 3.} $t_{1123}=-1$.

\textbf{Subcase 3.1} When $t_{1223},t_{1233}\in\{0,1\}$. Let $x=(-3,1,1)^{\top}$, then
$${\mathcal{T}}x^4\leq x_1^4+x_2^4+x_3^4+4(x_1^3x_2+x_2^3x_3)-12x_1^2x_2x_3+6(x_1^2x_2^2+x_1^2x_3^2+x_2^2x_3^2)=-15<0.$$

\textbf{Subcase 3.2} When $t_{1223}=-1$ and $t_{1233}\neq-1$. Let $x=(1,-5,2)^{\top}$, then
$${\mathcal{T}}x^4\leq x_1^4+x_2^4+x_3^4+4(x_1^3x_2+x_2^3x_3)-12(x_1^2x_2x_3+x_1x_2^2x_3)+6(x_1^2x_2^2+x_1^2x_3^2+x_2^2x_3^2)=-84<0.$$

\textbf{Subcase 3.3} When $t_{1233}=-1$ and $t_{1223}\neq-1$. Let $x=(-5,2,1)^{\top}$, then
$${\mathcal{T}}x^4\leq x_1^4+x_2^4+x_3^4+4(x_1^3x_2+x_2^3x_3)-12(x_1^2x_2x_3+x_1x_2x_3^2)+6(x_1^2x_2^2+x_1^2x_3^2+x_2^2x_3^2)=-32<0.$$

($\romannumeral3$) We might take $t_{1222}=t_{2223}=1$ and $t_{1113}=t_{1112}=t_{2333}=t_{1333}=0$. Then for $x=(x_1,x_2,x_3)^{\top}\in{\mathbb{R}}^3$
\begin{align*}
{\mathcal{T}}x^4&=x_1^4+x_2^4+x_3^4+4(x_1x_2^3+x_2^3x_3)+12(t_{1123}x_1^2x_2x_3+t_{1223}x_1x_2^2x_3+t_{1233}x_1x_2x_3^2)\\
&+6(x_1^2x_2^2+x_2^2x_3^2+t_{1133}x_1^2x_3^2).
\end{align*}

\textbf{Case 1}  $t_{1223}\neq1$.

\textbf{Subcase 1.1} When $t_{1123}\geq t_{1223}$ or $t_{1233}\geq t_{1223}$, without loss the generality, suppose $t_{1123}\geq t_{1223}$. Let $x=(2,-7,1)^{\top}$, then
$${\mathcal{T}}x^4\leq x_1^4+x_2^4+x_3^4+4(x_1x_2^3+x_2^3x_3)-12x_1x_2x_3^2+6(x_1^2x_2^2+x_1^2x_3^2+x_2^2x_3^2)=-36<0.$$

\textbf{Subcase 1.2} When $t_{1123}<t_{1223}$ and $t_{1233}\geq t_{1223}$, i.e., $t_{1223}=0$ and $t_{1123}=t_{1233}=-1$. Let $x=(5,3,5)^{\top}$, then
$${\mathcal{T}}x^4\leq x_1^4+x_2^4+x_3^4+4(x_1x_2^3+x_2^3x_3)-12(x_1^2x_2x_3+x_1x_2x_3^2)+6(x_1^2x_2^2+x_1^2x_3^2+x_2^2x_3^2)=-139<0.$$

\textbf{Case 2} $t_{1223}=1$.

\textbf{Subcase 2.1} When $t_{1123}=t_{1233}=1$, and $t_{1133}=0$. Let $x=(2,-1,1)^{\top}$,
$${\mathcal{T}}x^4=x_1^4+x_2^4+x_3^4+4(x_1x_2^3+x_2^3x_3)+12(x_1^2x_2x_3+x_1x_2^2x_3+x_1x_2x_3^2)+6(x_1^2x_2^2+x_2^2x_3^2)=-12<0.$$

\textbf{Subcase 2.2} When only one of $\{t_{1123},t_{1233}\}$ is $1$, without loss the generality, suppose $t_{1233}=1$. Let $x=(-4,5,2)^{\top}$, then
$${\mathcal{T}}x^4\leq x_1^4+x_2^4+x_3^4+4(x_1x_2^3+x_2^3x_3)+12(x_1x_2^2x_3+x_1x_2x_3^2)+6(x_1^2x_2^2+x_1^2x_3^2+x_2^2x_3^2)=-79<0.$$

\textbf{Subcase 2.3} When $t_{1123}=t_{1233}=-1$. Let $x=(2,5,4)^{\top}$, then
$${\mathcal{T}}x^4\leq x_1^4+x_2^4+x_3^4+4(x_1x_2^3+x_2^3x_3)-12(x_1^2x_2x_3-x_1x_2^2x_3+x_1x_2x_3^2)+6(x_1^2x_2^2+x_1^2x_3^2+x_2^2x_3^2)=-79<0.$$

\textbf{Subcase 2.4} When one of $\{t_{1123},t_{1233}\}$ is $0$ and the other one is $-1$. Without loss the generality, suppose $t_{1123}=-1$ and $t_{1233}=0$. Let $x=(-2,1,1)^{\top}$, then
$${\mathcal{T}}x^4\leq x_1^4+x_2^4+x_3^4+4(x_1x_2^3+x_2^3x_3)-12(x_1^2x_2x_3-x_1x_2^2x_3)+6(x_1^2x_2^2+x_1^2x_3^2+x_2^2x_3^2)=-4<0.$$
The necessity is proved.
\end{proof}

\begin{theorem}\label{th3.10}
Let ${\mathcal{T}}=(t_{ijkl})\in\widehat{{\mathcal{E}}}_{4,3}$ and $t_{iiij}t_{ijjj}=t_{jjjk}t_{jkkk}=t_{iiik}t_{ikkk}=t_{jjjk}+t_{jkkk}=t_{iiik}+t_{ikkk}=0$, and $t_{iiij}+t_{ijjj}\neq0$ for $i,j,k\in\{1,2,3\}$, $i\neq j$, $i\neq k$, $j\neq k$ (i.e., only one of $\{t_{iiij},t_{ijjj},t_{jjjk},t_{jkkk},t_{iiik},t_{ikkk}\}$ is not $0$). Then $\mathcal{T}$ is positive definite with $t_{iiij}=\pm1$, if and only if one of the following conditions is satisfied.
\begin{itemize}
\item [(a)] $t_{1123}=t_{1223}=t_{1233}=0$, and $t_{iijj}=1$, $t_{iikk},t_{jjkk}\in\{0,1\}$.
\item [(b)] $t_{ijkk}=0$, $t_{iijk}t_{ijjk}t_{iiij}=1$, and $t_{1122}=t_{2233}=t_{1133}=1$.
\item [(c)] $t_{ijkk}=\pm1$, $t_{iijk}=t_{ijjk}=0$ and $t_{1122}=t_{2233}=t_{1133}=1$.
\item [(d)] $t_{iijk}=t_{ijkk}=0$, $t_{ijjk}=\pm1$ and $t_{1122}=t_{2233}=t_{1133}=1$.
\end{itemize}
\end{theorem}
\begin{proof}
Let $t_{iiij}\neq0$ and $t_{iiik}=t_{ijjj}=t_{ikkk}=t_{jjjk}=t_{jkkk}=0$ for $i,j,k\in\{1,2,3\}$, $i\neq j$, $i\neq k$, $j\neq k$.

``\textbf{if (Sufficiency)}."
(a) For $x=(x_1,x_2,x_3)^{\top}\in{\mathbb{R}}^3$, we have
\begin{equation}\label{th3.10a}
{\mathcal{T}}x^4\geq x_1^4+x_2^4+x_3^4+4t_{iiij}x_i^3x_j+6x_i^2x_j^2=(x_i^2+2t_{iiij}x_ix_j)^2+x_j^4+x_k^4+2x_i^2x_j^2\geq0.
\end{equation}

(b) If $t_{ijkk}=0$, $t_{iijk}t_{ijjk}t_{iiij}=1$, and $t_{1122}=t_{2233}=t_{1133}=1$. When $t_{iijk}=t_{ijjk}=t_{iiij}=1$.
Without loss the generality, we might take $t_{1112}=1$. Then when $t_{1233}=0$, and $t_{1123}=t_{1223}=1$, for $x=(x_1,x_2,x_3)^{\top}\in{\mathbb{R}}^3$
\begin{eqnarray}\label{x3-3}
{\mathcal{T}}x^4=x_1^4+x_2^4+x_3^4+4x_1^3x_2+12(x_1^2x_2x_3+x_1x_2^2x_3)+6(x_1^2x_2^2+x_2^2x_3^2+x_1^2x_3^2).
\end{eqnarray}

By Lemma \ref{lem:23}, when $x_1=0$, or $x_2=0$, or $x_3=0$, ${\mathcal{T}}x^4\geq0$, with equality if and only if $x_1=x_2=x_3=0$.

According the method in \cite{nka}, rewrite (\ref{x3-3}) as
\begin{equation} \label{x3-33}
{\mathcal{T}}x^4=x_3^4+6(x_1^2+x_2^2)x_3^2+12(x_1^2x_2+x_1x_2^2)x_3+x_1^4+x_2^4+4x_1x_2^3+6x_1^2x_2^2.
\end{equation}
Then the inner determinants corresponding to (\ref{x3-33}) are
$$\triangle_1^1=4,\ \ ~~~~~~~~\triangle_3^1=-48(x_1^2+x_2^2),$$
$$\triangle_5^1=-192\times[4x_1^6+(2x_1^3-x_1^2x_2)^2+21x_1^4x_2^2+25(x_1^2x_2+x_1x_2^2)^2+22x_1^2x_2^4+8x_2^6],$$
\begin{align*}
\triangle_7^1&=256\times(64x_1^{12}+192x_1^{11}x_2+336x_1^{10}x_2^2+2296x_1^9x_2^3+4845x_1^8x_2^4-588x_1^7x_2^5-3570x_1^6x_2^6\\
&~~~-2628x_1^5x_2^7 +1245x_1^4x_2^8+192x_1^3x_2^9+576x_1^2x_2^{10}+64x_2^{12}).
\end{align*}

If $x_2\neq0$,
\begin{align*}
\triangle_7^1&=256x_2^{12}\times(64y^{12}+192y^{11}+336y^{10}+2296y^9+4845y^8-588y^7-3570y^6-2628y^5\\
&~~~+1245y^4+192y^3+576y^2+64),
\end{align*}
where $y=\frac{x_1}{x_2}$. Using Lemma \ref{lem:26} for $\frac{\triangle_7^1}{256x_2^{12}}$, and we get the number of distinct real roots of $\frac{\triangle_7^1}{256x_3^{12}}=0$ with each $x_2\neq0$ is $0$.
Since $\frac{\triangle_7^1}{256x_2^{12}}=37>0$ for $y=0$, $\frac{\triangle_7^1}{256x_2^{12}}>0$ for any $y\in{\mathbb{R}}$, thus $\triangle_7^1>0$ for any $x=(x_1,x_2)^{\top}\in{\mathbb{R}}^2$ and $x_2\neq0$.

Then if $x_1,x_2\neq0$, the number of distinct real roots $N$ of ${\mathcal{T}}x^4=0$ is
$$N=\mbox{var}[+,-,-,+,+]-\mbox{var}[+,+,-,-,+]=2-2=0.$$
Thus, ${\mathcal{T}}x^4=0$ if and only if $x=(0,0,0)^{\top}$. By Lemma \ref{lem:21}, ${\mathcal{T}}$ is positive definite.  

Similarly, we also build the  positive definiteness of ${\mathcal{T}}$   when $-t_{iijk}=-t_{ijjk}=t_{iiij}=1$, $t_{iijk}=-t_{ijjk}=-t_{iiij}=1$ and $-t_{iijk}=t_{ijjk}=-t_{iiij}=1$.

(c) If $t_{ijkk}=\pm1$, $t_{iijk}=t_{ijjk}=0$ and $t_{1122}=t_{2233}=t_{1133}=1$, then for $x=(x_1,x_2,x_3)^{\top}\in{\mathbb{R}}^3$,
\begin{equation}\label{th3.10c}
{\mathcal{T}}x^4=(x_i^2+2t_{iiij}x_ix_j)^2+x_j^4+x_k^4+6(x_ix_k+t_{ijkk}x_jx_k)^2+2x_i^2x_j^2\geq0.
\end{equation}

(d) If $t_{iijk}=t_{ijkk}=0$, $t_{ijjk}=\pm1$ and $t_{1122}=t_{2233}=t_{1133}=1$.Then for $x=(x_1,x_2,x_3)^{\top}\in{\mathbb{R}}^3$,
\begin{align}\label{th3.10d}
{\mathcal{T}}x^4&=(x_i^2+2t_{iiij}x_ix_j+2t_{iiij}t_{ijjk}x_jx_k)^2+(x_j^2+2t_{ijjk}x_ix_k)^2+x_k^4\nonumber\\
&~~~+2(x_ix_j-t_{iiij}t_{ijjk}x_jx_k)^2+2x_i^2x_k^2\geq0.
\end{align}

Furthermore, in (\ref{th3.10a}), (\ref{th3.10c}) and (\ref{th3.10d}), it is easily verified that ${\mathcal{T}}x^4=0$ if and only if $x=(0,0,0)^{\top}$, and then, $\mathcal{T}$ is positive definite by Lemma \ref{lem:21}.

``\textbf{only if (Necessity)}."  Similar to the prove of Theorem \ref{th3.3}, we only need to consider $t_{iiij}=1$. And it follows from the positive definiteness of $\mathcal{T}$ with Eq. \eqref{eq:3} that \begin{center}
	$t_{iijj}=1$, $t_{iikk},t_{jjkk}\in\{0,1\}$.
\end{center}
Without loss the generality, we might take $t_{1112}=1$ and $t_{1113}=t_{1222}=t_{2223}=t_{2333}=t_{1333}=0$. Then for $x=(x_1,x_2,x_3)^{\top}\in{\mathbb{R}}^3$
\begin{align*}
{\mathcal{T}}x^4=&x_1^4+x_2^4+x_3^4+4x_1^3x_2+12(t_{1123}x_1^2x_2x_3+t_{1223}x_1x_2^2x_3+t_{1233}x_1x_2x_3^2)\\
&+6(x_1^2x_2^2+t_{1133}x_1^2x_3^2+t_{2233}x_2^2x_3^2).
\end{align*}
For $x=(1,1,1)^\top$, we have\begin{align*}
	{\mathcal{T}}x^4=&13+12(t_{1123}+t_{1223}+t_{1233})+6(t_{1133}+t_{2233})>0,
\end{align*}
So, the following cases  could not occur,
\begin{itemize}
	\item Two of $\{t_{1123},t_{1223},t_{1233}\}$ are $-1$ and the third one is $0$, and $t_{1133}+t_{2233}<2$;
	\item $t_{1123}=t_{1223}=t_{1233}=-1$.
\end{itemize}
Next we use a proof by contradiction to prove the necessity.

\textbf{Case 1.} $t_{1123}+1<t_{1223}+t_{1233}$. Let $x=(-1,1,1)^{\top}$, then
$${\mathcal{T}}x^4\leq x_1^4+x_2^4+x_3^4+4x_1^3x_2+12(x_1x_2^2x_3+x_1x_2x_3^2)+6(x_1^2x_2^2+x_1^2x_3^2+x_2^2x_3^2)=-7<0.$$
$t_{1223}+1<t_{1123}+t_{1233}$ is similarly.

\textbf{Case 2.} When $t_{1123}+1=t_{1223}+t_{1233}$ and $t_{1223}+1=t_{1123}+t_{1233}$, i.e., $t_{1233}=1$ and $t_{1123}=t_{1223}$.

\textbf{Subcase 2.1} When $t_{1123}=t_{1223}=1$. Let $x=(2,-1,1)^{\top}$, then
$${\mathcal{T}}x^4\leq x_1^4+x_2^4+x_3^4+4x_1^3x_2+12(x_1^2x_2x_3+x_1x_2^2x_3+x_1x_2x_3^2)+6(x_1^2x_2^2+x_1^2x_3^2+x_2^2x_3^2)=-8<0.$$

\textbf{Subcase 2.2} When $t_{1123}=t_{1223}=-1$. Let $x=(-2,1,1)^{\top}$, then
$${\mathcal{T}}x^4\leq x_1^4+x_2^4+x_3^4+4x_1^3x_2-12(x_1^2x_2x_3+x_1x_2^2x_3-x_1x_2x_3^2)+6(x_1^2x_2^2+x_1^2x_3^2+x_2^2x_3^2)=-8<0$$

\textbf{Subcase 2.3} When $t_{1123}=t_{1223}=0$, and $t_{2233}=0$ or $t_{1133}=0$. If $t_{2233}=0$, let $x=(1,-1,2)^{\top}$, then
$${\mathcal{T}}x^4\leq x_1^4+x_2^4+x_3^4+4x_1^3x_2+12x_1x_2x_3^2+6(x_1^2x_2^2+x_1^2x_3^2)=-4<0,$$
$t_{1133}=0$ is same.

\textbf{Case 3.} $t_{1123}+1=t_{1223}+t_{1233}$ and $t_{1223}+1>t_{1123}+t_{1233}$.

\textbf{Subcase 3.1} When $t_{1123}=t_{1233}=0$, $t_{1223}=1$, and $t_{2233}=0$ or $t_{1133}=0$. If $t_{2233}=0$, let $x=(-1,1,1)^{\top}$, then
$${\mathcal{T}}x^4\leq x_1^4+x_2^4+x_3^4+4x_1^3x_2+12x_1x_2^2x_3+6(x_1^2x_2^2+x_1^2x_3^2)=-1<0.$$
$t_{1133}=0$ is same.

\textbf{Subcase 3.2} When $t_{1123}=-1$ and $t_{1223}+t_{1233}=0$. Let $x=(-2,1,1)^{\top}$, then
$${\mathcal{T}}x^4\leq x_1^4+x_2^4+x_3^4+4x_1^3x_2-12x_1^2x_2x_3+6(x_1^2x_2^2+x_1^2x_3^2+x_2^2x_3^2)=-8<0.$$

\textbf{Case 4.} $t_{1123}+1>t_{1223}+t_{1233}$ and $t_{1223}+1=t_{1123}+t_{1233}$.

\textbf{Subcase 4.1} When $t_{1223}=t_{1233}=0$ and $t_{1123}=1$. Let $x=(8,-3,1)^{\top}$, then
$${\mathcal{T}}x^4\leq x_1^4+x_2^4+x_3^4+4x_1^3x_2+12x_1^2x_2x_3+6(x_1^2x_2^2+x_1^2x_3^2+x_2^2x_3^2)=-376<0.$$

\textbf{Subcase 4.2} When $t_{1223}=-1$, $t_{1123}=t_{1233}=0$, and $t_{2233}=0$ or $t_{1133}=0$. When $t_{2233}=0$, let $x=(1,-1,1)^{\top}$, then
$${\mathcal{T}}x^4\leq x_1^4+x_2^4+x_3^4+4x_1^3x_2-12x_1x_2^2x_3+6(x_1^2x_2^2+x_1^2x_3^2)=-1<0.$$
$t_{1133}=0$ is same.

\textbf{Subcase 4.3} When $t_{1223}=t_{1233}=-1$ and $t_{1123}=1$. Let $x=(2,-1,1)^{\top}$, then
$${\mathcal{T}}x^4\leq x_1^4+x_2^4+x_3^4+4x_1^3x_2+12(x_1^2x_2x_3-x_1x_2^2x_3-x_1x_2x_3^2)+6(x_1^2x_2^2+x_1^2x_3^2+x_2^2x_3^2)=-8<0.$$

\textbf{Case 5.} $t_{1123}+1>t_{1223}+t_{1233}$ and $t_{1223}+1>t_{1123}+t_{1233}$, i.e., $t_{1233}=0$ and $t_{1123}=t_{1223}$, or $t_{1233}=-1$ and $t_{1123}-1\leq t_{1223}\leq t_{1123}+1$.

\textbf{Subcase 5.1} When $t_{1233}=0$, $t_{1123}=t_{1223}=1$, and $t_{2233}=0$ or $t_{1133}=0$ now. When $t_{2233}=0$, let $x=(0,-1,1)^{\top}$, then
$${\mathcal{T}}x^4\leq x_1^4+x_2^4+x_3^4+4x_1^3x_2+12(x_1^2x_2x_3+x_1x_2^2x_3)+6(x_1^2x_2^2+x_1^2x_3^2)=-4<0.$$
$t_{1133}=0$ is same.

\textbf{Subcase 5.2} When $-t_{1233}=t_{1223}=1$ and $t_{1123}\in\{0,1\}$. Let $x=(1,2,-2)^{\top}$, then
$${\mathcal{T}}x^4\leq  x_1^4+x_2^4+x_3^4+4x_1^3x_2+12(x_1x_2^2x_3-x_1x_2x_3^2)+6(x_1^2x_2^2+x_1^2x_3^2+x_2^2x_3^2)=-7<0.$$

\textbf{Subcase 5.3} When $t_{1233}=t_{1223}=-1$ and $t_{1123}=0$. Let $x=(2,3,3)^{\top}$, then
$${\mathcal{T}}x^4\leq x_1^4+x_2^4+x_3^4+4x_1^3x_2-12(x_1x_2^2x_3+x_1x_2x_3^2)+6(x_1^2x_2^2+x_1^2x_3^2+x_2^2x_3^2)=-104<0.$$

\textbf{Subcase 5.4} When $t_{1233}=t_{1123}=-1$ and $t_{1223}=0$. Let $x=(-4,2,1)^{\top}$, then
$${\mathcal{T}}x^4\leq x_1^4+x_2^4+x_3^4+4x_1^3x_2-12(x_1^2x_2x_3+x_1x_2x_3^2)+6(x_1^2x_2^2+x_1^2x_3^2+x_2^2x_3^2)=-23<0.$$

\textbf{Subcase 5.5} When $-t_{1233}=t_{1123}=-1$ and $t_{1223}=0$. Let $x=(4,-2,1)^{\top}$, then
$${\mathcal{T}}x^4\leq  x_1^4+x_2^4+x_3^4+4x_1^3x_2+12(x_1^2x_2x_3-x_1x_2x_3^2)+6(x_1^2x_2^2+x_1^2x_3^2+x_2^2x_3^2)=-23<0.$$

\textbf{Subcase 5.6} When $t_{1233}=-1$, $t_{1123}=t_{1223}=0$, and $t_{2233}=0$ or $t_{1133}=0$. Let $x=(1,2,2)^{\top}$, then
$${\mathcal{T}}x^4\leq x_1^4+x_2^4+x_3^4+4x_1^3x_2-12x_1x_2x_3^2+6(x_1^2x_2^2+x_1^2x_3^2)=-7<0$$
with $t_{2233}=0$. Let $x=(3,2,5)^{\top}$, then
$${\mathcal{T}}x^4\leq x_1^4+x_2^4+x_3^4+4x_1^3x_2-12x_1x_2x_3^2+6(x_1^2x_2^2+x_2^2x_3^2)=-46<0$$
with $t_{1133}=0$. The necessity is proved.
\end{proof}

From Theorem \ref{th3.1}, Theorem \ref{th3.8}, Theorem \ref{th3.9}, and Theorem \ref{th3.10}, we have the following corollary.
\begin{corollary}\label{cor2}
Let $t_{iiii}=1$ and $t_{iiij}t_{ijjj}=0$, for all $i,j\in\{1,2,3\}$, $i\neq j$. Then $\mathcal{T}$ is positive definite if and only if conditions in Theorem \ref{th3.1}, or Theorem \ref{th3.8}, or Theorem \ref{th3.9}, or Theorem \ref{th3.10} are satisfied.
\end{corollary}

\begin{theorem}\label{th3.12}
Let ${\mathcal{T}}=(t_{ijkl})\in\widehat{{\mathcal{E}}}_{4,3}$ and $t_{iiij}t_{ijjj}=0$, $t_{jjjk}t_{jkkk}=t_{iiik}t_{ikkk}=-1$ for $i,j,k\in\{1,2,3\}$, $i\neq j$, $i\neq k$, $j\neq k$. Then $\mathcal{T}$ is positive semi-definite if and only if one of the following conditions is satisfied.
\begin{itemize}
  \item [(a)] $t_{iijk}=t_{ijjk}=0$, $t_{jjjk}t_{iiik}t_{ijkk}=1$, $t_{iijj}\in\{0,1\}$, and $t_{jjkk}=t_{iikk}=1$, for $i,j,k\in\{1,2,3\}$, $i\neq j$, $i\neq k$, $j\neq k$, if $t_{iiij}=t_{ijjj}=0$.
  \item [(b)] $t_{iijk}t_{iiij}t_{iiik}=1$, $t_{ijjk}=t_{ijkk}=0$ or $t_{iijk}t_{ijjk}t_{ijkk}=1$ with $t_{iiij}t_{ijkk}=1$, and $t_{1122}=t_{2233}=t_{1133}=1$ for $i,j,k\in\{1,2,3\}$, $i\neq j$, $i\neq k$, $j\neq k$, if $t_{ijjj}=0$ and $t_{iiij}t_{jkkk}t_{ikkk}=1$.
\end{itemize}
\end{theorem}
\begin{proof}
``\textbf{if (Sufficiency)}."
(a) If $t_{iijj}=1$, then for $x=(x_1,x_2,x_3)^{\top}\in{\mathbb{R}}^3$,
\begin{align*}
{\mathcal{T}}x^4&=(x_i^2+x_j^2-x_k^2+2t_{iiik}x_ix_k+2t_{jjjk}x_jx_k)^2+(2x_ix_j-t_{iiik}x_jx_k-t_{jjjk}x_ix_k)^2\\\
&~~~+(x_ix_k+t_{ijkk}x_jx_k)^2+2(x_i^2x_k^2+x_j^2x_k^2)\geq0,
\end{align*}
with equality if and only if $x=(0,0,0)^{\top}$.

Assume $t_{iijj}=0$. Without loss the generality, we might take $t_{1112}=t_{1222}=0$, then $t_{1123}=t_{1223}=t_{1122}=0$, $t_{2233}=t_{1133}=t_{1113}t_{2223}t_{1233}=1$. Then for $x=(x_1,x_2,x_3)^{\top}\in{\mathbb{R}}^3$,
$${\mathcal{T}}x^4=x_1^4+x_2^4+x_3^4+4t_{1113}(x_1^3x_3-x_1x_3^3+x_2^3x_3-x_2x_3^3)+12x_1x_2x_3^2+6(x_2^2x_3^2+x_1^2x_3^2)$$
with $t_{1113}t_{2223}=t_{1233}=1$, and
\begin{align*}
{\mathcal{T}}x^4&=x_1^4+(-x_2)^4+x_3^4+4t_{1113}[x_1^3x_3-x_1x_3^3+(-x_2)^3x_3-(-x_2)x_3^3]+12x_1(-x_2)x_3^2\\
&~~~+6[(-x_2)^2x_3^2+x_1^2x_3^2]
\end{align*}
with $t_{1113}t_{2223}=t_{1233}=-1$. So we only need to consider one of these two cases. Suppose $-t_{1113}=t_{2223}=1$. Then for $x=(x_1,x_2,x_3)^{\top}\in{\mathbb{R}}^3$,
$${\mathcal{T}}x^4=x_1^4+x_2^4+x_3^4-4(x_1^3x_3-x_1x_3^3-x_2^3x_3+x_2x_3^3)-12x_1x_2x_3^2+6(x_2^2x_3^2+x_1^2x_3^2).$$
By Lemma \ref{lem:23}, when $x_2=0$, ${\mathcal{T}}x^4\geq0$, with equality if and only if $x_1=x_2=x_3=0$.
When $x_2\neq0$, let $y_1:=\frac{x_1}{x_2}$ and $y_3:=\frac{x_3}{x_2}$, then
$${\mathcal{T}}x^4=x_2^4[y_1^4+1+y_3^4-4(y_1^3y_3-y_1y_3^3-y_3+y_3^3)-12y_1y_3^2+6(y_3^2+y_1^2y_3^2)]:=x_2^4g(y_1,y_3).$$
Thus, to prove ${\mathcal{T}}x^4\geq0$ for any $x=(x_1,x_2,x_3)^{\top}\in{\mathbb{R}}^3$, we just need to prove $g(y_1,y_3)\geq0$ for any $y=(y_1,y_3)^{\top}\in{\mathbb{R}}^2$. Let $\triangle g(y_1,y_3)=0$, i.e.,
$$\left\{
\begin{array}{llll}
4(y_1^3+y_3^3-3y_1^2y_3-3y_3+3y_1y_3^2)=0, \\
4(y_3^3+3y_1y_3^2-y_1^3+1-3y_3^2-6y_1y_3+3y_1^2y^3+3y_1)=0.
\end{array}
\right.$$
$\left\{
\begin{array}{llll}
y_1=2 \\
y_3=1
\end{array}
\right.$ and
$\left\{
\begin{array}{llll}
y_1=\frac{1}{2}\\
y_3=-\frac{1}{2}
\end{array}
\right.$
are solutions of $\triangle g(y_1,y_3)=0$, and the Hessian matrixes of $g$ are positive definite at theses two points. Therefore, $(2,1)^{\top}$ and $(\frac{1}{2},-\frac{1}{2})^{\top}$ are local minimum points of $g$ with $g(2,1)=g(\frac{1}{2},-\frac{1}{2})=0$, and so, they are global also. Then $g(y_1,y_3)\geq0$ for any $y=(y_1,y_3)^{\top}\in{\mathbb{R}}^2$, which implies ${\mathcal{T}}$ is positive semi-definite.

(b) Suppose $t_{1222}=0$, $t_{1112}=t_{2333}t_{1333}=1$. Then $t_{1113}t_{1123}=1$, $t_{1223}=t_{1233}=0$ or $t_{1233}=t_{1123}t_{1223}=1$. Firstly, let $t_{1123}=t_{2223}=t_{1113}=-t_{2333}=-t_{1333}=1$ and $t_{1223}=t_{1233}=0$. Then for $x=(x_1,x_2,x_3)^{\top}\in{\mathbb{R}}^3$,
\begin{equation}\label{x2}
{\mathcal{T}}x^4=x_1^4+x_2^4+x_3^4+4(x_1^3x_2+x_2^3x_3+x_1^3x_3-x_2x_3^3-x_1x_3^3)+12x_1^2x_2x_3+6(x_1^2x_2^2+x_2^2x_3^2+x_1^2x_3^2).
\end{equation}
By Lemma \ref{lem:23}, when $x_1=0$, or $x_2=0$, or $x_3=0$, ${\mathcal{T}}x^4\geq0$. For $x_1\ne0$,  $x_2\ne0$,  $x_3\ne0$, rewritting (\ref{x2}) as
\begin{equation} \label{x2-1}
{\mathcal{T}}x^4=x_2^4+4x_3x_2^3+6(x_1^2+x_3^2)x_2^2+4(3x_1^2x_3+x_1^3-x_3^3)x_2+x_1^4+x_3^4+6x_1^2x_3^2+4(x_1^3x_3-x_1x_3^3).
\end{equation}
Then the inner determinants corresponding to (\ref{x2-1}) are
$$\triangle_1^1=4,\ \ ~~~~~\triangle_3^1=-48x_1^2,~~~~~\triangle_5^1=-192\times(11x_1^6-8x_1^3x_3^3-8x_1^2x_3^4+12x_3^6),$$
\begin{align*}
\triangle_7^1&=256\times(27x_1^{12}-192x_1^9x_3^3-816x_1^8x_3^4-24x_1^6x_3^6-768x_1^5x_3^7 \\ &~~~-960x_1^4x_3^8-64x_1^3x_3^9+2112x_1^2x_3^{10}-768x_1^1x_3^{11}+80x_3^{12}).
\end{align*}
If $x_1\neq0$,
$$\triangle_5^1=-192x_1^6\times(11-8\overline{x}_3^3-8\overline{x}_3^4+12\overline{x}_3^6),$$
$$\triangle_7^1=256x_1^{12}(27-192\overline{x}_3^3-816\overline{x}_3^4-24\overline{x}_3^6-768\overline{x}_3^7-960\overline{x}_3^8-64\overline{x}_3^9+2112\overline{x}_3^{10}-768\overline{x}_3^{11}+80\overline{x}_3^{12}),$$
where $\overline{x}_3=\frac{x_3}{x_1}$. Using Lemma \ref{lem:26} for $\frac{\triangle_5^1}{-192x_1^{6}}$ and $\frac{\triangle_7^1}{256x_1^{12}}$, then we get the numbers of distinct real roots of $\frac{\triangle_5^1}{-192x_1^{6}}=0$ and $\frac{\triangle_7^1}{256x_1^{12}}=0$ with each $x_1\neq0$ are $0$.

Since $\frac{\triangle_5^1}{-192x_1^{6}}=11>0$ and $\frac{\triangle_7^1}{256x_1^{12}}=27>0$ for $\overline{x}_3=0$, $\frac{\triangle_5^1}{-192x_1^{6}}>0$ and $\frac{\triangle_7^1}{256x_1^{12}}>0$ for any $y\in{\mathbb{R}}$, thus $\triangle_5^1<0$ and $\triangle_7^1>0$ for any$ x=(x_1,x_3)^{\top}\in{\mathbb{R}}^2$ and $x_1\neq0$.

Then if $x_1,x_2\neq0$, the number of distinct real roots $N$ of ${\mathcal{T}}x^4=0$ is
$$N=\mbox{var}[+,-,-,+,+]-\mbox{var}[+,+,-,-,+]=2-2=0.$$
Thus, ${\mathcal{T}}x^4=0$ if and only if $x=(0,0,0)^{\top}$. By Lemma \ref{lem:21}, ${\mathcal{T}}$ is positive definite.

Next, let $t_{1223}=t_{1233}=t_{1123}=t_{2223}=t_{1113}=-t_{2333}=-t_{1333}=1$, then for $x=(x_1,x_2,x_3)^{\top}\in{\mathbb{R}}^3$,
\begin{align}\label{111}
{\mathcal{T}}x^4=&x_1^4+x_2^4+x_3^4+4(x_1^3x_2+x_2^3x_3+x_1^3x_3-x_2x_3^3-x_1x_3^3)+12x_1^2x_2x_3+6(x_1^2x_2^2+x_2^2x_3^2+x_1^2x_3^2)\nonumber\\
&+12x_1x_2x_3(x_2+x_3)\nonumber\\
=&(x_1^2+x_2^2-x_3^2+2x_1x_2+2x_2x_3+2x_1x_3)^2+4(x_2x_3+x_1x_3)^2-4x_1x_2^3\nonumber\\
=&(x_1^2+x_2^2-x_3^2+2x_2x_3+2x_1x_3)^2+(x_1x_3-x_1x_2-2x_2x_3)^2+(x_1x_3-x_1x_2)^2+2x_1^2x_2^2\nonumber\\
&+2x_1^2x_3^2+4x_1^3x_2+12x_1^2x_2x_3+12x_1x_2^2x_3.
\end{align}
Then, when $x_1x_2x_3\neq0$, from the second and third equation of (\ref{111}), we can infer that, ${\mathcal{T}}x^4>0$ if $x_1x_2<0$, ${\mathcal{T}}x^4>0$ if $x_1x_2>0$ and $x_1x_3>0$, respectively. If $x_1x_2>0$ and $x_1x_3<0$, from the first equation of (\ref{111}) and previous discussion, ${\mathcal{T}}x^4>0$ if $|x_2|\leq|x_3|$ with $x_1>0$ or $|x_2|\geq|x_3|$ with $x_1<0$. In fact, ${\mathcal{T}}x^4={\mathcal{T}}(-x)^4$, therefor, ${\mathcal{T}}x^4>0$ if $|x_2|\leq|x_3|$ with $x_1<0$ or $|x_2|\geq|x_3|$ with $x_1>0$.

By Lemma \ref{lem:23}, when $x_1=0$, or $x_2=0$, or $x_3=0$, ${\mathcal{T}}x^4\geq0$, with equality if and only if $x_1=x_2=x_3=0$.
Therefore, ${\mathcal{T}}x^4\geq0$ for any $x=(x_1,x_2,x_3)^{\top}\in{\mathbb{R}}^3$, and ${\mathcal{T}}x^4=0$ if and only if $x=(0,0,0)^{\top}$.

All other cases satisfying condition (b) can be transformed into similar forms of $(\ref{x2})$ or $(\ref{111})$. 

``\textbf{only if (Necessity)}." The conditions,
$t_{iiij}t_{ijjj}=0$, $t_{jjjk}t_{jkkk}=t_{iiik}t_{ikkk}=-1$ for $i,j,k\in\{1,2,3\}$, $i\neq j$, $i\neq k$, $j\neq k$, can be divided into three cases, 
\begin{itemize}
    \item [($\romannumeral1$)] $t_{iiij}=t_{ijjj}=0$;
    \item [($\romannumeral2$)] $t_{iiij}t_{jjjk}t_{ikkk}=-1$ and $t_{ijjj}=0$;
	\item [($\romannumeral3$)] $t_{iiij}t_{jjjk}t_{ikkk}=1$ and $t_{ijjj}=0$;
\end{itemize}
Then we only need to consider $t_{jjjk}=t_{ikkk}=1$, $t_{iiij}=t_{jjjk}=-t_{ikkk}=1$, and $t_{iiij}=t_{jjjk}=t_{ikkk}=1$, respectively.

Without loss the generality, suppose $t_{1112}t_{1222}=0$, $t_{2223}t_{2333}=t_{1113}t_{1333}=-1$. And it follows from the positive semi-definiteness of $\mathcal{T}$ with Eq. \eqref{eq:2} that \begin{center}
	$t_{1133}=t_{2233}=1$ and $t_{1122}\in\{0,1\}$, for ($\romannumeral1$);\\
$t_{iijj}=1$ for all $i,j\in\{1,2,3\}$ and $i\neq j$, for ($\romannumeral2$) and ($\romannumeral3$).
\end{center}

($\romannumeral1$) We might take $t_{1112}=t_{1222}=0$, $-t_{1113}=t_{1333}=t_{2223}=-t_{2333}=1$.
Then for $x=(x_1,x_2,x_3)^{\top}\in{\mathbb{R}}^3$
\begin{align*}
{\mathcal{T}}x^4&=(x_1+x_2+x_3)^4-4(x_1^3x_2+x_1x_2^3)-8(x_1^3x_3+x_2x_3^3)+12[(t_{1123}-1)x_1^2x_2x_3\\
&~~~+(t_{1223}-1)x_1x_2^2x_3+(t_{1233}-1)x_1x_2x_3^2]+6(t_{1122}-1)x_1^2x_2^2.
\end{align*}
For $x=(1,1,1)^\top$, we have\begin{align*}
	{\mathcal{T}}x^4=15+12(t_{1123}+t_{1223}+t_{1233})+6t_{1122}\geq0.
\end{align*}
So, $t_{1123}+t_{1223}+t_{1233}\leq-2$ could not occur. Then we only need to consider $t_{1123}+t_{1223}+t_{1233}>-2$.

\textbf{Case 1.} $t_{1123}=-1$.

\textbf{Subcase 1.1} When $t_{1223}+t_{1233}\leq 0$. Let $x=(3,1,1)^{\top}$, then
$${\mathcal{T}}x^4\leq (x_1+x_2+x_3)^4-4(x_1^3x_2+x_1x_2^3)-8(x_1^3x_3+x_2x_3^3)-24(x_1^2x_2x_3+x_1x_2^2x_3)=-7<0.$$

\textbf{Subcase 1.2} When $t_{1223}+t_{1233}>0$. Let $x=(-1,1,1)^{\top}$, then
$${\mathcal{T}}x^4\leq (x_1+x_2+x_3)^4-4(x_1^3x_2+x_1x_2^3)-8(x_1^3x_3+x_2x_3^3)-12(2x_1^2x_2x_3+x_1x_2^2x_3)=-3<0.$$

\textbf{Case 2.} $t_{1123}\neq-1$.

\textbf{Subcase 2.1} When $t_{1233}=0$ and $t_{1223}\geq t_{1123}$, or $t_{1233}=1$. Let $x=(-1,1,4)^{\top}$, then
$${\mathcal{T}}x^4\leq(x_1+x_2+x_3)^4-4(x_1^3x_2+x_1x_2^3)-8(x_1^3x_3+x_2x_3^3)-12x_1x_2x_3^2=-24<0.$$

\textbf{Subcase 2.2} When $t_{1233}=0$, $t_{1223}=0$ and $t_{1123}=1$. Let $x=(3,-1,1)^{\top}$, then
$${\mathcal{T}}x^4\leq (x_1+x_2+x_3)^4-4(x_1^3x_2+x_1x_2^3)-8(x_1^3x_3+x_2x_3^3)-12(x_1x_2^2x_3+x_1x_2x_3^2)=-7<0.$$

\textbf{Subcase 2.3} When $t_{1233}=0$ and $t_{1223}=-1$. Let $x=(2,-5,2)^{\top}$, then
$${\mathcal{T}}x^4\leq (x_1+x_2+x_3)^4-4(x_1^3x_2+x_1x_2^3)-8(x_1^3x_3+x_2x_3^3)-24x_1x_2^2x_3-12(x_1^2x_2x_3+x_1x_2x_3^2)=-87<0$$

\textbf{Subcase 2.4} When $t_{1233}=-1$ and $t_{1123}+t_{1223}\geq1$. Let $x=(1,1,-1)^{\top}$, then
$${\mathcal{T}}x^4\leq (x_1+x_2+x_3)^4-4(x_1^3x_2+x_1x_2^3)-8(x_1^3x_3+x_2x_3^3)-12(x_1^2x_2x_3+2x_1x_2x_3^2)=-3<0.$$

\textbf{Subcase 2.5} When $t_{1233}=-1$ and $t_{1123}=-t_{1223}=1$. Let $x=(3,-1,1)^{\top}$, then
$${\mathcal{T}}x^4\leq (x_1+x_2+x_3)^4-4(x_1^3x_2+x_1x_2^3)-8(x_1^3x_3+x_2x_3^3)-24(x_1x_2^2x_3+x_1x_2x_3^2)=-7<0$$

\textbf{Subcase 2.6} When $t_{1233}=-1$ and  $-t_{1123}=t_{1223}=1$. Let $x=(5,1,3)^{\top}$, then
$${\mathcal{T}}x^4\leq (x_1+x_2+x_3)^4-4(x_1^3x_2+x_1x_2^3)-8(x_1^3x_3+x_2x_3^3)-24(x_1^2x_2x_3+x_1x_2x_3^2)=-55<0$$

($\romannumeral2$) We might take $t_{1222}=0$ and $t_{1112}=t_{1113}=-t_{1333}=t_{2223}=-t_{2333}=t_{1122}=t_{2233}=t_{1133}=1$.
Then for $x=(x_1,x_2,x_3)^{\top}\in{\mathbb{R}}^3$
\begin{align*}
{\mathcal{T}}x^4=&(x_1+x_2+x_3)^4-4x_1x_2^3-8(x_1x_3^3+x_2x_3^3)\\
&+12[(t_{1123}-1)x_1^2x_2x_3+(t_{1223}-1)x_1x_2^2x_3+(t_{1233}-1)x_1x_2x_3^2].
\end{align*}
For $x=(1,1,4)^\top$, we have\begin{align*}
	{\mathcal{T}}x^4=-20+48(t_{1123}+t_{1223}+4t_{1233})\geq0.
\end{align*}
So, the following cases could not occur,
 \begin{itemize}
    \item $\{t_{1123},t_{1223},t_{1233}\}$ are not $1$;
    \item $t_{1233}\neq1$ and $t_{1123}=-1$ or $t_{1223}=-1$;
	\item $t_{1233}=-1$.
\end{itemize}
Next, we discuss other situations.

\textbf{Case 1.} $t_{1233}=1$.

\textbf{Subcase 1.1} When $t_{1123}\neq1$. Let $x=(-3,1,1)^{\top}$, then
$${\mathcal{T}}x^4\leq (x_1+x_2+x_3)^4-4x_1x_2^3-8(x_1x_3^3+x_2x_3^3)-12(x_1^2x_2x_3+2x_1x_2^2x_3)=-7<0.$$

\textbf{Subcase 1.2} When $t_{1123}=1$ and $t_{1223}\neq1$. Let $x=(1,-1,1)^{\top}$, then
$${\mathcal{T}}x^4\leq (x_1+x_2+x_3)^4-4x_1x_2^3-8(x_1x_3^3+x_2x_3^3)-12x_1x_2^2x_3=-7<0.$$

\textbf{Case 2.} $t_{1233}=0$, $\{t_{1123},t_{1223}\}$ are not $-1$, and $t_{1123}=1$ or $t_{1223}=1$

\textbf{Subcase 2.1} $t_{1123}=t_{1223}=1$. Let $x=(3,3,-2)^{\top}$, then
$${\mathcal{T}}x^4= (x_1+x_2+x_3)^4-4x_1x_2^3-8(x_1x_3^3+x_2x_3^3)-12x_1x_2x_3^2=-116<0.$$

\textbf{Subcase 2.2} When $t_{1223}=1$ and $t_{1123}=0$. Let $x=(-3,1,1)^{\top}$, then
$${\mathcal{T}}x^4\leq (x_1+x_2+x_3)^4-4x_1x_2^3-8(x_1x_3^3+x_2x_3^3)-12(x_1^2x_2x_3+x_1x_2x_3^2)=-43<0.$$

($\romannumeral3$) We might take $t_{1222}=0$ and $t_{1112}=-t_{1113}=t_{1333}=t_{2223}=-t_{2333}=1$. Then for $x=(x_1,x_2,x_3)^{\top}\in{\mathbb{R}}^3$
\begin{align*}
{\mathcal{T}}x^4=&(x_1+x_2+x_3)^4-4x_1x_2^3-8(x_1^3x_3+x_2x_3^3)\\
&+12[(t_{1123}-1)x_1^2x_2x_3+(t_{1223}-1)x_1x_2^2x_3+(t_{1233}-1)x_1x_2x_3^2].
\end{align*}
For $x=(-1,1,5)^\top$, we have\begin{align*}
	{\mathcal{T}}x^4=-31+60(t_{1123}-t_{1223}-5t_{1233})\geq0.
\end{align*}
So, the following cases could not occur,
 \begin{itemize}
    \item $t_{1233}=0$ and $t_{1123}\leq t_{1223}$;
	\item $t_{1233}=1$.
\end{itemize}
Next, we discuss other situations.

\textbf{Case 1.} $t_{1233}=0$ and $t_{1123}>t_{1223}$. Let $x=(7,-2,1)^{\top}$, then
$${\mathcal{T}}x^4\leq(x_1+x_2+x_3)^4-4x_1x_2^3-8(x_1^3x_3+x_2x_3^3)-24x_1x_2^2x_3-12(x_1^2x_2x_3+x_1x_2x_3^2)=-536<0.$$

\textbf{Case 2.} $t_{1233}=-1$.

\textbf{Subcase 2.1} When $t_{1223}=1$. Let $x=(1,5,-2)^{\top}$, then
$${\mathcal{T}}x^4\leq(x_1+x_2+x_3)^4-4x_1x_2^3-8(x_1^3x_3+x_2x_3^3)-24(x_1^2x_2x_3+x_1x_2x_3^2)=-148<0.$$

\textbf{Subcase 2.2} When $t_{1123}+t_{1223}\leq-1$, let $x=(2,2,3)^{\top}$, then
$${\mathcal{T}}x^4\leq(x_1+x_2+x_3)^4-4x_1x_2^3-8(x_1^3x_3+x_2x_3^3)-12x_1^2x_2x_3-24(x_1x_2^2x_3+x_1x_2x_3^2)=-15<0.$$

\textbf{Subcase 2.3} When $t_{1123}=-t_{1223}=1$. Let $x=(6,-5,1)^{\top}$, then
$${\mathcal{T}}x^4=(x_1+x_2+x_3)^4-4x_1x_2^3-8(x_1^3x_3+x_2x_3^3)-24(x_1x_2^2x_3+x_1x_2x_3^2)=-1552<0.$$

\textbf{Subcase 2.4} When $t_{1123}=t_{1223}=0$. Let $x=(7,-2,1)^{\top}$, then
$${\mathcal{T}}x^4=(x_1+x_2+x_3)^4-4x_1x_2^3-8(x_1^3x_3+x_2x_3^3)-24x_1x_2x_3^2-12(x_1^2x_2x_3+x_1x_2^2x_3)=-32<0.$$

\textbf{Subcase 2.5} When $t_{1123}=1$ and $t_{1223}=0$. Let $x=(4,4,-5)^{\top}$, then
$${\mathcal{T}}x^4=(x_1+x_2+x_3)^4-4x_1x_2^3-8(x_1^3x_3+x_2x_3^3)-12x_1x_2^2x_3-24x_1x_2x_3^2=-143<0.$$

The necessity is proved.
\end{proof}

\begin{corollary}\label{cor3}
Let ${\mathcal{T}}=(t_{ijkl})\in\widehat{{\mathcal{E}}}_{4,3}$ and $t_{iiij}t_{ijjj}=0$, $t_{jjjk}t_{jkkk}=t_{iiik}t_{ikkk}=-1$ for $i,j,k\in\{1,2,3\}$, $i\neq j$, $i\neq k$, $j\neq k$. Then $\mathcal{T}$ is positive definite if and only if one of the following conditions is satisfied.
\begin{itemize}
  \item [(a)] $t_{iijk}=t_{ijjk}=0$, $t_{jjjk}t_{iiik}t_{ijkk}=1$ $t_{iijj}=t_{jjkk}=t_{iikk}=1$, for $i,j,k\in\{1,2,3\}$, $i\neq j$, $i\neq k$, $j\neq k$, if $t_{iiij}=t_{ijjj}=0$.
  \item [(b)] $t_{iijk}t_{iiij}t_{iiik}=1$, $t_{ijjk}=t_{ijkk}=0$ or $t_{iijk}t_{ijjk}t_{ijkk}=1$ with $t_{iiij}t_{ijkk}=1$, and  $t_{iijj}=t_{jjkk}=t_{iikk}=1$ for $i,j,k\in\{1,2,3\}$, $i\neq j$, $i\neq k$, $j\neq k$, if $t_{ijjj}=0$ and $t_{iiij}=t_{jkkk}t_{ikkk}$.
\end{itemize}
\end{corollary}

\begin{theorem}\label{th3.13}
Let ${\mathcal{T}}=(t_{ijkl})\in\widehat{{\mathcal{E}}}_{4,3}$ and $t_{iiij}t_{ijjj}=t_{iiik}t_{ikkk}=0$, $t_{jjjk}t_{jkkk}=1$ for $i,j,k\in\{1,2,3\}$, $i\neq j$, $i\neq k$, $j\neq k$. Then $\mathcal{T}$ is positive semi-definite if and only if one of the following conditions is satisfied.
\begin{itemize}
  \item [(a)]  $t_{iiij}=t_{ijjj}=t_{iiik}=t_{ikkk}=0$, and
  \begin{itemize}
    \item [(a$_1$)] $t_{1123}=t_{1223}=t_{1233}=0$, $t_{jjkk}=1$ and $t_{iijj},t_{iikk}\in\{0,1\}$, or
    \item [(a$_2$)] $t_{iijk}=\pm1$, $t_{ijjk}=t_{ijkk}=0$, and $t_{1122}=t_{2233}=t_{1133}=1$, further, if $t_{iijk}t_{jjjk}=-1$, only need $t_{ijjk}=t_{ijkk}=0$, $t_{jjkk}=1$, and $t_{iijj}+t_{iikk}\geq1$.
  \end{itemize}
  \item [(b)] $t_{1123}=t_{1223}=t_{1233}=t_{ijjj}=t_{ikkk}=t_{iiik}=0$, $t_{iiij}=\pm1$, $t_{iijj}=t_{jjkk}=1$ and $t_{iikk}\in\{0,1\}$.
  \item [(c)] $t_{ijjj}=t_{ikkk}=t_{ijjk}=t_{ijkk}=0$, $t_{iijk}t_{iiij}t_{iiik}=t_{1122}=t_{2233}=t_{1133}=1$.
  \item [(d)] $t_{iiij}=t_{iiik}=0$, and $t_{ijjj}t_{ijkk}=t_{ikkk}t_{ijjk}=t_{ijjj}t_{ikkk}t_{iijk}=t_{ikkk}t_{jjjk}t_{ijkk}=t_{ijjj}t_{jjjk}t_{ijjk}=t_{1122}=t_{2233}=t_{1133}=1$.
\end{itemize}
\end{theorem}
\begin{proof}
``\textbf{if (Sufficiency)}." (a) Suppose $t_{iiij}=t_{ijjj}=t_{iiik}=t_{ikkk}=0$ for $i,j,k\in\{1,2,3\}$, $i\neq j$, $i\neq k$, $j\neq k$.

(a$_1$) If $t_{1123}=t_{1223}=t_{1233}=0$, $t_{jjkk}=1$ and $t_{iijj},t_{iikk}\in\{0,1\}$. Then for $x=(x_1,x_2,x_3)^{\top}\in{\mathbb{R}}^3$,
$${\mathcal{T}}x^4\geq x_i^4+(x_j+t_{jjjk}x_k)^4\geq0.$$

(a$_2$) If $t_{iijk}=\pm1$, $t_{ijjk}=t_{ijkk}=0$, and $t_{1122}=t_{2233}=t_{1133}=1$. Then for $x=(x_1,x_2,x_3)^{\top}\in{\mathbb{R}}^3$,
$${\mathcal{T}}x^4= x_i^4+(x_j+t_{jjjk}x_k)^4+6(x_ix_j+t_{iijk}x_ix_k)\geq0.$$
When $t_{iijk}t_{jjjk}=-1$, $t_{ijjk}=t_{ijkk}=0$, $t_{jjkk}=t_{iijj}=1$, and $t_{iikk}=0$, for $x=(x_1,x_2,x_3)^{\top}\in{\mathbb{R}}^3$,
$${\mathcal{T}}x^4= (-x_i^2+x_j^2+x_k^2+2t_{jjjk}x_jx_k)^2+2(x_ix_k+2t_{iijk}x_ix_j)^2\geq0,$$
when $t_{iijk}t_{jjjk}=-1$, $t_{ijjk}=t_{ijkk}=0$, $t_{jjkk}=t_{iikk}=1$, and $t_{iijj}=0$, for $x=(x_1,x_2,x_3)^{\top}\in{\mathbb{R}}^3$,
$${\mathcal{T}}x^4=(-x_i^2+x_j^2+x_k^2+2t_{jjjk}x_jx_k)^2+2(x_ix_j+2t_{iijk}x_ix_k)^2\geq0.$$

(b) Suppose $t_{1123}=t_{1223}=t_{1233}=0$, $t_{iijj}=t_{jjkk}=1$ and $t_{iikk}\in\{0,1\}$ with $t_{ijjj}=t_{ikkk}=t_{iiik}=0$, $t_{iiij}=\pm1$ for $i,j,k\in\{1,2,3\}$, $i\neq j$, $i\neq k$, $j\neq k$. Then for $x=(x_1,x_2,x_3)^{\top}\in{\mathbb{R}}^3$,
$${\mathcal{T}}x^4\geq (x_i^2+2t_{iiij}x_ix_j)^2+(x_j+t_{jjjk}x_k)^4+2x_i^2x_j^2\geq0.$$

(c) Suppose $t_{ijjj}=t_{ikkk}=0$, $t_{iijk}t_{iiij}t_{iiik}=1$, $t_{ijjk}=t_{ijkk}=0$ and $t_{1122}=t_{2233}=t_{1133}=1$ for $i,j,k\in\{1,2,3\}$, $i\neq j$, $i\neq k$, $j\neq k$. Then for $x=(x_1,x_2,x_3)^{\top}\in{\mathbb{R}}^3$,
$${\mathcal{T}}x^4=(x_i^2+2t_{iiij}x_ix_j+2t_{iiik}x_ix_k)^2+(x_j+t_{jjjk}x_k)^4+2(x_ix_j+2t_{iijk}x_ix_k)^2\geq0.$$

(d) Suppose $t_{iiij}=t_{iiik}=0$, $t_{ijjj}t_{ijkk}=t_{ikkk}t_{ijjk}=t_{ijjj}t_{ikkk}t_{iijk}=t_{ikkk}t_{jjjk}t_{ijkk}=t_{ijjj}t_{jjjk}t_{ijjk}=t_{1122}=t_{2233}=t_{1133}=1$ for $i,j,k\in\{1,2,3\}$, $i\neq j$, $i\neq k$, $j\neq k$. Then for $x=(x_1,x_2,x_3)^{\top}\in{\mathbb{R}}^3$,
$${\mathcal{T}}x^4=x_i^4+(x_j^2+x_k^2+2t_{ijjj}x_ix_j+2t_{ikkk}x_ix_k+2t_{jjjk}x_jx_k)^2+2(x_ix_j+t_{iijk}x_ix_k)^2\geq0.$$

``\textbf{only if (Necessity)}."  The conditions, 
$t_{iiij}t_{ijjj}=t_{iiik}t_{ikkk}=0$, $t_{jjjk}t_{jkkk}=1$ for $i,j,k\in\{1,2,3\}$, $i\neq j$, $i\neq k$, $j\neq k$, can be divided into nine cases.
\begin{flushleft}
  \begin{itemize}
    \item [($\romannumeral1$)] $t_{iiij}=t_{ijjj}=t_{iiik}=t_{ikkk}=0$;
    \item [($\romannumeral2$)] $t_{ijjj}=t_{ikkk}=t_{iiik}=0$, and $t_{iiij}=\pm1$;
    \item [($\romannumeral3$)] $t_{iiij}=t_{ikkk}=t_{iiik}=0$, and $t_{ijjj}=\pm1$;
    \item [($\romannumeral4$)] $t_{ijjj}=t_{ikkk}=0$, and $t_{iiij}t_{iiik}t_{jjjk}=1$;
    \item [($\romannumeral5$)] $t_{ijjj}=t_{ikkk}=0$, and $t_{iiij}t_{iiik}t_{jjjk}=-1$;
    \item [($\romannumeral6$)] $t_{iiij}=t_{iiik}=0$, and $t_{ijjj}t_{ikkk}t_{jjjk}=1$;
    \item [($\romannumeral7$)] $t_{iiij}=t_{iiik}=0$, and $t_{ijjj}t_{ikkk}t_{jjjk}=-1$;
    \item [($\romannumeral8$)] $t_{ijjj}=t_{iiik}=0$, and $t_{iiij}t_{ikkk}t_{jjjk}=1$;
    \item [($\romannumeral9$)] $t_{ijjj}=t_{iiik}=0$, and $t_{iiij}t_{ikkk}t_{jjjk}=-1$.
  \end{itemize}
\end{flushleft}
Similar to the prove of Theorem \ref{th3.3}, we only need to consider $t_{jjjk}=1$, $t_{iiij}=t_{jjjk}=1$, $t_{ijjj}=t_{jjjk}=1$, $t_{ijjj}=t_{ikkk}=t_{jjjk}=1$, $t_{ijjj}=-t_{ikkk}=t_{jjjk}=1$, $t_{iiij}=t_{iiik}=t_{jjjk}=1$, $t_{iiij}=-t_{iiik}=t_{jjjk}=1$, $t_{iiij}=t_{ikkk}=t_{jjjk}=1$ and $t_{iiij}=-t_{ikkk}=t_{jjjk}=1$, respectively.

Without loss the generality, suppose $t_{1112}t_{1222}=t_{1113}t_{1333}=0$, $t_{2223}t_{2333}=1$. And it follows from the positive semi-definiteness of $\mathcal{T}$ with Eq. \eqref{eq:2} that \begin{center}
$t_{2233}=1$ and $t_{1122},t_{1133}\in\{0,1\}$, for ($\romannumeral1$)\\
	$t_{1122}=t_{2233}=1$ and $t_{1133}\in\{0,1\}$, for ($\romannumeral2$) and ($\romannumeral3$);\\
$t_{iijj}=1$ for all $i,j\in\{1,2,3\}$ and $i\neq j$, for others.
\end{center}

($\romannumeral1$) We might take $t_{1112}=t_{1222}=t_{1113}=t_{1333}=0$, and $t_{2333}=t_{2223}=1$.
Then For $x=(x_1,x_2,x_3)^{\top}\in{\mathbb{R}}^3$,
\begin{align*}
{\mathcal{T}}x^4=&(x_1+x_2+x_3)^4-4(x_1^3x_2+x_1x_2^3+x_1^3x_3+x_1x_3^3)+12[(t_{1123}-1)x_1^2x_2x_3+(t_{1223}-1)x_1x_2^2x_3\\
&+(t_{1233}-1)x_1x_2x_3^2]+6[(t_{1122}-1)x_1^2x_2^2+(t_{1133}-1)x_1^2x_3^2].
\end{align*}
For $x=(1,6,-3)^\top$, we have\begin{align*}
	{\mathcal{T}}x^4=82-216(t_{1123}+6t_{1223}-3t_{1233})+54(4t_{1122}+t_{1133})\geq0,
\end{align*}
and for $x=(1,-3,6)^\top$, we have\begin{align*}
	{\mathcal{T}}x^4=82-216(t_{1123}-3t_{1223}+6t_{1233})+54(t_{1122}+4t_{1133})\geq0.
\end{align*}
So, the following cases could not occur,
 \begin{itemize}
    \item $t_{1223}=1$ or $t_{1233}=1$;
    \item $t_{1223}=-1$ and $t_{1233}=0$;
	\item $t_{1223}=0$ and $t_{1233}=-1$.
\end{itemize}
Next, we discuss other situations.

\textbf{Case 1.} $t_{1223}=t_{1233}=-1$. Let $x=(1,-6,3)^{\top}$, then
$${\mathcal{T}}x^4\leq (x_1+x_2+x_3)^4-4(x_1^3x_2+x_1x_2^3+x_1^3x_3+x_1x_3^3)-24(x_1^2x_2x_3+x_1x_2^2x_3+x_1x_2x_3^2)=-80<0.$$

\textbf{Case 2.} $t_{1223}=t_{1233}=0$.

\textbf{Case 2.1}  When $t_{1123}=1$, and $t_{1122}=0$ or $t_{2233}=0$. Without loss the generality, suppose $t_{1122}=1$. Let $x=(1,1,-1)^{\top}$, then
$${\mathcal{T}}x^4\leq(x_1+x_2+x_3)^4-4(x_1^3x_2+x_1x_2^3+x_1^3x_3+x_1x_3^3)-12(x_1x_2^2x_3+x_1x_2x_3^2)-6x_1^2x_2^2=-5<0.$$

\textbf{Case 2.2}  When $t_{1123}=-1$, and $t_{1122}=t_{1133}=0$. Let $x=(5,1,3)^{\top}$, then
\begin{align*}
{\mathcal{T}}x^4=&(x_1+x_2+x_3)^4-4(x_1^3x_2+x_1x_2^3+x_1^3x_3+x_1x_3^3)-24x_1^2x_2x_3\\
&-12(x_1x_2^2x_3+x_1x_2x_3^2)-6(x_1^2x_2^2+x_1^2x_3^2)=-19<0.
\end{align*}

($\romannumeral2$) We might take $t_{1222}=t_{1113}=t_{1333}=0$, and $t_{1112}=t_{2223}=t_{2333}=1$.
Then for $x=(x_1,x_2,x_3)^{\top}\in{\mathbb{R}}^3$
\begin{align*}
{\mathcal{T}}x^4=&(x_1+x_2+x_3)^4-4(x_1x_2^3+x_1^3x_3+x_1x_3^3)\\
&+12[(t_{1123}-1)x_1^2x_2x_3+(t_{1223}-1)x_1x_2^2x_3+(t_{1233}-1)x_1x_2x_3^2]+6(t_{1133}-1)x_1^2x_3^2.
\end{align*}
For $x=(1,8,-3)^\top$, we have\begin{align*}
	{\mathcal{T}}x^4=1042-288(t_{1123}+8t_{1223}-3t_{1233})+54t_{1133}\geq0.
\end{align*}
So, the following cases could not occur,
 \begin{itemize}
    \item $t_{1223}=1$;
    \item $t_{1123}=1$, $t_{1223}=0$ and $t_{1233}=-1$.
\end{itemize}
Next, we discuss other situations.

\textbf{Case 1.} $t_{1223}=0$.

\textbf{Subcase 1.1} When $t_{1233}=-1$ and $t_{1123}\neq1$. Let $x=(1,5,-3)^{\top}$, then
$${\mathcal{T}}x^4\leq (x_1+x_2+x_3)^4-4(x_1x_2^3+x_1^3x_3+x_1x_3^3)-12x_1x_2^2x_3-24(x_1^2x_2x_3+x_1x_2x_3^2)=-119<0.$$

\textbf{Subcase 1.2} When $t_{1233}\neq-1$ and $t_{1123}=-1$, let $x=(-6,3,1)^{\top}$, then
$${\mathcal{T}}x^4\leq (x_1+x_2+x_3)^4-4(x_1x_2^3+x_1^3x_3+x_1x_3^3)-24x_1^2x_2x_3-12(x_1x_2^2x_3+x_1x_2x_3^2)=-176<0.$$

\textbf{Subcase 1.3} When at least one of $\{t_{1123},t_{1233}\}$ is $1$ and the other one is not $-1$. Without loss the generality, suppose $t_{1233}=1$. Let $x=(1,-1,1)^{\top}$, then
$${\mathcal{T}}x^4\leq (x_1+x_2+x_3)^4-4(x_1x_2^3+x_1^3x_3+x_1x_3^3)-12(x_1^2x_2x_3+x_1x_2^2x_3)=-3<0.$$

\textbf{Case 2.} $t_{1223}=-1$.

\textbf{Subcase 2.1} When $t_{1123}=t_{1233}=-1$, let $x=(1,1,1)^{\top}$, then
$${\mathcal{T}}x^4\leq (x_1+x_2+x_3)^4-4(x_1x_2^3+x_1^3x_3+x_1x_3^3)-24(x_1^2x_2x_3+x_1x_2^2x_3+x_1x_2x_3^2)=-3<0.$$

\textbf{Subcase 2.2} When one of $\{t_{1123},t_{1233}\}$ is $0$ and the other one is $-1$. Let $x=(1,-4,2)^{\top}$, then
$${\mathcal{T}}x^4\leq (x_1+x_2+x_3)^4-4(x_1x_2^3+x_1^3x_3+x_1x_3^3)-12x_1^2x_2x_3-24(x_1x_2^2x_3+x_1x_2x_3^2)=-71<0.$$

\textbf{Subcase 2.3} When $\{t_{1123},t_{1233}\}$ are not $-1$. Let $x=(1,-1,1)^{\top}$, then
$${\mathcal{T}}x^4\leq (x_1+x_2+x_3)^4-4(x_1x_2^3+x_1^3x_3+x_1x_3^3)-12(x_1^2x_2x_3+x_1x_2x_3^2)-24x_1x_2^2x_3=-3<0.$$

($\romannumeral3$) We might take $t_{1112}=t_{1113}=t_{1333}=0$, and $t_{1222}=t_{2223}=t_{2333}=1$.
For $x=(x_1,x_2,x_3)^{\top}\in{\mathbb{R}}^3$
\begin{align*}
{\mathcal{T}}x^4=&(x_1+x_2+x_3)^4-4(x_1^3x_2+x_1^3x_3+x_1x_3^3)\\
&+12[(t_{1123}-1)x_1^2x_2x_3+(t_{1223}-1)x_1x_2^2x_3+(t_{1233}-1)x_1x_2x_3^2]+6(t_{1133}-1)x_1^2x_3^2.
\end{align*}
For $x=(1,-4,1)^\top$, we have\begin{align*}
	{\mathcal{T}}x^4=-78-48(t_{1123}-4t_{1223}+t_{1233})+6t_{1133}\geq0.
\end{align*}
So, the following cases could not occur,
 \begin{itemize}
    \item $t_{1223}=-1$;
    \item $t_{1223}=0$, and at most one of $\{t_{1123},t_{1233}\}$ is $1$.
\end{itemize}
Next, we discuss other situations.

\textbf{Case 1.} $t_{1223}=0$ and $t_{1123}=t_{1233}=-1$. Let $x=(1,2,-3)^{\top}$, then
$${\mathcal{T}}x^4\leq (x_1+x_2+x_3)^4-4(x_1^3x_2+x_1^3x_3+x_1x_3^3)-12x_1x_2^2x_3-24(x_1^2x_2x_3+x_1x_2x_3^2)=-32<0.$$

\textbf{Case 2.} $t_{1223}=1$.

\textbf{Case 2.1} When $t_{1233}\neq1$. Let $x=(1,4,-3)^{\top}$, then
$${\mathcal{T}}x^4\leq (x_1+x_2+x_3)^4-4(x_1^3x_2+x_1^3x_3+x_1x_3^3)-24x_1^2x_2x_3-12x_1x_2x_3^2=-24<0.$$

\textbf{Case 2.2} When $t_{1233}=1$ and $t_{1123}\geq0$. Let $x=(1,-2,3)^{\top}$, then
$${\mathcal{T}}x^4\leq (x_1+x_2+x_3)^4-4(x_1^3x_2+x_1^3x_3+x_1x_3^3)-12x_1^2x_2x_3=-24<0.$$

\textbf{Case 2.3} When $t_{1233}=-t_{1123}=1$. Let $x=(-1,1,1)^{\top}$, then
$${\mathcal{T}}x^4\leq (x_1+x_2+x_3)^4-4(x_1^3x_2+x_1^3x_3+x_1x_3^3)-24x_1^2x_2x_3=-11<0.$$

($\romannumeral4$) We might take $t_{1222}=t_{1333}=0$ and $t_{1112}=t_{1113}=t_{2223}=t_{2333}=t_{1122}=t_{2233}=t_{1133}=1$.
Then for $x=(x_1,x_2,x_3)^{\top}\in{\mathbb{R}}^3$
\begin{align*}
{\mathcal{T}}x^4=&(x_1+x_2+x_3)^4-4(x_1x_2^3+x_1x_3^3)\\
&+12[(t_{1123}-1)x_1^2x_2x_3+(t_{1223}-1)x_1x_2^2x_3+(t_{1233}-1)x_1x_2x_3^2].
\end{align*}
For $x=(-5,2,1)^\top$, we have\begin{align*}
	{\mathcal{T}}x^4=-44+96(5t_{1123}-2t_{1223}-t_{1233})\geq0,
\end{align*}
and for $x=(-5,1,2)^\top$, we have\begin{align*}
	{\mathcal{T}}x^4=-44+120(5t_{1123}-t_{1223}-2t_{1233})\geq0.
\end{align*}
So, the following cases could not occur,
 \begin{itemize}
    \item $t_{1123}=-1$;
	\item $t_{1123}=0$ and $t_{1223}=1$ or $t_{1233}=1$;
    \item $t_{1123}=t_{1223}=t_{1223}=0$.
\end{itemize}
Discuss other situations later.

\textbf{Case 1.} $t_{1123}=0$, at least one $\{t_{1223},t_{1233}\}$ is $-1$ and the other one is not $1$. Let $x=(1,-3,1)^{\top}$, then
$${\mathcal{T}}x^4\leq(x_1+x_2+x_3)^4-4(x_1x_2^3+x_1x_3^3)-12x_1^2x_2x_3-24(x_1x_2^2x_3+x_1x_2x_3^2)=-3<0.$$

\textbf{Case 2.} $t_{1123}=1$.

\textbf{Subcase 2.1} When $t_{1223}=1$ or $t_{1233}=1$. Without loss the generality, suppose $t_{1223}=1$. Let $x=(1,2,-1)^{\top}$, then
$${\mathcal{T}}x^4\leq(x_1+x_2+x_3)^4-4(x_1x_2^3+x_1x_3^3)=-12<0.$$

\textbf{Subcase 2.2} When at least one $\{t_{1223},t_{1233}\}$ is $-1$ and the other one is not $1$. Without loss the generality, suppose $t_{1233}=-1$. Let $x=(3,1,-2)^{\top}$, then
$${\mathcal{T}}x^4\leq(x_1+x_2+x_3)^4-4(x_1x_2^3+x_1x_3^3)-24(x_1x_2^2x_3+x_1x_2x_3^2)=-44<0.$$

($\romannumeral5$). We might take $t_{1222}=t_{1333}=0$, $t_{1112}=-t_{1113}=t_{2223}=t_{2333}=1$.
Then for $x=(x_1,x_2,x_3)^{\top}\in{\mathbb{R}}^3$
\begin{align*}
{\mathcal{T}}x^4=&(x_1+x_2+x_3)^4-4(x_1x_2^3+x_1x_3^3)-8x_1^3x_3\\
&+12[(t_{1123}-1)x_1^2x_2x_3+(t_{1223}-1)x_1x_2^2x_3+(t_{1233}-1)x_1x_2x_3^2].
\end{align*}
For  $x=(1,-2,3)^\top$, we have
\begin{align*}
	{\mathcal{T}}x^4=60-72(t_{1123}-2t_{1223}+3t_{1233})\geq0.
\end{align*}
So, the following cases could not occur,
 \begin{itemize}
    \item $t_{1123}=1$ and $-2t_{1223}+3t_{1233}\leq0$;
    \item $t_{1233}=1$ and $t_{1123}-2t_{1223}\geq-2$.
\end{itemize}
Discuss other situations later.

\textbf{Case 1.} $t_{1123}=1$ or $t_{1233}=1$.

\textbf{Subcase 1.1} When $t_{1123}=1$ and $t_{1223}>t_{1233}$. Let $x=(1,1,-1)^{\top}$, then
$${\mathcal{T}}x^4\leq(x_1+x_2+x_3)^4-4(x_1x_2^3+x_1x_3^3)-8x_1^3x_3-12x_1x_2x_3^2=-3<0.$$

\textbf{Subcase 1.2} When $t_{1123}=1$ and $t_{1223}=t_{1233}=-1$, let $x=(1,-1,1)^{\top}$, then
$${\mathcal{T}}x^4= (x_1+x_2+x_3)^4-4(x_1x_2^3+x_1x_3^3)-8x_1^3x_3-24(x_1x_2^2x_3+x_1x_2x_3^2)=-7<0.$$

\textbf{Subcase 1.3} When $t_{1233}=t_{1223}=1$ and $t_{1123}=-1$. Let $x=(-1,1,1)^{\top}$, then
$${\mathcal{T}}x^4=(x_1+x_2+x_3)^4-4(x_1x_2^3+x_1x_3^3)-8x_1^3x_3-24x_1^2x_2x_3=-7<0.$$

\textbf{Case 2.} $t_{1123}\neq1$ and $t_{1233}\neq1$.

\textbf{Subcase 2.1} When $t_{1223}=1$, or $t_{1223}=0$ and $t_{1233}=-1$. Let $x=(1,3,-3)^{\top}$, then
$${\mathcal{T}}x^4\leq (x_1+x_2+x_3)^4-4(x_1x_2^3+x_1x_3^3)-8x_1^3x_3-24(x_1^2x_2x_3+x_1x_2x_3^2)+12x_1x_2^2x_3=-83<0.$$

\textbf{Subcase 2.2} When $t_{1123}=t_{1223}=t_{1233}=0$. Let $x=(4,-2,1)^{\top}$, then
$${\mathcal{T}}x^4= (x_1+x_2+x_3)^4-4(x_1x_2^3+x_1x_3^3)-8x_1^3x_3-12(x_1^2x_2x_3+x_1x_2^2x_3+x_1x_2x_3^2)=-31<0.$$

\textbf{Subcase 2.3} When $t_{1223}=-1$ and $t_{1123}+t_{1233}\geq-1$. Let $x=(1,-3,2)^{\top}$, then
$${\mathcal{T}}x^4\leq(x_1+x_2+x_3)^4-4(x_1x_2^3+x_1x_3^3)-8x_1^3x_3-12x_1^2x_2x_3-24(x_1x_2^2x_3+x_1x_2x_3^2)=-12<0.$$

\textbf{Subcase 2.4} When $t_{1223}=t_{1123}=t_{1233}=-1$. Let $x=(1,1,1)^{\top}$, then
$${\mathcal{T}}x^4= (x_1+x_2+x_3)^4-4(x_1x_2^3+x_1x_3^3)-8x_1^3x_3-24(x_1^2x_2x_3+x_1x_2^2x_3+x_1x_2x_3^2)=-7<0.$$

($\romannumeral6$) We might take $t_{1112}=t_{1113}=0$ and $t_{1222}=t_{1333}=t_{2223}=t_{2333}=t_{1122}=t_{2233}=t_{1133}=1$.
Then for $x=(x_1,x_2,x_3)^{\top}\in{\mathbb{R}}^3$
\begin{align*}
{\mathcal{T}}x^4=&(x_1+x_2+x_3)^4-4(x_1^3x_2+x_1^3x_3)\\
&+12[(t_{1123}-1)x_1^2x_2x_3+(t_{1223}-1)x_1x_2^2x_3+(t_{1233}-1)x_1x_2x_3^2].
\end{align*}

\textbf{Case 1.} $t_{1223}\neq1$.

\textbf{Subcase 1.1} When $t_{1123}+t_{1233}\geq 2t_{1223}$. Let $x=(1,-3,1)^{\top}$, then
$${\mathcal{T}}x^4\leq (x_1+x_2+x_3)^4-4(x_1^3x_2+x_1^3x_3)-12(x_1^2x_2x_3+x_1x_2^2x_3+x_1x_2x_3^2)=-27<0.$$

\textbf{Subcase 1.2} When $t_{1223}=0$ and $t_{1233}=-1$. Let $x=(1,1,-2)^{\top}$, then
$${\mathcal{T}}x^4\leq (x_1+x_2+x_3)^4-4(x_1^3x_2+x_1^3x_3)-12x_1x_2^2x_3-24(x_1^2x_2x_3+x_1x_2x_3^2)=-20<0.$$

\textbf{Subcase 1.3} When $t_{1223}=t_{1233}=0$ and $t_{1123}=-1$. Let $x=(1,-5,2)^{\top}$, then
$${\mathcal{T}}x^4= (x_1+x_2+x_3)^4-4(x_1^3x_2+x_1^3x_3)-24x_1^2x_2x_3-12(x_1x_2^2x_3+x_1x_2x_3^2)=-92<0.$$

\textbf{Case 2.} $t_{1223}=1$.

\textbf{Subcase 2.1} When $t_{1123}\geq t_{1233}\neq1$. Let $x=(1,1,-2)^{\top}$, then
$${\mathcal{T}}x^4\leq(x_1+x_2+x_3)^4-4(x_1^3x_2+x_1^3x_3)-12(x_1^2x_2x_3+x_1x_2x_3^2)=-20<0.$$

\textbf{Subcase 2.2} When $t_{1123}<t_{1233}$. Let $x=(-1,1,1)^{\top}$, then
$${\mathcal{T}}x^4\leq (x_1+x_2+x_3)^4-4(x_1^3x_2+x_1^3x_3)-12x_1^2x_2x_3=-3<0.$$

($\romannumeral7$) We might take $t_{1112}=t_{1113}=0$ and $t_{1222}=-t_{1333}=t_{2223}=t_{2333}=1$.
Then for $x=(x_1,x_2,x_3)^{\top}\in{\mathbb{R}}^3$
\begin{align*}
{\mathcal{T}}x^4=&(x_1+x_2+x_3)^4-4(x_1^3x_2+x_1^3x_3)-8x_1x_3^3\\
&+12[(t_{1123}-1)x_1^2x_2x_3+(t_{1223}-1)x_1x_2^2x_3+(t_{1233}-1)x_1x_2x_3^2].
\end{align*}
For  $x=(1,-2,2)^\top$, we have
\begin{align*}
	{\mathcal{T}}x^4=-15-48(t_{1123}-2t_{1223}+2t_{1233})\geq0.
\end{align*}
So, the following cases could not occur,
 \begin{itemize}
    \item $t_{1223}<t_{1233}$;
    \item $t_{1223}=t_{1233}$ and $t_{1123}\geq0$.
\end{itemize}
Discuss other situations later.

\textbf{Case 1.} $t_{1123}=-1$ and $t_{1223}=t_{1233}$. Let $x=(1,-4,4)^{\top}$, then
$${\mathcal{T}}x^4=(x_1+x_2+x_3)^4-4(x_1^3x_2+x_1^3x_3)-8x_1x_3^3-24x_1^2x_2x_3=-127<0.$$

\textbf{Case 2.} $t_{1223}>t_{1233}$. Let $x=(1,7,-7)^{\top}$, then
$${\mathcal{T}}x^4\leq(x_1+x_2+x_3)^4-4(x_1^3x_2+x_1^3x_3)-8x_1x_3^3-24x_1^2x_2x_3-12x_1x_2x_3^2=-195<0.$$

($\romannumeral8$) We might take $t_{1222}=t_{1113}=0$, $t_{1112}=t_{1333}=t_{2223}=t_{2333}=1$.
For $x=(x_1,x_2,x_3)^{\top}\in{\mathbb{R}}^3$
\begin{align*}
{\mathcal{T}}x^4=&(x_1+x_2+x_3)^4-4(x_1x_2^3+x_1^3x_3)\\
&+12[(t_{1123}-1)x_1^2x_2x_3+(t_{1223}-1)x_1x_2^2x_3+(t_{1233}-1)x_1x_2x_3^2].
\end{align*}
For  $x=(1,3,-5)^\top$, we have
\begin{align*}
	{\mathcal{T}}x^4=-267-180(t_{1123}+3t_{1223}-5t_{1233})\geq0.
\end{align*}
So, the following cases could not occur,
 \begin{itemize}
    \item $t_{1233}\neq1$ and $t_{1223}\geq t_{1233}$;
    \item $t_{1123}=t_{1223}=t_{1233}=1$.
\end{itemize}
Discuss other situations later.

\textbf{Case 1.} $t_{1223}=t_{1233}=1$ and $t_{1123}\neq1$. Let $x=(-1,1,1)^{\top}$, then
$${\mathcal{T}}x^4\leq (x_1+x_2+x_3)^4-4(x_1x_2^3+x_1^3x_3)-12x_1^2x_2x_3=-3<0.$$

\textbf{Case 2.} $t_{1223}<t_{1233}$.

\textbf{Subcase 2.1} When $t_{1223}=0$ and $t_{1123}\geq0$, or $t_{1223}=-1$. Let $x=(1,-4,2)^{\top}$, then
$${\mathcal{T}}x^4\leq (x_1+x_2+x_3)^4-4(x_1x_2^3+x_1^3x_3)-12(x_1^2x_2x_3+x_1x_2^2x_3)=-39<0.$$

\textbf{Subcase 2.2} When $t_{1223}=0$ and $t_{1123}=-1$, let $x=(-1,1,1)^{\top}$, then
$${\mathcal{T}}x^4= (x_1+x_2+x_3)^4-4(x_1x_2^3+x_1^3x_3)-24x_1^2x_2x_3-12x_1x_2^2x_3=-3<0.$$

($\romannumeral9$) We might take $t_{1222}=t_{1113}=0$, $t_{1112}=-t_{1333}=t_{2223}=t_{2333}=1$.
Then for $x=(x_1,x_2,x_3)^{\top}\in{\mathbb{R}}^3$
\begin{align*}
{\mathcal{T}}x^4=&(x_1+x_2+x_3)^4-4(x_1x_2^3+x_1^3x_3)-8x_1x_3^3\\
&+12[(t_{1123}-1)x_1^2x_2x_3+(t_{1223}-1)x_1x_2^2x_3+(t_{1233}-1)x_1x_2x_3^2].
\end{align*}

\textbf{Case 1.} $t_{1223}\leq t_{1233}$. Let $x=(1,-6,6)^{\top}$, then
$${\mathcal{T}}x^4= (x_1+x_2+x_3)^4-4(x_1x_2^3+x_1^3x_3)-8x_1x_3^3-24x_1^2x_2x_3=-23<0.$$

\textbf{Case 2.} $t_{1223}>t_{1233}$. Let $x=(1,4,-4)^{\top}$, then
$${\mathcal{T}}x^4\leq (x_1+x_2+x_3)^4-4(x_1x_2^3+x_1^3x_3)-8x_1x_3^3-24x_1^2x_2x_3-12x_1x_2x_3^2=-111<0.$$

The necessity is proved.
\end{proof}

\begin{theorem}\label{th3.14}
Let ${\mathcal{T}}=(t_{ijkl})\in\widehat{{\mathcal{E}}}_{4,3}$ and $t_{iiij}t_{ijjj}=t_{iiik}t_{ikkk}=0$, $t_{jjjk}t_{jkkk}=-1$ for $i,j,k\in\{1,2,3\}$, $i\neq j$, $i\neq k$, $j\neq k$. Then $\mathcal{T}$ is positive definite if and only if one of the following conditions is satisfied.
\begin{itemize}
  \item [(a)]  $t_{iiij}=t_{ijjj}=t_{iiik}=t_{ikkk}=0$, and
  \begin{itemize}
    \item [(a$_1$)] $t_{1123}=t_{1223}=t_{1233}=0$, $t_{jjkk}=1$ and $t_{iijj},t_{iikk}\in\{0,1\}$, or
    \item [(a$_2$)] $t_{iijk}=\pm1$, $t_{ijjk}=t_{ijkk}=0$, and $t_{1122}=t_{2233}=t_{1133}=1$.
  \end{itemize}
  \item [(b)] $t_{iiij}=\pm1$, $t_{ijjj}=t_{ikkk}=t_{iiik}=0$, and
  \begin{itemize}
   \item [(b$_1$)] $t_{1123}=t_{1223}=t_{1233}=0$, $t_{iijj}=t_{jjkk}=1$ and $t_{iikk}\in\{0,1\}$, or
   \item [(b$_2$)] $t_{ijkk}=0$, and $t_{iijk}t_{jkkk}=t_{ijjk}t_{iiij}t_{jkkk}=t_{1122}=t_{2233}=t_{1133}=1$.
  \end{itemize}
  \item [(c)] $t_{ijjj}=\pm1$, $t_{iiij}=t_{ikkk}=t_{iiik}=0$, $t_{1122}=t_{2233}=t_{1133}=1$, and
  \begin{itemize}
   \item [(c$_1$)] $t_{iijk}=t_{ijkk}=0$, and $t_{ijjj}t_{jjjk}t_{ijjk}=1$, or
   \item [(c$_2$)] $t_{ijjk}=t_{ijkk}=0$, and $t_{iijk}t_{jkkk}=1$.
  \end{itemize}
  \item [(d)] $t_{ijjj}=t_{ikkk}=t_{ijjk}=t_{ijkk}=0$, and $t_{iijk}t_{iiij}t_{iiik}=t_{1122}=t_{2233}=t_{1133}=1$.
  \item [(e)] $t_{iiij}=t_{iiik}=0$, $t_{jjjk}=t_{1122}=t_{2233}=t_{1133}=1$, and
  \begin{itemize}
   \item [(e$_1$)] $t_{ijjj}t_{ikkk}=-t_{iijk}=-t_{ijkk}t_{ijjj}=1$, and $t_{ijjk}=0$, or
   \item [(e$_2$)] $t_{ijjj}t_{ikkk}=-t_{iijk}=-t_{ijjk}t_{ijjj}=-1$, and $t_{ijkk}=0$.
  \end{itemize}
  \item [(f)] $t_{ijjj}=t_{iiik}=t_{iijk}=0$, $t_{ikkk}t_{jkkk}t_{ijkk}=t_{1122}=t_{2233}=t_{1133}=1$, and
   \begin{itemize}
   \item [(f$_1$)] $t_{ijjk}=0$, or
   \item [(f$_2$)] $t_{iiij}t_{ikkk}t_{jjjk}=-t_{ijjk}t_{ikkk}=-t_{ijkk}t_{iiij}=1$.
  \end{itemize}
\end{itemize}
\end{theorem}
\begin{proof}
``\textbf{if (Sufficiency)}."
(a) Suppose $t_{iiij}=t_{ijjj}=t_{iiik}=t_{ikkk}=0$ for $i,j,k\in\{1,2,3\}$, $i\neq j$, $i\neq k$, $j\neq k$.

(a$_1$) If $t_{1123}=t_{1223}=t_{1233}=0$, $t_{jjkk}=1$ and $t_{iijj},t_{iikk}\in\{0,1\}$. Then for $x=(x_1,x_2,x_3)^{\top}\in{\mathbb{R}}^3$,
\begin{equation}\label{th3.14a1}
{\mathcal{T}}x^4\geq x_i^4+(t_{jjjk}x_j^2+t_{jkkk}x_k^2+2x_jx_k)^2+4x_j^2x_k^2\geq0.
\end{equation}

(a$_2$) If $t_{iijk}=\pm1$, $t_{ijjk}=t_{ijkk}=0$, and $t_{1122}=t_{2233}=t_{1133}=1$. Then for $x=(x_1,x_2,x_3)^{\top}\in{\mathbb{R}}^3$,
\begin{equation}\label{th3.14a2}
{\mathcal{T}}x^4=x_i^4+(t_{jjjk}x_j+t_{jkkk}x_k+2x_jx_k)^2+4x_j^2x_k^2+6(x_ix_j+t_{iijk}x_ix_k)^2\geq0.
\end{equation}

(b) Suppose $t_{iiij}=\pm1$, and $t_{ijjj}=t_{ikkk}=t_{iiik}=0$ for $i,j,k\in\{1,2,3\}$, $i\neq j$, $i\neq k$, $j\neq k$.

(b$_1$) If $t_{1123}=t_{1223}=t_{1233}=0$, $t_{iijj}=t_{jjkk}=1$ and $t_{iikk}\in\{0,1\}$. Then for $x=(x_1,x_2,x_3)^{\top}\in{\mathbb{R}}^3$,
\begin{equation}\label{th3.14b1}
{\mathcal{T}}x^4\geq (x_i^2+2t_{iiij}x_ix_j)^2+(t_{jjjk}x_j^2+t_{jkkk}x_k^2+2x_jx_k)^2+4x_j^2x_k^2+x_i^2x_j^2\geq0.
\end{equation}

(b$_2$) If $t_{ijkk}=0$, and $t_{iijk}t_{jkkk}=t_{ijjk}t_{iiij}t_{jkkk}=t_{1122}=t_{2233}=t_{1133}=1$.

Suppose $t_{1112}=t_{2223}=-t_{2333}=t_{1122}=t_{2233}=t_{1133}=1$, $t_{1222}=t_{1113}=t_{1333}=0$ then $t_{1233}=0$, $t_{1123}=t_{1223}=-1$. Then for $x=(x_1,x_2,x_3)^{\top}\in{\mathbb{R}}^3$,
\begin{align}\label{th9x3}
{\mathcal{T}}x^4&=x_1^4+x_2^4+x_3^4+4(x_1^3x_2+x_2^3x_3-x_2x_3^3)-12(x_1^2x_2x_3+x_1x_2^2x_3)+6(x_1^2x_2^2+x_2^2x_3^2+x_1^2x_3^2)\nonumber\\
&=(x_1^2-x_2^2+x_3^2+2x_1x_2-2x_2x_3)^2+(2x_1x_2-2x_1x_3+x_2x_3)^2+3x_2^2x_3^2+4x_1x_2^3.
\end{align}
By this equation can deduce ${\mathcal{T}}x^4>0$ if $x_1x_2>0$.

By Lemma \ref{lem:23}, when $x_1=0$, or $x_2=0$, or $x_3=0$, ${\mathcal{T}}x^4\geq0$, with equality if and only if $x_1=x_2=x_3=0$.

According the method in \cite{nka}, rewrite (\ref{th9x3}) as
\begin{equation} \label{th9x33}
{\mathcal{T}}x^4=x_3^4-4x_2x_3^3+6(x_1^2+x_2^2)x_3^2-[12(x_1^2x_2+x_1x_2^2)+4x_2^3]x_3+x_1^4+x_2^4+4x_1x_2^3+6x_1^2x_2^2.
\end{equation}
Then the inner determinants corresponding to (\ref{th9x33}) are
$$\triangle_1^1=4,\ \ ~~~~~~~~\triangle_3^1=-48x_1^2,$$
$$\triangle_5^1=-192\times(8x_1^6-4x_1^5x_2+12x_1^3x_2^3+19x_1^2x_2^4-36x_1x_2^5+12x_2^6),$$
\begin{align*}
\triangle_7^1&=256\times(64x_1^{12}+192x_1^{11}x_2-240x_1^{10}x_2^2-512x_1^9x_2^3+1824x_1^8x_2^4+408x_1^7x_2^5-4368x_1^6x_2^6\\
&~~~-360x_1^5x_2^7+6213x_1^4x_2^8-2904x_1^3x_2^9-648x_1^2x_2^{10}+288x_1x_2^{11}+80x_2^{12}).
\end{align*}

If $x_1,x_2\neq0$ and the signs of $x_1$ and $x_2$ are different,
$$\triangle_3^1=-48x_1^2<0,$$
$$\triangle_5^1=-192\times[4x_1^6-4x_1^5x_2+(2x_1^3+3x_2^3)^2+19x_1^2x_2^4-36x_1x_2^5+3x_2^6]<0.$$

\begin{align*}
\triangle_7^1&=256x_2^{12}\times(64y^{12}+192y^{11}-240y^{10}-512y^9+1824y^8+408y^7-4368y^6-360y^5+6213y^4\\
&~~~-2904y^3-648y^2+288y+80),
\end{align*}
where $y=\frac{x_1}{x_2}$. Using Lemma \ref{lem:26} to $\frac{\triangle_7^1}{256x_2^{12}}$, then we get the numbers of distinct real roots of $\frac{\triangle_7^1}{256x_2^{12}}=0$ with each $x_2\neq0$ are $0$.
Since $\frac{\triangle_7^1}{256x_2^{12}}=37>0$ for $y=0$, $\frac{\triangle_7^1}{256x_2^{12}}>0$ for any $y\in{\mathbb{R}}$, thus $\triangle_7^1>0$ for any$ x=(x_1,x_2)^{\top}\in{\mathbb{R}}^2$ and $x_2\neq0$.

Then if $x_1x_2<0$ , the number of distinct real roots $N$ of ${\mathcal{T}}x^4=0$ is
$$N=\mbox{var}[+,-,-,+,+]-\mbox{var}[+,+,-,-,+]=2-2=0.$$
Thus, ${\mathcal{T}}x^4=0$ if and only if $x=(0,0,0)^{\top}$. By Lemma \ref{lem:21}, ${\mathcal{T}}$ is positive definite.

All other cases satisfying condition (b$_2$) can be transformed into similar forms of $(\ref{th9x3})$.

(c) Suppose $t_{ijjj}=\pm1$, $t_{iiij}=t_{ikkk}=t_{iiik}=0$, and $t_{1122}=t_{2233}=t_{1133}=1$ for $i,j,k\in\{1,2,3\}$, $i\neq j$, $i\neq k$, $j\neq k$.

(c$_1$) If $t_{iijk}=t_{ijkk}=0$, and $t_{ijjj}t_{jjjk}t_{ijjk}=1$. Then for $x=(x_1,x_2,x_3)^{\top}\in{\mathbb{R}}^3$,
\begin{align}\label{th3.14c1}
{\mathcal{T}}x^4&\geq x_i^4+(t_{jjjk}x_j^2+t_{jkkk}x_k^2+2x_jx_k+2t_{ijjj}t_{jjjk}x_ix_j)^2+(2x_ix_k+t_{ijjj}x_jx_k)^2\nonumber\\
&~~~+2(x_ix_j+t_{ijjk}x_jx_k)^2+x_j^2x_k^2+2x_i^2x_k^2\geq0.
\end{align}

(c$_2$) If $t_{ijjk}=t_{ijkk}=0$, and $t_{iijk}t_{jkkk}=1$.

Suppose $t_{1222}=t_{2223}=-t_{2333}=t_{1122}=t_{2233}=t_{1133}=1$, $t_{1112}=t_{1113}=t_{1333}=0$ then $t_{1123}=0$, $t_{1223}=t_{1223}=0$. Then for $x=(x_1,x_2,x_3)^{\top}\in{\mathbb{R}}^3$,
\begin{equation}\label{th9x3-2}
{\mathcal{T}}x^4=x_1^4+x_2^4+x_3^4+4(x_1x_2^3+x_2^3x_3-x_2x_3^3)-12x_1^2x_2x_3+6(x_1^2x_2^2+x_2^2x_3^2+x_1^2x_3^2).
\end{equation}

By Lemma \ref{lem:23}, when $x_1=0$, or $x_2=0$, or $x_3=0$, ${\mathcal{T}}x^4\geq0$, with equality if and only if $x_1=x_2=x_3=0$.

According the method in \cite{nka}, rewrite (\ref{th9x3}) as
\begin{equation} \label{th9x3-23}
{\mathcal{T}}x^4=x_3^4-4x_2x_3^3+6(x_1^2+x_2^2)x_3^2-(12x_1^2x_2+4x_2^3)x_3+x_1^4+x_2^4+4x_1^3x_2+6x_1^2x_2^2.
\end{equation}
Then the inner determinants corresponding to (\ref{th9x3-23}) are
$$\triangle_1^1=4,\ \ ~~~~~~~~\triangle_3^1=-48x_1^2,$$
$$\triangle_5^1=-768\times(2x_1^6-x_1^3x_2^3-2x_1^2x_2^4+3x_2^6)=-768\times[(x_1^3-\frac{x_2^3}{2})^2+(x_1^3-x_1x_2^2)^2+(x_1^2x_2-\frac{3x_2^3}{2})^2+x_1^4x_2^2+\frac{x_2^6}{2}],$$
$$\triangle_7^1=4096\times(4x_1^{12}+12x_1^9x_2^3+24x_1^8x_2^4-15x_1^6x_2^6-60x_1^5x_2^7-60x_1^4x_2^8+58x_1^3x_2^9+132x_1^2x_2^{10}+48x_1x_2^{11}+5x_2^{12}).$$
If $x_1\neq0$, then
$$\triangle_7^1=4096x_1^{12}\times(4+12y^3+24y^4-15y^6-60y^7-50y^8+58y^9+132y^{10}+48y^{11}+5y^{12}),$$
where $y=\frac{x_2}{x_1}$.  Using Lemma \ref{lem:26} to $\frac{\triangle_7^1}{4096x_1^{12}}$, then we get the numbers of distinct real roots of $\frac{\triangle_7^1}{4096x_1^{12}}=0$ with each $x_1\neq0$ are $0$.
Since $\frac{\triangle_7^1}{4069x_1^{12}}=4>0$ for $y=0$, $\frac{\triangle_7^1}{4096x_1^{12}}>0$ for any $y\in{\mathbb{R}}$, thus $\triangle_7^1>0$ for any$ x=(x_1,x_2)^{\top}\in{\mathbb{R}}^2$ and $x_1\neq0$.

Then if $x_1x_2\neq0$, the number of distinct real roots $N$ of ${\mathcal{T}}x^4=0$ is
$$N=\mbox{var}[+,-,-,+,+]-\mbox{var}[+,+,-,-,+]=2-2=0.$$
Thus, ${\mathcal{T}}x^4=0$ if and only if $x=(0,0,0)^{\top}$. By Lemma \ref{lem:21}, ${\mathcal{T}}$ is positive definite.

All other cases satisfying condition (c$_2$) can be transformed into similar forms of $(\ref{th9x3-2})$.

(d) Suppose $t_{ijjj}=t_{ikkk}=t_{ijjk}=t_{ijkk}=0$, and $t_{iijk}t_{iiij}t_{iiik}=t_{1122}=t_{2233}=t_{1133}=1$ for $i,j,k\in\{1,2,3\}$, $i\neq j$, $i\neq k$, $j\neq k$. Then for $x=(x_1,x_2,x_3)^{\top}\in{\mathbb{R}}^3$,
\begin{equation}\label{th3.14d}\aligned
{\mathcal{T}}x^4=&(x_i^2+2t_{iiij}x_ix_j+3t_{iiik}x_ix_k)^2+(t_{jjjk}x_j^2+t_{jkkk}x_k^2+2x_jx_k)^2\\&+4x_j^2x_k^2+2(x_ix_j+t_{iijk}x_ix_k)^2\geq0.\endaligned
\end{equation}

(e) Suppose $t_{iiij}=t_{iiik}=0$, $t_{jjjk}=t_{1122}=t_{2233}=t_{1133}=1$ for $i,j,k\in\{1,2,3\}$, $i\neq j$, $i\neq k$, $j\neq k$.

(e$_1$) If $t_{ijjj}t_{ikkk}=-t_{iijk}=-t_{ijkk}t_{ijjj}=1$, and $t_{ijjk}=0$. Then for $x=(x_1,x_2,x_3)^{\top}\in{\mathbb{R}}^3$,
\begin{align}\label{th3.14e1}
{\mathcal{T}}x^4&=(x_i^2-x_jx_k)^2+(x_j^2-x_k^2+2x_jx_k+2t_{ijjj}x_ix_j-2t_{ikkk}x_ix_k)^2+(x_ix_k+t_{ikkk}x_jx_k-x_ix_j)^2\nonumber\\
&~~~+(x_jx_k+t_{ikkk}x_ix_j)^2+(x_ix_k-t_{ikkk}x_jx_k)^2\geq0.
\end{align}
(e$_2$) If $t_{ijjj}t_{ikkk}=-t_{iijk}=-t_{ijjk}t_{ijjj}=-1$, and $t_{ijkk}=0$. Then for $x=(x_1,x_2,x_3)^{\top}\in{\mathbb{R}}^3$,
\begin{align}\label{th3.14e2}
{\mathcal{T}}x^4&=(x_i^2+x_jx_k)^2+(x_j^2-x_k^2+2x_jx_k+2t_{ijjj}x_ix_j-2t_{ikkk}x_ix_k)^2+(x_ix_k+t_{ijjj}x_jx_k+x_ix_j)^2\nonumber\\
&~~~+(x_jx_k+t_{ijjj}x_ix_k)^2+(x_ix_j-t_{ijjj}x_jx_k)^2\geq0.
\end{align}

(f) Suppose $t_{ijjj}=t_{iiik}=t_{iijk}=0$, and $t_{ikkk}t_{jkkk}t_{ijkk}=t_{1122}=t_{2233}=t_{1133}=1$ for $i,j,k\in\{1,2,3\}$, $i\neq j$, $i\neq k$, $j\neq k$.

(f$_1$) If $t_{ijjk}=0$. Then for $x=(x_1,x_2,x_3)^{\top}\in{\mathbb{R}}^3$,
\begin{align}\label{th3.14f1}
{\mathcal{T}}x^4&=(x_i^2+2t_{iiij}x_ix_j)^2+(t_{jjjk}x_j^2+t_{jkkk}x_k^2+2x_jx_k+2t_{jkkk}t_{ikkk}x_ix_k)^2+2(x_ix_k+t_{ijkk}x_jx_k)^2\nonumber\\
&~~~+2(x_ix_j-t_{ijkk}x_jx_k)^2\geq0.
\end{align}

(f$_2$) $t_{iiij}t_{ikkk}t_{jjjk}=-t_{ijjk}t_{ikkk}=-t_{ijkk}t_{iiij}=1$. Then for $x=(x_1,x_2,x_3)^{\top}\in{\mathbb{R}}^3$,
\begin{equation}\label{th3.14f2}\aligned
{\mathcal{T}}x^4=&(x_i^2+2t_{iiij}x_ix_j+2t_{iiij}t_{ijjk}x_jx_k)^2+(t_{jjjk}x_j^2+t_{jkkk}x_k^2+2x_jx_k+2t_{jkkk}t_{ikkk}x_ix_k)^2\\
&+(x_ix_k+t_{ijkk}x_jx_k-t_{iiij}t_{ijjk}x_ix_j)^2+(x_ix_k+t_{ijkk}x_jx_k)^2\\&+(x_ix_j+t_{ijjk}x_jx_k)^2\geq0.
\endaligned\end{equation}

Furthermore, in (\ref{th3.14a1}), (\ref{th3.14a2}), (\ref{th3.14b1}), (\ref{th3.14c1}), and (\ref{th3.14d})-(\ref{th3.14f2}), it is easily verified that ${\mathcal{T}}x^4=0$ if and only if $x=(0,0,0)^{\top}$, and then, $\mathcal{T}$ is positive definite.

``\textbf{only if (Necessity)}." The conditions,
$t_{iiij}t_{ijjj}=t_{iiik}t_{ikkk}=0$, $t_{jjjk}t_{jkkk}=-1$ for $i,j,k\in\{1,2,3\}$, $i\neq j$, $i\neq k$, $j\neq k$, can be divided into seven cases.
\begin{itemize}
    \item [($\romannumeral1$)] $t_{iiij}=t_{ijjj}=t_{iiik}=t_{ikkk}=0$;
	\item [($\romannumeral2$)] $t_{ijjj}=t_{ikkk}=t_{iiik}=0$, and $t_{iiij}=\pm1$;
    \item [($\romannumeral3$)] $t_{iiij}=t_{ikkk}=t_{iiik}=0$, and $t_{ijjj}=\pm1$;
    \item [($\romannumeral4$)] $t_{ijjj}=t_{ikkk}=0$, and $t_{iiij}t_{iiik}t_{jjjk}=1$ ($t_{iiij}t_{iiik}t_{jkkk}=1$ is similar);
    \item [($\romannumeral5$)] $t_{iiij}=t_{iiik}=0$, and $t_{ijjj}t_{ikkk}t_{jjjk}=1$ ($t_{ijjj}t_{ikkk}t_{jkkk}=1$ is similar);
    \item [($\romannumeral6$)] $t_{ijjj}=t_{iiik}=0$, and $t_{iiij}t_{ikkk}t_{jjjk}=1$;
    \item [($\romannumeral7$)]  $t_{ijjj}=t_{iiik}=0$, and $t_{iiij}t_{ikkk}t_{jjjk}=-1$.
\end{itemize}
Similar to the prove of Theorem \ref{th3.3}, we only need to consider $t_{jjjk}=1$, $t_{iiij}=t_{jjjk}=1$, $t_{ijjj}=t_{jjjk}=1$, $t_{iiij}=t_{iiik}=t_{jjjk}=1$, $t_{ijjj}=t_{ikkk}=t_{jjjk}=1$, $t_{iiij}=t_{ikkk}=t_{jjjk}=1$ and $t_{iiij}=-t_{ikkk}=-t_{jjjk}=1$ respectively.

Without loss the generality, suppose $t_{1112}t_{1222}=t_{1113}t_{1333}=0$, $t_{2223}t_{2333}=-1$. And it follows from the positive definiteness of $\mathcal{T}$ with Eq. \eqref{eq:3} that \begin{center}
$t_{2233}=1$ and $t_{1122},t_{1133}\in\{0,1\}$, for ($\romannumeral1$)\\
	$t_{1122}=t_{2233}=1$ and $t_{1133}\in\{0,1\}$, for ($\romannumeral2$) and ($\romannumeral3$);\\
$t_{iijj}=1$ for all $i,j\in\{1,2,3\}$ and $i\neq j$, for others.
\end{center}

($\romannumeral1$) We might take $t_{1112}=t_{1222}=t_{1113}=t_{1333}=0$, and $t_{2223}=-t_{2333}=1$.
Then for $x=(x_1,x_2,x_3)^{\top}\in{\mathbb{R}}^3$
\begin{align*}
{\mathcal{T}}x^4=&(x_1+x_2+x_3)^4-4(x_1^3x_2+x_1x_2^3+x_1^3x_3+x_1x_3^3)-8x_2x_3^3+12[(t_{1123}-1)x_1^2x_2x_3\\
&+(t_{1223}-1)x_1x_2^2x_3+(t_{1233}-1)x_1x_2x_3^2]+6[(t_{1122}-1)x_1^2x_2^2+(t_{1133}-1)x_1^2x_3^2].
\end{align*}
For $x=(1,1,3)^\top$, we have\begin{align*}
	{\mathcal{T}}x^4=41+36(t_{1123}+t_{1223}+3t_{1233})+6(t_{1122}+9t_{1133})>0,
\end{align*}
So, the following cases  could not occur, \begin{itemize}
	\item $t_{1233}=-1$ and at least one of $\{t_{1123},t_{1223}\}$ is $-1$;
	\item $t_{1233}=-1$ and $t_{1123}=t_{1223}=0$;
    \item $t_{1133}=0$ and $t_{1123}+t_{1223}+3t_{1233}\leq-2$.
\end{itemize}
Next, we discuss other situations.

\textbf{Case 1.} $t_{1233}=1$.

\textbf{Subcase 1.1} When $t_{1123}\leq t_{1223}$. Let $x=(-1,1,3)^{\top}$, then
$${\mathcal{T}}x^4\leq (x_1+x_2+x_3)^4-4(x_1^3x_2+x_1x_2^3+x_1^3x_3+x_1x_3^3)-8x_2x_3^3=-7<0.$$

\textbf{Subcase 1.2} When $t_{1123}>t_{1223}$. Let $x=(1,-1,1)^{\top}$, then
$${\mathcal{T}}x^4\leq (x_1+x_2+x_3)^4-4(x_1^3x_2+x_1x_2^3+x_1^3x_3+x_1x_3^3)-8x_2x_3^3-12x_1x_2^2x_3=-3<0.$$

\textbf{Case 2.} $t_{1233}=-1$ and $t_{1123}+t_{1223}>0$. Let $x=(1,1,-1)^{\top}$, then
$${\mathcal{T}}x^4\leq (x_1+x_2+x_3)^4-4(x_1^3x_2+x_1x_2^3+x_1^3x_3+x_1x_3^3)-8x_2x_3^3-12x_1^2x_2x_3-24x_1x_2x_3^2=-3<0.$$

\textbf{Case 3.} $t_{1233}=0$.

\textbf{Subcase 3.1} When $t_{1223}=0$, $t_{1123}=1$, and  $t_{1122}=0$ or $t_{1133}=0$. Let $x=(2,-2,1)^{\top}$, then
$${\mathcal{T}}x^4\leq (x_1+x_2+x_3)^4-4(x_1^3x_2+x_1x_2^3+x_1^3x_3+x_1x_3^3)-8x_2x_3^3-12(x_1x_2^2x_3+x_1x_2x_3^2)-6x_1^2x_2^2=-39<0$$
when $t_{1122}=0$, let $x=(4,-2,3)^{\top}$, then
$${\mathcal{T}}x^4\leq (x_1+x_2+x_3)^4-4(x_1^3x_2+x_1x_2^3+x_1^3x_3+x_1x_3^3)-8x_2x_3^3-12(x_1x_2^2x_3+x_1x_2x_3^2)-6x_1^2x_3^2=-79<0$$
when $t_{1133}=0$.

\textbf{Subcase 3.2} When $t_{1223}=0$, $t_{1123}=-1$, and $t_{1122}=0$ or $t_{1133}=0$. Let $x=(-3,2,1)^{\top}$, then
$${\mathcal{T}}x^4\leq (x_1+x_2+x_3)^4-4(x_1^3x_2+x_1x_2^3+x_1^3x_3+x_1x_3^3)-8x_2x_3^3-24x_1^2x_2x_3-12(x_1x_2^2x_3+x_1x_2x_3^2)-6x_1^2x_2^2=-16<0$$
when $t_{1122}=0$, let $x=(-3,1,2)^{\top}$, then
$${\mathcal{T}}x^4\leq (x_1+x_2+x_3)^4-4(x_1^3x_2+x_1x_2^3+x_1^3x_3+x_1x_3^3)-8x_2x_3^3-24x_1^2x_2x_3-12(x_1x_2^2x_3+x_1x_2x_3^2)-6x_1^2x_3^2=-64<0$$
when $t_{1133}=0$.

\textbf{Subcase 3.3} When $t_{1223}=1$ and $t_{1123}\neq1$. Let $x=(1,3,-1)^{\top}$, then
$${\mathcal{T}}x^4\leq (x_1+x_2+x_3)^4-4(x_1^3x_2+x_1x_2^3+x_1^3x_3+x_1x_3^3)-8x_2x_3^3-12(x_1^2x_2x_3+x_1x_2x_3^2)=-7<0.$$

\textbf{Subcase 3.4} When $t_{1223}=1$ and $t_{1123}=1$. Let $x=(-1,1,1)^{\top}$, then
$${\mathcal{T}}x^4\leq (x_1+x_2+x_3)^4-4(x_1^3x_2+x_1x_2^3+x_1^3x_3+x_1x_3^3)-8x_2x_3^3-24x_1^2x_2x_3-12x_1x_2x_3^2=-3<0.$$

\textbf{Subcase 3.5} When $t_{1223}=-1$ and $t_{1123}\neq-1$. Let $x=(1,-3,1)^{\top}$, then
$${\mathcal{T}}x^4\leq (x_1+x_2+x_3)^4-4(x_1^3x_2+x_1x_2^3+x_1^3x_3+x_1x_3^3)-8x_2x_3^3-24x_1x_2^2x_3-12(x_1^2x_2x_3+x_1x_2x_3^2)=-7<0.$$

\textbf{Subcase 3.6} When $t_{1223}=t_{1123}=-1$. Let $x=(1,1,1)^{\top}$, then
$${\mathcal{T}}x^4\leq (x_1+x_2+x_3)^4-4(x_1^3x_2+x_1x_2^3+x_1^3x_3+x_1x_3^3)-8x_2x_3^3-12x_1x_2x_3^2-24(x_1^2x_2x_3+x_1x_2^2x_3)=-3<0.$$

($\romannumeral2$) We might take $t_{1222}=t_{1113}=t_{1333}=0$, and $t_{1112}=t_{2223}=-t_{2333}=1$.
Then for $x=(x_1,x_2,x_3)^{\top}\in{\mathbb{R}}^3$
\begin{align*}
{\mathcal{T}}x^4=&(x_1+x_2+x_3)^4-4(x_1x_2^3+x_1^3x_3+x_1x_3^3)-8x_2x_3^3\\
&+12[(t_{1123}-1)x_1^2x_2x_3+(t_{1223}-1)x_1x_2^2x_3+(t_{1233}-1)x_1x_2x_3^2]+6(t_{1133}-1)x_1^2x_3^2.
\end{align*}

\textbf{Case 1.} $t_{1233}=1$.

\textbf{Subcase 1.1} When $t_{1123}\leq t_{1223}$. Let $x=(-1,1,3)^{\top}$, then
$${\mathcal{T}}x^4\leq (x_1+x_2+x_3)^4-4(x_1x_2^3+x_1^3x_3+x_1x_3^3)-8x_2x_3^3=-11<0.$$

\textbf{Subcase 1.2} When $t_{1123}>t_{1223}$. Let $x=(1,-1,1)^{\top}$, then
$${\mathcal{T}}x^4\leq (x_1+x_2+x_3)^4-4(x_1x_2^3+x_1^3x_3+x_1x_3^3)-8x_2x_3^3-12x_1x_2^2x_3=-7<0.$$

\textbf{Case 2.} $t_{1233}=-1$.

\textbf{Subcase 2.1} When $t_{1123}+t_{1223}\leq0$. Let $x=(1,1,3)^{\top}$, then
$${\mathcal{T}}x^4\leq (x_1+x_2+x_3)^4-4(x_1x_2^3+x_1^3x_3+x_1x_3^3)-8x_2x_3^3-24(x_1^2x_2x_3+x_1x_2x_3^2)=-3<0.$$

\textbf{Subcase 2.2} When $t_{1123}\geq0$ and $t_{1223}=1$. Let $x=(1,2,-1)^{\top}$, then
$${\mathcal{T}}x^4\leq (x_1+x_2+x_3)^4-4(x_1x_2^3+x_1^3x_3+x_1x_3^3)-8x_2x_3^3-12x_1^2x_2x_3-24x_1x_2x_3^2=-16<0.$$

\textbf{Subcase 2.3} When $t_{1123}=1$ and $t_{1223}=0$. Let $x=(5,-2,1)^{\top}$, then
$${\mathcal{T}}x^4\leq (x_1+x_2+x_3)^4-4(x_1x_2^3+x_1^3x_3+x_1x_3^3)-8x_2x_3^3-12x_1x_2^2x_3-24x_1x_2x_3^2=-88<0.$$

\textbf{Case 3.} $t_{1233}=0$.

\textbf{Subcase 3.1} When $t_{1123}=-1$ and $t_{1223}\neq-1$. Let $x=(-7,3,1)^{\top}$, then
$${\mathcal{T}}x^4\leq (x_1+x_2+x_3)^4-4(x_1x_2^3+x_1^3x_3+x_1x_3^3)-8x_2x_3^3-24x_1^2x_2x_3-12(x_1x_2^2x_3+x_1x_2x_3^2)=-307<0.$$

\textbf{Subcase 3.2} When $t_{1123}=t_{1223}=-1$ and $t_{1133}=0$. Let $x=(-3,1,1)^{\top}$, then
$${\mathcal{T}}x^4= (x_1+x_2+x_3)^4-4(x_1x_2^3+x_1^3x_3+x_1x_3^3)-8x_2x_3^3-12x_1x_2x_3^2-24(x_1^2x_2x_3+x_1x_2^2x_3)-6x_1^2x_3^2=-37<0.$$

\textbf{Subcase 3.3} When $t_{1123}\neq-1$ and $t_{1223}=1$. Let $x=(3,9,-4)^{\top}$, then
$${\mathcal{T}}x^4\leq (x_1+x_2+x_3)^4-4(x_1x_2^3+x_1^3x_3+x_1x_3^3)-8x_2x_3^3-12(x_1^2x_2x_3+x_1x_2x_3^2)=-140<0.$$

\textbf{Subcase 3.4} When $t_{1123}\neq-1$ and $t_{1223}=-1$. Let $x=(1,-3,1)^{\top}$, then
$${\mathcal{T}}x^4\leq (x_1+x_2+x_3)^4-4(x_1x_2^3+x_1^3x_3+x_1x_3^3)-8x_2x_3^3-24x_1x_2^2x_3-12(x_1^2x_2x_3+x_1x_2x_3^2)=-19<0$$

\textbf{Subcase 3.5} When $t_{1123}=1$ and $t_{1223}=0$. Let $x=(6,-3,1)^{\top}$, then
$${\mathcal{T}}x^4\leq (x_1+x_2+x_3)^4-4(x_1x_2^3+x_1^3x_3+x_1x_3^3)-8x_2x_3^3-12(x_1x_2^2x_3+x_1x_2x_3^2)=-392<0.$$

($\romannumeral3$) We might take $t_{1112}=t_{1113}=t_{1333}=0$, and $t_{1222}=t_{2223}=-t_{2333}=1$.
For $x=(x_1,x_2,x_3)^{\top}\in{\mathbb{R}}^3$
\begin{align*}
{\mathcal{T}}x^4=&(x_1+x_2+x_3)^4-4(x_1^3x_2+x_1^3x_3+x_1x_3^3)-8x_2x_3^3\\
&+12[(t_{1123}-1)x_1^2x_2x_3+(t_{1223}-1)x_1x_2^2x_3+(t_{1233}-1)x_1x_2x_3^2]+6(t_{1133}-1)x_1^2x_3^2.
\end{align*}
For $x=(2,-8,1)^\top$, we have\begin{align*}
	{\mathcal{T}}x^4=-79-192(2t_{1123}-8t_{1223}+t_{1233})+24t_{1133}>0.
\end{align*}
So, the following cases could not occur,
 \begin{itemize}
    \item $t_{1223}=-1$;
    \item $t_{1223}=0$ and $t_{1123}=1$;
	\item $t_{1223}=t_{1123}=0$ and $t_{1233}\neq-1$.
\end{itemize}
Next, we discuss other situations.

\textbf{Case 1.} $t_{1223}=0$.

\textbf{Subcase 1.1} When $t_{1123}\leq0$ and $t_{1233}=-1$. Let $x=(1,1,3)^{\top}$, then
$${\mathcal{T}}x^4\leq (x_1+x_2+x_3)^4-4(x_1^3x_2+x_1^3x_3+x_1x_3^3)-8x_2x_3^3-24x_1x_2x_3^2-12(x_1^2x_2x_3+x_1x_2^2x_3)=-3<0.$$

\textbf{Subcase 1.2} When $t_{1123}=-1$ and $t_{1233}=1$. Let $x=(-3,1,2)^{\top}$, then
$${\mathcal{T}}x^4\leq (x_1+x_2+x_3)^4-4(x_1^3x_2+x_1^3x_3+x_1x_3^3)-8x_2x_3^3-24x_1^2x_2x_3-12x_1x_2^2x_3=-4<0.$$

\textbf{Subcase 1.3} When $t_{1123}=-1$, $t_{1233}=0$ and $t_{1133}=0$. Let $x=(-3,1,3)^{\top}$, then
$${\mathcal{T}}x^4=(x_1+x_2+x_3)^4-4(x_1^3x_2+x_1^3x_3+x_1x_3^3)-8x_2x_3^3-24x_1^2x_2x_3-12(x_1x_2^2x_3+x_1x_2x_3^2)-6x_1^2x_3^2=-161<0.$$

\textbf{Case 2.} $t_{1223}=1$.

\textbf{Subcase 2.1} When $t_{1233}=1$. Let $x=(-1,1,3)^{\top}$, then
$${\mathcal{T}}x^4\leq (x_1+x_2+x_3)^4-4(x_1^3x_2+x_1^3x_3+x_1x_3^3)-8x_2x_3^3=-11<0.$$

\textbf{Subcase 2.2} When $t_{1123}=t_{1233}=0$ and $t_{1133}=0$. Let $x=(-3,2,5)^{\top}$, then
$${\mathcal{T}}x^4= (x_1+x_2+x_3)^4-4(x_1^3x_2+x_1^3x_3+x_1x_3^3)-8x_2x_3^3-12(x_1^2x_2x_3+x_1x_2x_3^2)-6x_1^2x_3^2=-118<0.$$

\textbf{Subcase 2.3} When $t_{1233}=0$ and $t_{1123}=1$. Let $x=(4,4,-3)^{\top}$, then
$${\mathcal{T}}x^4\leq (x_1+x_2+x_3)^4-4(x_1^3x_2+x_1^3x_3+x_1x_3^3)-8x_2x_3^3-12x_1x_2x_3^2=-63<0.$$

\textbf{Subcase 2.4} When $t_{1233}=0$ and $t_{1123}=-1$. Let $x=(-1,1,1)^{\top}$, then
$${\mathcal{T}}x^4\leq (x_1+x_2+x_3)^4-4(x_1^3x_2+x_1^3x_3+x_1x_3^3)-8x_2x_3^3-24x_1^2x_2x_3-12x_1x_2x_3^2=-7<0.$$

\textbf{Subcase 2.5}  When $t_{1233}=-1$ and $t_{1123}\neq-1$. Let $x=(3,5,-4)^{\top}$, then
$${\mathcal{T}}x^4\leq (x_1+x_2+x_3)^4-4(x_1^3x_2+x_1^3x_3+x_1x_3^3)-8x_2x_3^3-12x_1^2x_2x_3-24x_1x_2x_3^2=-124<0.$$

\textbf{Subcase 2.6} When $t_{1123}=t_{1233}=-1$. Let $x=(1,1,3)^{\top}$, then
$${\mathcal{T}}x^4\leq (x_1+x_2+x_3)^4-4(x_1^3x_2+x_1^3x_3+x_1x_3^3)-8x_2x_3^3-24(x_1^2x_2x_3+x_1x_2x_3^2)=-3<0.$$

($\romannumeral4$). We might take $t_{1222}=t_{1333}=0$ and $t_{1112}=t_{1113}=t_{2223}=-t_{2333}=1$.
Then for $x=(x_1,x_2,x_3)^{\top}\in{\mathbb{R}}^3$
\begin{align*}
{\mathcal{T}}x^4=&(x_1+x_2+x_3)^4-4(x_1x_2^3+x_1x_3^3)-8x_2x_3^3\\
&+12[(t_{1123}-1)x_1^2x_2x_3+(t_{1223}-1)x_1x_2^2x_3+(t_{1233}-1)x_1x_2x_3^2].
\end{align*}
For $x=(-6,1,2)^\top$, we have\begin{align*}
	{\mathcal{T}}x^4=-199+144(6t_{1123}-t_{1223}-2t_{1233})>0.
\end{align*}
So, the following cases could not occur,
 \begin{itemize}
    \item $t_{1123}=-1$;
    \item $t_{1123}=0$ and $t_{1233}\neq-1$
	\item $t_{1123}=0$ and $t_{1223}=1$.
\end{itemize}
Next, we discuss other situations.

\textbf{Case 1.} $t_{1123}=0$.

\textbf{Subcase 1.1} When $t_{1223}=0$ and $t_{1233}=-1$. Let $x=(-4,1,1)^{\top}$, then
$${\mathcal{T}}x^4= (x_1+x_2+x_3)^4-4(x_1^3x_2+x_1x_3^3)-8x_2x_3^3-24x_1x_2x_3^2-12(x_1^2x_2x_3+x_1x_2^2x_3)=-8<0.$$

\textbf{Subcase 1.2} When $t_{1223}=t_{1233}=-1$. Let $x=(1,1,3)^{\top}$, then
$${\mathcal{T}}x^4= (x_1+x_2+x_3)^4-4(x_1x_2^3+x_1x_3^3)-8x_2x_3^3-12x_1^2x_2x_3-24(x_1x_2^2x_3+x_1x_2x_3^2)=-27<0.$$

\textbf{Case 2.} $t_{1123}=1$.

\textbf{Subcase 2.1} When $t_{1223}<t_{1233}$. Let $x=(1,-1,1)^{\top}$, then
$${\mathcal{T}}x^4\leq (x_1+x_2+x_3)^4-4(x_1x_2^3+x_1x_3^3)-8x_2x_3^3-12x_1x_2^2x_3=-3<0.$$

\textbf{Subcase 2.2} When $t_{1223}>t_{1233}$. Let $x=(1,1,-1)^{\top}$, then
$${\mathcal{T}}x^4\leq (x_1+x_2+x_3)^4-4(x_1x_2^3+x_1x_3^3)-8x_2x_3^3-12x_1x_2x_3^2=-3<0.$$

\textbf{Subcase 2.3} When $t_{1223}=t_{1233}=1$. Let $x=(-1,1,3)^{\top}$, then
$${\mathcal{T}}x^4= (x_1+x_2+x_3)^4-4(x_1x_2^3+x_1x_3^3)-8x_2x_3^3=-23<0.$$

\textbf{Subcase 2.4} When $t_{1223}=t_{1233}=-1$. Let $x=(1,2,4)^{\top}$,
$${\mathcal{T}}x^4= (x_1+x_2+x_3)^4-4(x_1x_2^3+x_1x_3^3)-8x_2x_3^3-24(x_1x_2^2x_3+x_1x_2x_3^2)=-63<0.$$

($\romannumeral5$) We might take $t_{1112}=t_{1113}=0$ and $t_{1222}=t_{1333}=t_{2223}=-t_{2333}=1$.
Then for $x=(x_1,x_2,x_3)^{\top}\in{\mathbb{R}}^3$
\begin{align*}
{\mathcal{T}}x^4=&(x_1+x_2+x_3)^4-4(x_1^3x_2+x_1^3x_3)-8x_2x_3^3\\
&+12[(t_{1123}-1)x_1^2x_2x_3+(t_{1223}-1)x_1x_2^2x_3+(t_{1233}-1)x_1x_2x_3^2].
\end{align*}
For $x=(3,-10,2)^\top$, we have\begin{align*}
	{\mathcal{T}}x^4=-1471-720(3t_{1123}-10t_{1223}+2t_{1233})>0.
\end{align*}
So, the following cases could not occur,
 \begin{itemize}
    \item $t_{1223}=-1$;
    \item $t_{1223}=0$ and $t_{1123}\neq-1$;
	\item $t_{1223}=0$ and $t_{1233}=1$.
\end{itemize}
Next, we discuss other situations.

\textbf{Case 1.} $t_{1123}\neq-1$ and $t_{1223}=t_{1233}=0$. Let $x=(-1,1,2)^{\top}$, then
$${\mathcal{T}}x^4= (x_1+x_2+x_3)^4-4(x_1^3x_2+x_1^3x_3)-8x_2x_3^3-24x_1^2x_2x_3-12(x_1x_2^2x_3+x_1x_2x_3^2)=-12<0.$$

\textbf{Case 2.} $t_{1223}=1$.

\textbf{Subcase 2.1} When $t_{1233}\neq-1$. Let $x=(-1,1,3)^{\top}$, then
$${\mathcal{T}}x^4\leq (x_1+x_2+x_3)^4-4(x_1^3x_2+x_1^3x_3)-8x_2x_3^3-12x_1x_2x_3^2=-11<0.$$

\textbf{Subcase 2.2} When $t_{1233}=-1$ and $t_{1123}\neq-1$. Let $x=(1,1,-1)^{\top}$, then
$${\mathcal{T}}x^4\leq (x_1+x_2+x_3)^4-4(x_1^3x_2+x_1^3x_3)-8x_2x_3^3-12x_1^2x_2x_3-24x_1x_2x_3^2=-3<0.$$

\textbf{Subcase 2.3} When $t_{1123}=t_{1233}=-1$. Let $x=(-2,2,1)^{\top}$, then
$${\mathcal{T}}x^4= (x_1+x_2+x_3)^4-4(x_1^3x_2+x_1^3x_3)-8x_2x_3^3-24(x_1^2x_2x_3+x_1x_2x_3^2)=-64<0.$$

($\romannumeral6$) We might take $t_{1222}=t_{1113}=0$ and $t_{1112}=t_{1333}=t_{2223}=-t_{2333}=1$.
Then for $x=(x_1,x_2,x_3)^{\top}\in{\mathbb{R}}^3$
\begin{align*}
{\mathcal{T}}x^4=&(x_1+x_2+x_3)^4-4(x_1x_2^3+x_1^3x_3)-8x_2x_3^3\\
&+12[(t_{1123}-1)x_1^2x_2x_3+(t_{1223}-1)x_1x_2^2x_3+(t_{1233}-1)x_1x_2x_3^2].
\end{align*}
For $x=(-1,1,4)^\top$, we have\begin{align*}
	{\mathcal{T}}x^4=-44+48(t_{1123}-t_{1223}-4t_{1233})>0.
\end{align*}
So, the following cases could not occur,
 \begin{itemize}
    \item $t_{1233}=1$;
    \item $t_{1233}=0$ and $t_{1123}\leq t_{1223}$.
\end{itemize}
Next, we discuss other situations.

\textbf{Case 1.} $t_{1233}=0$.

\textbf{Subcase 1.1} When $t_{1123}\neq-1$ and $t_{1223}=-1$. Let $x=(1,-3,1)^{\top}$, then
$${\mathcal{T}}x^4\leq (x_1+x_2+x_3)^4-4(x_1x_2^3+x_1^3x_3)-8x_2x_3^3-24x_1x_2^2x_3-12(x_1^2x_2x_3+x_1x_2x_3^2)=-15<0.$$

\textbf{Subcase 1.2} When $t_{1123}=1$ and $t_{1223}=0$. Let $x=(6,-4,1)^{\top}$,
$${\mathcal{T}}x^4= (x_1+x_2+x_3)^4-4(x_1x_2^3+x_1^3x_3)-8x_2x_3^3-12(x_1x_2^2x_3+x_1x_2x_3^2)=-79<0.$$

\textbf{Case 2.} $t_{1233}=-1$.

\textbf{Subcase 2.1} When $t_{1123}+t_{1223}\geq1$. Let $x=(1,1,-1)^{\top}$, then
$${\mathcal{T}}x^4= (x_1+x_2+x_3)^4-4(x_1x_2^3+x_1^3x_3)-8x_2x_3^3-12x_1^2x_2x_3-24x_1x_2x_3^2=-3<0.$$

\textbf{Subcase 2.2} When $t_{1123}=1$ and $t_{1223}=-1$, let $x=(6,-4,1)^{\top}$, then
$${\mathcal{T}}x^4= (x_1+x_2+x_3)^4-4(x_1x_2^3+x_1^3x_3)-8x_2x_3^3-24(x_1x_2^2x_3+x_1x_2x_3^2)=-943<0.$$

\textbf{Subcase 2.3} When $t_{1123}=-1$ and $t_{1223}\neq-1$. Let $x=(-6,3,1)^{\top}$, then
$${\mathcal{T}}x^4= (x_1+x_2+x_3)^4-4(x_1x_2^3+x_1^3x_3)-8x_2x_3^3-12x_1x_2^2x_3-24(x_1^2x_2x_3+x_1x_2x_3^2)=-8<0.$$

\textbf{Subcase 2.4} When $t_{1123}=t_{1223}=-1$. Let $x=(1,1,1)^{\top}$, then
$${\mathcal{T}}x^4= (x_1+x_2+x_3)^4-4(x_1x_2^3+x_1^3x_3)-8x_2x_3^3-24(x_1^2x_2x_3+x_1x_2^2x_3+x_1x_2x_3^2)=-7<0.$$

($\romannumeral7$) We might take $t_{1222}=t_{1113}=0$ and $t_{1112}=t_{1333}=-t_{2223}=t_{2333}=1$.
Then for $x=(x_1,x_2,x_3)^{\top}\in{\mathbb{R}}^3$
\begin{align*}
{\mathcal{T}}x^4=&(x_1+x_2+x_3)^4-4(x_1x_2^3+x_1^3x_3)-8x_2^3x_3\\
&+12[(t_{1123}-1)x_1^2x_2x_3+(t_{1223}-1)x_1x_2^2x_3+(t_{1233}-1)x_1x_2x_3^2].
\end{align*}
For $x=(1,1,-3)^\top$, we have\begin{align*}
	{\mathcal{T}}x^4=-3-36(t_{1123}+t_{1223}-3t_{1233})>0.
\end{align*}
So, the following cases could not occur,
 \begin{itemize}
    \item $t_{1233}=-1$;
    \item $t_{1233}=0$ and $t_{1123}+t_{1223}\geq 0$.
\end{itemize}
Next, we discuss other situations.

\textbf{Case 1.} $t_{1233}=0$.

\textbf{Subcase 1.1} When $t_{1223}=-1$ and $t_{1123}\neq1$. Let $x=(7,27,10)^{\top}$, then
$${\mathcal{T}}x^4\leq (x_1+x_2+x_3)^4-4(x_1x_2^3+x_1^3x_3)-8x_2^3x_3-24x_1x_2^2x_3-12(x_1^2x_2x_3+x_1x_2x_3^2)=-1668<0.$$

\textbf{Subcase 1.2} When $t_{1223}=0$ and $t_{1123}=-1$. Let $x=(-6,2,1)^{\top}$, then
$${\mathcal{T}}x^4\leq (x_1+x_2+x_3)^4-4(x_1x_2^3+x_1^3x_3)-8x_2^3x_3-24x_1^2x_2x_3-12(x_1x_2^2x_3+x_1x_2x_3^2)=-223<0.$$

\textbf{Case 2.} $t_{1233}=1$.

\textbf{Subcase 2.1} When $t_{1123}=-1$. Let $x=(-6,2,1)^{\top}$, then
$${\mathcal{T}}x^4\leq (x_1+x_2+x_3)^4-4(x_1x_2^3+x_1^3x_3)-8x_2^3x_3-24(x_1^2x_2x_3+x_1x_2^2x_3)=-79<0.$$

\textbf{Subcase 2.2} When $t_{1123}>t_{1223}$. Let $x=(1,-1,1)^{\top}$, then
$${\mathcal{T}}x^4\leq (x_1+x_2+x_3)^4-4(x_1x_2^3+x_1^3x_3)-8x_2^3x_3-12x_1x_2^2x_3=-3<0.$$

\textbf{Subcase 2.3} When $t_{1223}=1$. Let $x=(-1,3,1)^{\top}$, then
$${\mathcal{T}}x^4\leq (x_1+x_2+x_3)^4-4(x_1x_2^3+x_1^3x_3)-8x_2^3x_3=-23<0.$$

The necessity is proved.
\end{proof}

\begin{theorem}\label{th3.15}
Let ${\mathcal{T}}=(t_{ijkl})\in\widehat{{\mathcal{E}}}_{4,3}$ and $t_{iiij}t_{ijjj}=0$, $-t_{jjjk}t_{jkkk}=t_{iiik}t_{ikkk}=1$ for $i,j,k\in\{1,2,3\}$, $i\neq j$, $i\neq k$, $j\neq k$. Then $\mathcal{T}$ is positive semi-definite if and only if
$t_{1122}=t_{2233}=t_{1133}=t_{iijk}t_{iiij}t_{iiik}=t_{ijkk}t_{ikkk}t_{jkkk}=t_{iiij}t_{iiik}t_{jkkk}=1$ and $t_{ijjk}=t_{ijjj}=0$.
\end{theorem}
\begin{proof}
``\textbf{if (Sufficiency)}."
Assume $t_{iiij}=t_{iiik}=t_{ikkk}=t_{jkkk}=-t_{jjjk}=1$. Without loss the generality, we might take $t_{1222}=0$, and $t_{1112}=t_{1333}=t_{1113}=t_{2333}=-t_{2223}=1$ then $t_{1122}=t_{2233}=t_{1133}=t_{1123}=t_{1233}=1$ and $t_{1223}=0$. Then for $x=(x_1,x_2,x_3)^{\top}\in{\mathbb{R}}^3$,
\begin{equation}\label{-1-0-1}
{\mathcal{T}}x^4=(x_1+x_2+x_3)^4-4x_1x_2^3-8x_2^3x_3-12x_1x_2^2x_3.
\end{equation}
When $x_2=0$, by Lemma \ref{lem:23}, ${\mathcal{T}}x^4\geq0$, with equality if and only if $x_1=x_3=0$.

When $x_2\neq0$, let $y_1:=\frac{x_1}{x_2}$ and $y_3:=\frac{x_3}{x_2}$, then
$${\mathcal{T}}x^4=x_2^4[(y_1+1+y_3)^4-4y_1-8y_3-12y_1y_3:=x_2^4g(y_1,y_3).$$
Thus, to prove ${\mathcal{T}}x^4\geq0$ for any $x=(x_1,x_2,x_3)^{\top}\in{\mathbb{R}}^3$, we just need to prove $g(y_1,y_3)\geq0$ for any $y=(y_1,y_3)^{\top}\in{\mathbb{R}}^2$. Let $\triangle g(y_1,y_3)=0$, i.e.,
$$\left\{
\begin{array}{llll}
4[(y_1+1+y_3)^3-1-3y_3]=0, \\
4[(y_1+1+y_3)^3-2-3y_1]=0.
\end{array}
\right.$$
$\left\{
\begin{array}{llll}
y_1=-\frac{2}{3} \\
y_3=-\frac{1}{3}
\end{array}
\right.$,
$\left\{
\begin{array}{llll}
y_1=-\frac{8+3\sqrt{6}}{12}\\
y_3=-\frac{4+3\sqrt{6}}{12}
\end{array}
\right.$ and
$\left\{
\begin{array}{llll}
y_1=-\frac{8-3\sqrt{6}}{12}\\
y_3=-\frac{4-3\sqrt{6}}{12}
\end{array}
\right.$
are solutions of $\triangle g(y_1,y_3)=0$, and the Hessian matrixes of $g$ are positive definite at $(-\frac{8+3\sqrt{6}}{12},-\frac{4+3\sqrt{6}}{12})$ and $(-\frac{8-3\sqrt{6}}{12},-\frac{4-3\sqrt{6}}{12})$, the Hessian matrixes of $g$ is positive semi-definite at $(-\frac{2}{3},-\frac{1}{3})$. Therefore, $(-\frac{8+3\sqrt{6}}{12},-\frac{4+3\sqrt{6}}{12})$ and $(-\frac{8-3\sqrt{6}}{12},-\frac{4-3\sqrt{6}}{12})$ are local minimum points of $g$, and so, they are global also. Since $g(-\frac{2}{3},-\frac{1}{3})>g(-\frac{8+3\sqrt{6}}{12},-\frac{4+3\sqrt{6}}{12})=g(-\frac{8-3\sqrt{6}}{12},-\frac{4-3\sqrt{6}}{12})>0$. Then $g(y_1,y_3)>0$ for any $y=(y_1,y_3)^{\top}\in{\mathbb{R}}^2$, which implies ${\mathcal{T}}$ is positive semi-definite.

All other cases which satisfying
$t_{iijj}=t_{jjkk}=t_{iikk}=t_{iijk}t_{iiij}t_{iiik}=t_{ijkk}t_{ikkk}t_{jkkk}=t_{iiij}t_{iiik}t_{jkkk}=1$ and $t_{ijjk}=0$ when $t_{ijjj}=0$ and $t_{iiik}t_{ikkk}=-t_{jkkk}t_{jjjk}=1$, can be transformed into similar forms of $(\ref{-1-0-1})$.

``\textbf{only if (Necessity)}." The conditions,
$t_{iiij}t_{ijjj}=0$, $-t_{jjjk}t_{jkkk}=t_{iiik}t_{ikkk}=1$ for $i,j,k\in\{1,2,3\}$, $i\neq j$, $i\neq k$, $j\neq k$, can be divided into five cases, i.e.,
\begin{itemize}
    \item [($\romannumeral1$)] $t_{iiij}=t_{ijjj}=0$;
	\item [($\romannumeral2$)] $t_{ijjj}=0$ and $t_{iiij}t_{jkkk}t_{ikkk}=1$;
    \item [($\romannumeral3$)] $t_{ijjj}=0$ and $t_{iiij}t_{jkkk}t_{ikkk}=-1$;
    \item [($\romannumeral4$)] $t_{iiij}=0$ and $t_{ijjj}t_{jjjk}t_{iiik}=1$;
    \item [($\romannumeral5$)] $t_{iiij}=0$ and $t_{ijjj}t_{jjjk}t_{iiik}=-1$.
\end{itemize}
Similar to the prove of Theorem \ref{th3.3}, we only need to consider $t_{iiik}=t_{jkkk}=1$, $t_{iiij}=t_{jkkk}=t_{ikkk}=1$, $t_{iiij}=-t_{jkkk}=t_{ikkk}=1$, $t_{ijjj}=t_{jjjk}=t_{iiik}=1$ and $t_{ijjj}=-t_{jjjk}=t_{iiik}=-1$ respectively.
Without loss the generality, suppose $t_{1112}t_{1222}=0$, $-t_{2223}t_{2333}=t_{1113}t_{1333}=1$. And it follows from the positive semi-definiteness of $\mathcal{T}$ with Eq. \eqref{eq:3} that \begin{center}
	$t_{1133}=t_{2233}=1$ and $t_{11122}\in\{0,1\}$, for ($\romannumeral1$);\\
$t_{iijj}=1$ for all $i,j\in\{1,2,3\}$ and $i\neq j$, for others.
\end{center}

($\romannumeral1$) We might take $t_{1112}=t_{1222}=0$, and $t_{1113}=t_{2333}=1$.
Then for $x=(x_1,x_2,x_3)^{\top}\in{\mathbb{R}}^3$
\begin{align*}
{\mathcal{T}}x^4=&(x_1+x_2+x_3)^4-4(x_1^3x_2+x_1x_2^3)-8x_2^3x_3\\
&+12[(t_{1123}-1)x_1^2x_2x_3+(t_{1223}-1)x_1x_2^2x_3+(t_{1233}-1)x_1x_2x_3^2]+6(t_{1122}-1)x_1^2x_2^2.
\end{align*}
For $x=(1,1,-4)^\top$, we have\begin{align*}
	{\mathcal{T}}x^4=-62-48(t_{1123}+t_{1223}-4t_{1233})+6t_{1122}\geq0,
\end{align*}
So, the following cases  could not occur, \begin{itemize}
	\item $t_{1233}=-1$;
	\item $t_{1233}=0$ and $t_{1123}+t_{1223}\geq-1$.
\end{itemize}
Next, we discuss other situations.

\textbf{Case 1.} $t_{1233}=0$ and and $t_{1123}=t_{1223}=-1$. Let $x=(-5,1,4)^{\top}$, then
$${\mathcal{T}}x^4\leq (x_1+x_2+x_3)^4-4(x_1^3x_2+x_1x_2^3)-8x_2^3x_3-12x_1x_2x_3^2-24(x_1^2x_2x_3+x_1x_2^2x_3)=-472<0.$$

\textbf{Case 2.} $t_{1233}=1$

\textbf{Subcase 2.1} When $t_{1123}\neq1$. Let $x=(-3,1,4)^{\top}$, then
$${\mathcal{T}}x^4\leq (x_1+x_2+x_3)^4-4(x_1^3x_2+x_1x_2^3)-8x_2^3x_3-12x_1^2x_2x_3-24x_1x_2^2x_3=-40<0.$$

\textbf{Subcase 2.2} When $t_{1123}=1$ and $t_{1223}\neq-1$. Let $x=(5,1,-2)^{\top}$, then
$${\mathcal{T}}x^4\leq (x_1+x_2+x_3)^4-4(x_1^3x_2+x_1x_2^3)-8x_2^3x_3-12x_1x_2^2x_3=-120<0.$$

\textbf{Subcase 2.3} When $t_{1123}=1$ and $t_{1223}=-1$, let $x=(1,-1,1)^{\top}$, then
$${\mathcal{T}}x^4\leq (x_1+x_2+x_3)^4-4(x_1^3x_2+x_1x_2^3)-8x_2^3x_3-24x_1x_2^2x_3=-7<0.$$

($\romannumeral2$) We might take $t_{1222}=0$ and $t_{1112}=t_{2333}=t_{1333}=1$.
Then for $x=(x_1,x_2,x_3)^{\top}\in{\mathbb{R}}^3$
\begin{align*}
{\mathcal{T}}x^4=&(x_1+x_2+x_3)^4-4x_1x_2^3-8x_2^3x_3\\
&+12[(t_{1123}-1)x_1^2x_2x_3+(t_{1223}-1)x_1x_2^2x_3+(t_{1233}-1)x_1x_2x_3^2].
\end{align*}
For $x=(-5,2,1)^\top$, we have\begin{align*}
	{\mathcal{T}}x^4=-128+120(5t_{1123}-2t_{1223}-t_{1233})\geq0,
\end{align*}
So, the following cases  could not occur, \begin{itemize}
	\item $t_{1123}=-1$;
    \item $t_{1123}=0$ and $t_{1223}\neq-1$;
	\item $t_{1123}=0$ and $t_{1233}=1$.
\end{itemize}
Next, we discuss other situations.

\textbf{Case 1.} $t_{1123}=0$, $t_{1223}=-1$ and $t_{1233}\neq1$. Let $x=(3,1,-6)^{\top}$, then
$${\mathcal{T}}x^4\leq (x_1+x_2+x_3)^4-4x_1x_2^3-8x_2^3x_3-24x_1x_2^2x_3-12(x_1^2x_2x_3+x_1x_2x_3^2)=-164<0.$$

\textbf{Case 2.} $t_{1123}=1$.

\textbf{Subcase 2.1} When $t_{1233}\neq1$. Let $x=(3,1,-6)^{\top}$, then
$${\mathcal{T}}x^4\leq (x_1+x_2+x_3)^4-4x_1x_2^3-8x_2^3x_3-24x_1x_2^2x_3-12x_1x_2x_3^2=-812<0.$$

\textbf{Subcase 2.2} When $t_{1223}=t_{1233}=1$. Let $x=(-1,1,1)^{\top}$, then
$${\mathcal{T}}x^4= (x_1+x_2+x_3)^4-4x_1x_2^3-8x_2^3x_3=-3<0.$$

\textbf{Subcase 2.3} When $t_{1223}=-1$ and $t_{1233}=1$. Let $x=(1,-1,1)^{\top}$, then
$${\mathcal{T}}x^4= (x_1+x_2+x_3)^4-4x_1x_2^3-8x_2^3x_3-24x_1x_2^2x_3=-11<0.$$

($\romannumeral3$) We might take $t_{1222}=0$ and $t_{1112}=-t_{2333}=t_{1333}=1$.
Then for $x=(x_1,x_2,x_3)^{\top}\in{\mathbb{R}}^3$
\begin{align*}
{\mathcal{T}}x^4=&(x_1+x_2+x_3)^4-4x_1x_2^3-8x_2x_3^3\\
&+12[(t_{1123}-1)x_1^2x_2x_3+(t_{1223}-1)x_1x_2^2x_3+(t_{1233}-1)x_1x_2x_3^2].
\end{align*}
For $x=(-5,1,12)^\top$, we have\begin{align*}
	{\mathcal{T}}x^4=-3948+720(5t_{1123}-t_{1223}-12t_{1233})\geq0,
\end{align*}
So, the following cases  could not occur, \begin{itemize}
	\item $t_{1233}=1$;
    \item $t_{1233}=0$ and $t_{1123}\neq1$;
	\item $t_{1233}=0$ and $t_{1223}\neq-1$.
\end{itemize}
Next, we discuss other situations.

\textbf{Case 1.} $t_{1233}=0$, $t_{1123}=-1$ and $t_{1223}=-1$. Let $x=(2,-5,2)^{\top}$, then
$${\mathcal{T}}x^4= (x_1+x_2+x_3)^4-4x_1x_2^3-8x_2x_3^3-24x_1x_2^2x_3-12x_1x_2x_3^2=-599<0.$$

\textbf{Case 2.} $t_{1233}=-1$.

\textbf{Subcase 2.1} When $t_{1123}\neq-1$. Let $x=(5,1,-6)^{\top}$, then
$${\mathcal{T}}x^4\leq (x_1+x_2+x_3)^4-4x_1x_2^3-8x_2x_3^3-12x_1^2x_2x_3-24(x_1x_2^2x_3+x_1x_2x_3^2)=-92<0.$$

\textbf{Subcase 2.2} When $t_{1123}=-1$ and $t_{1223}\neq-1$. Let $x=(-10,5,1)^{\top}$, then
$${\mathcal{T}}x^4\leq (x_1+x_2+x_3)^4-4x_1x_2^3-8x_2x_3^3-12x_1x_2^2x_3-24(x_1^2x_2x_3+x_1x_2x_3^2)=-2584<0.$$

\textbf{Subcase 2.3} When $t_{1123}=t_{1223}=-1$. Let $x=(1,1,1)^{\top}$, then
$${\mathcal{T}}x^4= (x_1+x_2+x_3)^4-4x_1x_2^3-8x_2x_3^3-24(x_1^2x_2x_3+x_1x_2^2x_3+x_1x_2x_3^2)=-3<0.$$

($\romannumeral4$) We might take $t_{1112}=0$ and $t_{1222}=t_{2223}=t_{1113}=1$.
Then for $x=(x_1,x_2,x_3)^{\top}\in{\mathbb{R}}^3$
\begin{align*}
{\mathcal{T}}x^4=&(x_1+x_2+x_3)^4-4x_1^3x_2-8x_2x_3^3\\
&+12[(t_{1123}-1)x_1^2x_2x_3+(t_{1223}-1)x_1x_2^2x_3+(t_{1233}-1)x_1x_2x_3^2].
\end{align*}
For $x=(2,-5,1)^\top$, we have\begin{align*}
	{\mathcal{T}}x^4=-24-120(2t_{1123}-5t_{1223}+t_{1233})\geq0,
\end{align*}
So, the following cases could not occur, \begin{itemize}
	\item $t_{1223}=-1$;
    \item $t_{1223}=0$ and $t_{1123}=1$;
	\item $t_{1223}=0$, $t_{1123}=0$ and $t_{1233}\neq-1$.
\end{itemize}
Next, we discuss other situations.

\textbf{Case 1.} $t_{1223}=0$.

\textbf{Subcase 1.1} When $t_{1123}=0$ and $t_{1233}=-1$. Let $x=(6,1,-3)^{\top}$, then
$${\mathcal{T}}x^4= (x_1+x_2+x_3)^4-4x_1^3x_2-8x_2x_3^3-24x_1x_2x_3^2-12(x_1^2x_2x_3+x_1x_2^2x_3)=-176<0.$$

\textbf{Subcase 1.2} When $t_{1123}=-1$. Let $x=(-3,1,1)^{\top}$, then
$${\mathcal{T}}x^4\leq (x_1+x_2+x_3)^4-4x_1^3x_2-8x_2x_3^3-12x_1x_2^2x_3-24(x_1^2x_2x_3+x_1x_2x_3^2)=-7<0.$$

\textbf{Case 2.} $t_{1223}=1$.

\textbf{Subcase 2.1} When $t_{1123}\leq t_{1233}$. Let $x=(-1,1,1)^{\top}$, then
$${\mathcal{T}}x^4= (x_1+x_2+x_3)^4-4x_1^3x_2-8x_2x_3^3=-3<0.$$

\textbf{Subcase 2.2} When $t_{1123}> t_{1233}$. Let $x=(1,1,-1)^{\top}$, then
$${\mathcal{T}}x^4= (x_1+x_2+x_3)^4-4x_1^3x_2-8x_2x_3^3-12x_1x_2x_3^2=-7<0.$$

($\romannumeral5$) We might take $t_{1112}=0$ and $t_{1222}=-t_{2223}=t_{1333}=1$.
Then for $x=(x_1,x_2,x_3)^{\top}\in{\mathbb{R}}^3$
\begin{align*}
{\mathcal{T}}x^4=&(x_1+x_2+x_3)^4-4x_1^3x_2-8x_2^3x_3\\
&+12[(t_{1123}-1)x_1^2x_2x_3+(t_{1223}-1)x_1x_2^2x_3+(t_{1233}-1)x_1x_2x_3^2].
\end{align*}
For $x=(4,1,-6)^\top$, we have\begin{align*}
	{\mathcal{T}}x^4=-495-288(4t_{1123}+t_{1223}-6t_{1233})\geq0,
\end{align*}
So, the following cases could not occur, \begin{itemize}
	\item $t_{1233}=-1$;
    \item $t_{1233}=0$ and $t_{1123}\neq-1$;
	\item $t_{1123}=t_{1223}=t_{1233}=1$.
\end{itemize}
Next, we discuss other situations.

\textbf{Case 1.} $t_{1233}=0$ and $t_{1123}=-1$.

\textbf{Subcase 1.1} When $t_{1223}\neq-1$. Let $x=(-1,1,1)^{\top}$, then
$${\mathcal{T}}x^4\leq (x_1+x_2+x_3)^4-4x_1^3x_2-8x_2^3x_3-24x_1^2x_2x_3-12(x_1x_2^2x_3+x_1x_2x_3^2)=-3<0.$$

\textbf{Subcase 1.2} When $t_{1223}=-1$. Let $x=(-5,1,2)^{\top}$, then
$${\mathcal{T}}x^4= (x_1+x_2+x_3)^4-4x_1^3x_2-8x_2^3x_3-12x_1x_2x_3^2-24(x_1^2x_2x_3+x_1x_2^2x_3)=-220<0.$$

\textbf{Case 2.} $t_{1233}=1$ and at least one of $\{t_{1123},t_{1223}\}$ is not $1$.

\textbf{Subcase 2.1} When $t_{1123}\leq t_{1223}$. Let $x=(-1,1,1)^{\top}$, then
$${\mathcal{T}}x^4\leq (x_1+x_2+x_3)^4-4x_1^3x_2-8x_2^3x_3-12(x_1^2x_2x_3+x_1x_2^2x_3)=-3<0.$$

\textbf{Subcase 2.2} When $t_{1123}=1$ and $t_{1223}=0$. Let $x=(5,1,-2)^{\top}$, then
$${\mathcal{T}}x^4= (x_1+x_2+x_3)^4-4x_1^3x_2-8x_2^3x_3-12x_1x_2^2x_3=-108<0.$$

\textbf{Subcase 2.3} When $t_{1223}=-1$. Let $x=(3,-5,3)^{\top}$, then
$${\mathcal{T}}x^4\leq (x_1+x_2+x_3)^4-4x_1^3x_2-8x_2^3x_3-12x_1^2x_2x_3-24x_1x_2^2x_3=-239<0.$$

The necessity is proved.
\end{proof}

\begin{remark}\label{r1}
Let ${\mathcal{T}}=(t_{ijkl})\in\widehat{{\mathcal{E}}}_{4,3}$. When $t_{iiij}t_{ijjj}=t_{jjjk}t_{jkkk}=-t_{iiik}t_{ikkk}=1$, $t_{iiij}t_{ijjj}=t_{jjjk}t_{jkkk}=t_{iiik}t_{ikkk}=-1$, and $t_{iiij}t_{ijjj}=t_{jjjk}t_{jkkk}=1$ with $t_{iiik}t_{ikkk}=0$
for $i,j,k\in\{1,2,3\}$, $i\neq j$, $i\neq k$, $j\neq k$, there are some $x=(x_1,x_2,x_3)^{\top}\in{\mathbb{R}}^3$ such that ${\mathcal{T}}x^4<0$, i.e., $\mathcal{T}$ are not positive semi-definite under these conditions.
\end{remark}

Thee following conclusions can obtained from Theorem \ref{th3.1}, Corollary \ref{cor1}, Theorem \ref{th3.8}, Theorem \ref{th3.9}, Theorem \ref{th3.10}, Theorem \ref{th3.13}, and Theorem \ref{th3.14}.

\begin{corollary}\label{cor4}
Let ${\mathcal{T}}=(t_{ijkl})\in{\mathcal{E}}_{4,3}$ and $t_{iiii}=0$, $t_{jjjj}=t_{kkkk}=1$ for $i,j,k\in\{1,2,3\}$, $i\neq j$, $j\neq k$, $i\neq k$. Then $\mathcal{T}$ is positive semi-definite if and only if $t_{iiij}=t_{iiik}=0$ and one of the following conditions is satisfied.
\begin{itemize}
\item [(a)] $t_{ijjj}=t_{jjjk}=t_{ikkk}=t_{jkkk}=0$, and
\begin{itemize}
  \item [(a$_1$)] $t_{1123}=t_{1223}=t_{1233}=t_{iijj}=t_{jjkk}=t_{iikk}=1$, or
  \item [(a$_2$)] $t_{1123}=t_{1223}=t_{1233}=0$, and $t_{iijj},t_{jjkk},t_{iikk}\in\{0,1\}$, or
  \item [(a$_3$)] Two of $\{t_{1123},t_{1223},t_{1233}\}$ are $-1$, and the other one is $1$, $t_{1122}=t_{2233}=t_{1133}=1$, or
  \item [(a$_4$)] $t_{iijk}=\pm1$, $t_{ijjk}=t_{ijkk}=0$, the other one is $1$ or $-1$, and $t_{iijj}=t_{iikk}=1$ and $t_{jjkk}\in\{0,1\}$.
\end{itemize}
\item [(b)] $t_{jjjk}=0$, $t_{iijj}=t_{jjkk}=t_{iikk}=1$,$t_{iijk}=t_{ijjk}=0$ and $t_{ijjj}t_{jkkk}t_{ikkk}=t_{ijkk}t_{ijjj}=\pm1$.
\item [(c)] $t_{jkkk}=\pm1$, $t_{ijjj}=t_{jjjk}=t_{ikkk}=0$, and
\begin{itemize}
\item [(c$_1$)] $t_{1123}=t_{1223}=t_{1233}=0$, and $t_{jjkk}=1$, $t_{iikk},t_{iijj}\in\{0,1\}$, or
\item [(c$_2$)] $t_{iijk}=\pm1$, $t_{ijkk}=t_{ijjk}=0$ and $t_{1122}=t_{2233}=t_{1133}=1$, or
\item [(c$_3$)] $t_{iijk}=t_{ijkk}=0$, $t_{ijjk}=\pm1$ and $t_{1122}=t_{2233}=t_{1133}=1$.
\end{itemize}
\item [(d)] $t_{ijjj}=t_{jjjk}=0$, $t_{ikkk},t_{jkkk}\in\{-1,1\}$, $t_{ikkk}t_{jkkk}t_{ijkk}=t_{iikk}=t_{jjkk}=1$, $t_{iijk}=t_{ijjk}=0$, and $t_{iijj}\in\{0,1\}$.
\item [(e)]  $t_{ijjj},t_{ikkk}\in\{-1,1\}$, $t_{jjjk}=t_{jkkk}=0$ and
  \begin{itemize}
    \item [(e$_1$)] $t_{1123}=t_{1223}=t_{1233}=0$, and $t_{jjkk}\in\{0,1\}$, $t_{iijj}=t_{iikk}=1$, or
    \item [(e$_2$)] $t_{ijkk}=t_{ijjk}=0$, $t_{iijk}t_{ijjj}t_{ikkk}=1$, and $t_{1122}=t_{2233}=t_{1133}=1$.
  \end{itemize}
\item [(f)] $t_{ijjj},t_{jkkk}\in\{-1,1\}$ $t_{ikkk}=t_{jjjk}=t_{1123}=t_{1223}=t_{1233}=0$, and $t_{iikk}\in\{0,1\}$, $t_{iijj}=t_{jjkk}=1$.
\item [(g)]  $t_{jjjk}=t_{jkkk}=\pm1$, and
  \begin{itemize}
    \item [(g$_1$)] $t_{ijjj}=t_{ikkk}=t_{1123}=t_{1223}=t_{1233}=0$, $t_{jjkk}=1$ and $t_{iijj},t_{iikk}\in\{0,1\}$, or
    \item [(g$_2$)] $t_{iijk}=\pm1$, $t_{ijjj}=t_{ikkk}=t_{ijjk}=t_{ijkk}=0$, and $t_{1122}=t_{2233}=t_{1133}=1$, or
    \item [(g$_3$)] $t_{ijjj}t_{ijkk}=t_{ikkk}t_{ijjk}=t_{ijjj}t_{ikkk}t_{iijk}=t_{ikkk}t_{jjjk}t_{ijkk}=t_{ijjj}t_{jjjk}t_{ijjk}=1$ and $t_{1122}=t_{2233}=t_{1133}=1$.
      \end{itemize}
\item [(h)]  $t_{jjjk}=-t_{jkkk}=\pm1$  and
  \begin{itemize}
    \item [(h$_1$)] $t_{ijjj}=t_{ikkk}=t_{1123}=t_{1223}=t_{1233}=0$, $t_{jjkk}=1$ and $t_{iijj},t_{iikk}\in\{0,1\}$, or
    \item [(h$_2$)] $t_{iijk}=\pm1$, $t_{ijjj}=t_{ikkk}=t_{ijjk}=t_{ijkk}=0$, and $t_{1122}=t_{2233}=t_{1133}=1$, or
    \item [(h$_3$)]$t_{ijjj}=\pm1$, $t_{ikkk}=t_{iijk}=t_{ijkk}=0$, $t_{ijjj}t_{jjjk}t_{ijjk}=t_{1122}=t_{2233}=t_{1133}=1$.
  \end{itemize}
\end{itemize}
\end{corollary}
\begin{proof}
``\textbf{if (Sufficiency)}." Sufficiency can be obtained from Theorem \ref{th3.1}, Corollary \ref{cor1}, Theorem \ref{th3.8}, Theorem \ref{th3.9}, Theorem \ref{th3.10}, Theorem \ref{th3.13}, Theorem \ref{th3.14} and their proof procedures.

``\textbf{only if (Necessity)}." At the same time, according Theorem \ref{th3.1}, Corollary \ref{cor1}, Theorem \ref{th3.8}, Theorem \ref{th3.9}, Theorem \ref{th3.10}, Theorem \ref{th3.13}, Theorem \ref{th3.14}, necessity can be obtained only by proving there are some points $x=(x_1,x_2,x_3)^{\top}\in{\mathbb{R}}^3$ such that ${\mathcal{T}}x^4<0$ under the following situations.
\begin{itemize}
  \item [($\romannumeral1 )$] $t_{ijjj}t_{jkkk}t_{ikkk}=-t_{iijk}t_{jkkk}=-t_{ijjk}t_{ikkk}=t_{1122}=t_{2233}=t_{1133}=1$ and $t_{jjjk}=t_{ijkk}=0$.
  \item [($\romannumeral2$)] $t_{ikkk},t_{jkkk}\in\{-1,1\}$, $t_{ijjj}=t_{jjjk}=0$, and $t_{ikkk}t_{jkkk}t_{ijkk}=t_{ikkk}t_{ijjk}=t_{jkkk}t_{iijk}=t_{1122}=t_{2233}=t_{1133}=1$.
  \item [($\romannumeral3$)] $t_{ijjj},t_{ikkk}\in\{-1,1\}$, $t_{jjjk}=t_{jkkk}=t_{ijjk}=t_{ijkk}=0$, and $-t_{ikkk}t_{ijjj}t_{iijk}=t_{1122}=t_{2233}=t_{1133}=1$.
  \item [($\romannumeral4$)] $t_{ijjj},t_{ikkk}\in\{-1,1\}$, $t_{jjjk}=t_{jkkk}=t_{ijkk}=0$, and $-t_{ikkk}t_{ijjj}t_{iijk}=-t_{ijjk}t_{ikkk}=t_{1122}=t_{2233}=t_{1133}=1$.
  \item [($\romannumeral5$)] $t_{ijjj},t_{ikkk}\in\{-1,1\}$, $t_{jjjk}=t_{jkkk}=t_{ijjk}=0$, and $-t_{ikkk}t_{ijjj}t_{iijk}=-t_{ijkk}t_{ijjj}=t_{1122}=t_{2233}=t_{1133}=1$.
  \item [($\romannumeral6$)] $t_{ijjj},t_{jkkk},t_{iijk}\in\{-1,1\}$, $t_{ikkk}=t_{jjjk}=t_{ijkk}=t_{ijjk}=0$, and $t_{1122}=t_{2233}=t_{1133}=1$.
  \item [($\romannumeral7$)] $t_{ijjj},t_{jkkk}\in\{-1,1\}$, $t_{ikkk}=t_{jjjk}=t_{ijkk}=0$, and $-t_{jkkk}t_{ijjj}t_{ijjk}=-t_{iijk}t_{jkkk}=t_{1122}=t_{2233}=t_{1133}=1$.
  \item [($\romannumeral8$)] $t_{jjjk}=t_{jkkk}=\pm1$, $t_{ijjj}=t_{ikkk}=t_{ijjk}=t_{ijkk}=0$, and $-t_{iijk}t_{jjjk}=t_{jjkk}=t_{iijj}+t_{iikk}=1$.
  \item [($\romannumeral9$)] $t_{jjjk}=-t_{jkkk}=\pm1$, $t_{ijjj}=\pm1$, $t_{ikkk}=t_{ijjk}=t_{ijkk}=0$, $t_{iijk}t_{jkkk}=t_{1122}=t_{2233}=t_{1133}=1$.
  \item [($\romannumeral10$)] $t_{jjjk}=-t_{jkkk}=\pm1$, $t_{jjjk}=t_{ijjj}t_{ikkk}=-t_{iijk}=-t_{ijkk}t_{ijjj}=t_{1122}=t_{2233}=t_{1133}=1$, and $t_{ijjk}=0$.
  \item [($\romannumeral11$)] $t_{jjjk}=-t_{jkkk}=\pm1$, $t_{jjjk}=-t_{ijjj}t_{ikkk}=t_{iijk}=t_{ijjk}t_{ijjj}=t_{1122}=t_{2233}=t_{1133}=1$, and $t_{ijkk}=0$.
\end{itemize}
Then we only need to consider $t_{ijjj}=t_{jkkk}=t_{ikkk}=1$, $t_{ikkk}=t_{jkkk}=1$, $t_{ijjj}=t_{ikkk}=1$, $t_{ijjj}=t_{ikkk}=1$, $t_{ijjj}=t_{ikkk}=1$, $t_{ijjj}=t_{jkkk}=\pm t_{iijk}=1$, $t_{ijjj}=t_{jkkk}=1$, $t_{jjjk}=t_{jkkk}=1$, $t_{ijjj}=t_{jjjk}=-t_{jkkk}=1$, $t_{ijjj}=t_{jjjk}=-t_{jkkk}=1$ and $t_{ijjj}=t_{jjjk}=-t_{jkkk}=1$, respectively. Without loss the generality, suppose $t_{1111}=t_{1112}=t_{1113}=0$.

($\romannumeral1$) $t_{1222}=t_{2333}=t_{1333}=-t_{1123}=-t_{1223}=t_{1122}=t_{2233}=t_{1133}=1$ and $t_{2223}=t_{1233}=0$. Then for $x=(5,1,1)^{\top}$,
$${\mathcal{T}}x^4\leq (x_1+x_2+x_3)^4-4(x_1^3x_2+x_1^3x_3+x_2^3x_3)-x_1^4-24(x_1^2x_2x_3+x_1x_2^2x_3)-12x_1x_2x_3^2=-8<0.$$

($\romannumeral2$) $t_{2223}=t_{1222}=0$ and $t_{2333}=t_{1333}=t_{1123}=t_{1223}=t_{1233}=t_{1122}=t_{2233}=t_{1133}=1$. Then for $x=(2,1,-1)^{\top}$,
$${\mathcal{T}}x^4\leq (x_1+x_2+x_3)^4-4(x_1^3x_2+x_1^3x_3+x_1x_2^3+x_2^3x_3)-x_1^4=-4<0.$$

($\romannumeral3$) $t_{2223}=t_{2333}=t_{1223}=t_{1233}=0$ and $-t_{1123}=t_{1222}=t_{1333}=t_{1122}=t_{2233}=t_{1133}=1$. Then for $x=(-4,1,1)^{\top}$,
$${\mathcal{T}}x^4\leq x_2^4+x_3^4+4(x_1x_2^3+x_1x_3^3)-12x_1^2x_2x_3+6(x_1^2x_2^2+x_2^2x_3^2+x_1^2x_3^2)=-24<0.$$

($\romannumeral4$) $t_{2223}=t_{2333}=t_{1233}=0$ and $t_{1222}=t_{1333}=-t_{1123}=-t_{1223}=t_{1122}=t_{2233}=t_{1133}=1$. Then for $x=(5,1,1)^{\top}$,
$${\mathcal{T}}x^4\leq x_2^4+x_3^4+4(x_1x_2^3+x_1x_3^3)-12(x_1^2x_2x_3+x_1x_2^2x_3)+6(x_1^2x_2^2+x_2^2x_3^2+x_1^2x_3^2)=-12<0.$$

($\romannumeral5$) $t_{2223}=t_{2333}=t_{1223}=0$ and $t_{1222}=t_{1333}=-t_{1123}=-t_{1233}=t_{1122}=t_{2233}=t_{1133}=1$. Then for $x=(5,1,1)^{\top}$,
$${\mathcal{T}}x^4\leq x_2^4+x_3^4+4(x_1x_2^3+x_1x_3^3)-12(x_1^2x_2x_3+x_1x_2x_3^2)+6(x_1^2x_2^2+x_2^2x_3^2+x_1^2x_3^2)=-12<0.$$

($\romannumeral6$) $t_{1333}=t_{2223}=t_{1223}=t_{1233}=0$ and $t_{1222}=t_{2333}=t_{1123}=t_{1122}=t_{2233}=t_{1133}=1$. Then for $x=(2,-1,1)^{\top}$,
$${\mathcal{T}}x^4\leq x_2^4+x_3^4+4(x_1x_2^3+x_2x_3^3)+12x_1^2x_2x_3+6(x_1^2x_2^2+x_2^2x_3^2+x_1^2x_3^2)=-4<0.$$
If $t_{1333}=t_{2223}=t_{1123}=t_{1223}=0$ and $t_{1222}=t_{2333}=-t_{1123}=t_{1122}=t_{2233}=t_{1133}=1$. Then for $x=(-4,1,1)^{\top}$,
$${\mathcal{T}}x^4\leq x_2^4+x_3^4+4(x_1x_2^3+x_2x_3^3)-12x_1^2x_2x_3+6(x_1^2x_2^2+x_2^2x_3^2+x_1^2x_3^2)=-4<0.$$

($\romannumeral7$) $t_{2223}=t_{1333}=t_{1233}=0$ and $t_{1222}=t_{2333}=-t_{1123}=-t_{1223}=t_{1122}=t_{2233}=t_{1133}=1$. Then for $x=(5,1,1)^{\top}$,
$${\mathcal{T}}x^4\leq x_2^4+x_3^4+4(x_1x_2^3+x_2x_3^3)-12(x_1^2x_2x_3+x_1x_2^2x_3)+6(x_1^2x_2^2+x_2^2x_3^2+x_1^2x_3^2)=-28<0.$$

($\romannumeral8$) $t_{1222}=t_{1333}=t_{1223}=t_{1233}=t_{1122}=0$ and $t_{2223}=t_{2333}=-t_{1123}=t_{2233}=t_{1133}=1$. Then for $x=(5,1,1)^{\top}$,
$${\mathcal{T}}x^4\leq x_2^4+x_3^4+4(x_2^3x_3+x_2x_3^3)-12x_1^2x_2x_3+6(x_2^2x_3^2+x_1^2x_3^2)=-134<0.$$
The conditions that $t_{1122}=1$ and $t_{1133}=0$ is similar.

($\romannumeral9$) $t_{1333}=t_{1223}=t_{1233}=0$ and $t_{1222}=t_{2223}=-t_{2333}=-t_{1123}=t_{1122}=t_{2233}=t_{1133}=1$. Then for $x=(-4,1,1)^{\top}$,
$${\mathcal{T}}x^4\leq x_2^4+x_3^4+4(x_1x_2^3+x_2^3x_3-x_2x_3^3)-12x_1^2x_2x_3+6(x_1^2x_2^2+x_2^2x_3^2+x_1^2x_3^2)=-8<0.$$

($\romannumeral10$) $t_{1223}=0$, and $t_{1222}=t_{1333}=t_{2223}=-t_{2333}=-t_{1123}=-t_{1233}=t_{1122}=t_{2233}=t_{1133}=1$. Then for $x=(5,1,1)^{\top}$,
$${\mathcal{T}}x^4\leq x_2^4+x_3^4+4(x_1x_2^3+x_1x_3^3+x_2^3x_3-x_2x_3^3)-12(x_1^2x_2x_3+x_1x_2x_3^2)+6(x_1^2x_2^2+x_2^2x_3^2+x_1^2x_3^2)=-12<0.$$

($\romannumeral11$) $t_{1233}=0$, and $t_{1222}=-t_{1333}=t_{2223}=-t_{2333}=t_{1123}=t_{1223}=t_{1122}=t_{2233}=t_{1133}=1$. Then for $x=(3,1,-1)^{\top}$,
$${\mathcal{T}}x^4\leq x_2^4+x_3^4+4(x_1x_2^3-x_1x_3^3+x_2^3x_3-x_2x_3^3)+12(x_1^2x_2x_3+x_1x_2^2x_3)+6(x_1^2x_2^2+x_2^2x_3^2+x_1^2x_3^2)=-4<0.$$
The necessity is proved.
\end{proof}

\begin{corollary}\label{cor5}
Let ${\mathcal{T}}=(t_{ijkl})\in{\mathcal{E}}_{4,3}$ and $t_{iiii}=t_{jjjj}=0$, $t_{kkkk}=1$ for $i,j,k\in\{1,2,3\}$, $i\neq j$, $j\neq k$, $i\neq k$. Then $\mathcal{T}$ is positive semi-definite if and only if $t_{iiij}=t_{iiik}=t_{ijjj}=t_{jjjk}=0$ and one of the following conditions is satisfied.
\begin{itemize}
\item [(a)] $t_{ikkk}=t_{jkkk}=0$, and
\begin{itemize}
  \item [(a$_1$)] $t_{1123}=t_{1223}=t_{1233}=t_{iijj}=t_{jjkk}=t_{iikk}=1$, or
  \item [(a$_2$)] $t_{1123}=t_{1223}=t_{1233}=0$, and $t_{iijj},t_{jjkk},t_{iikk}\in\{0,1\}$, or
  \item [(a$_3$)] Two of $\{t_{iijk},t_{ijjk},t_{ijkk}\}$ are $-1$, and the other one is $1$, $t_{iijj}=t_{jjkk}=t_{iikk}=1$, or
  \item [(a$_4$)] Two of $\{t_{iijk},t_{ijjk},t_{ijkk}\}$ are $0$, the other one is $1$ or $-1$, and $t_{iijj}=t_{iikk}=1$ and $t_{jjkk}\in\{0,1\}$ with $|t_{iijk}|=1$.
\end{itemize}
\item [(b)] $t_{jkkk}=\pm1$, $t_{ikkk}=0$, and
\begin{itemize}
\item [(b$_1$)] $t_{1123}=t_{1223}=t_{1233}=0$, and $t_{jjkk}=1$, $t_{iikk},t_{iijj}\in\{0,1\}$, or
\item [(b$_2$)] $t_{iijk}=\pm1$, $t_{ijkk}=t_{ijjk}=0$ and $t_{1122}=t_{2233}=t_{1133}=1$.
\end{itemize}
\item [(c)] $t_{ikkk},t_{jkkk}\in\{-1,1\}$, $t_{ikkk}t_{jkkk}t_{ijkk}=t_{iikk}=t_{jjkk}=1$, $t_{iijk}=t_{ijjk}=0$, and $t_{iijj}\in\{0,1\}$.
\end{itemize}
\end{corollary}
\begin{proof} ``\textbf{if (Sufficiency)}." Sufficiency can be obtained from Theorem \ref{th3.1}, Corollary \ref{cor1}, Theorem \ref{th3.9}, Theorem \ref{th3.10}, Corollary \ref{cor4} and their proof procedures.

``\textbf{only if (Necessity)}." At the same time, according Theorem \ref{th3.1}, Corollary \ref{cor1}, Theorem \ref{th3.9}, Theorem \ref{th3.10}, Corollary \ref{cor4}, necessity can be obtained only by proving there are some points $x=(x_1,x_2,x_3)^{\top}\in{\mathbb{R}}^3$ such that ${\mathcal{T}}x^4<0$ when $t_{jkkk}=\pm 1$, $t_{ikkk}=t_{iijk}=t_{ijkk}=0$, $t_{ijjk}=\pm1$ and $t_{1122}=t_{2233}=t_{1133}=1$.

Suppose $t_{1111}=t_{2222}=t_{1112}=t_{1113}=t_{2223}=t_{1222}=t_{1333}=t_{1123}=t_{1233}=0$, and $t_{1122}=t_{2233}=t_{1133}=1$. Then for
$x=(1,2,-1)^{\top}$,
$${\mathcal{T}}x^4\leq x_3^4+4x_2x_3^3+12x_1x_2^2x_3+6(x_1^2x_2^2+x_2^2x_3^2+x_1^2x_3^2)=-1<0.$$
Then when $t_{2333}=t_{1223}=1$. $-t_{2333}=t_{1223}=1$, $t_{2333}=t_{1223}=-1$ and $t_{2333}=-t_{1223}=1$, the proof are similar.
The necessity is proved.
\end{proof}

\section{Conclusions} For a $4$th order $3$-dimensional symmetric tensor ${\mathcal{T}}=(t_{ijkl})$ with entries $t_{ijkl}\in\{-1,0,1\}$, 
we mainly give the necessary and sufficient conditions of its  positive (semi-)definiteness by separating  its entries into groups. And even more specifically,   we put such a class of tensors  in categories based on the values of $t_{1112}t_{1222}$, $t_{2223}t_{2333}$ and $t_{1333}t_{1113}$, and then dicuss recpectively their positive (semi-)definiteness conditions by using the different argument techniques with the help of themselves distinct structure.  

\section*{Competing interest}
The author declares that he has no known competing financial interests or personal relationships that could have appeared to influence the work reported in this paper.
\section*{Availability of data and materials}
This manuscript has no associated data or the data will not be deposited. [Author’s comment: This is a theoretical study and there are no external data associated with the manuscript].
\section*{Funding}
This  work was supported by
the National Natural Science Foundation of P.R. China (Grant No.12171064), by The team project of innovation leading talent in chongqing (No.CQYC20210309536) and by the Foundation of Chongqing Normal university (20XLB009).



\end{document}